\documentclass[fleqn,11pt]{article}
\usepackage{amssymb,amsmath,amsfonts}
\usepackage[margin=1in]{geometry}
\usepackage{dsfont}
\usepackage{color}
\usepackage{graphicx}
\usepackage{latexsym}
\usepackage{amsmath}
\usepackage{amssymb}
\usepackage{graphics}
\usepackage[dvips]{epsfig}
\usepackage{mathrsfs}
\usepackage{leqno}
\usepackage{amsthm}

\numberwithin{equation}{section}

\newtheorem{thm}{Theorem}[section]

\newtheorem{lem}[thm]{Lemma}
\newtheorem{prop}[thm]{Proposition}
\newtheorem{rmk}[thm]{Remark}

\numberwithin{equation}{section}

\newcommand{\pl}{\partial}

\newcommand{\curl}{\mathrm{curl}\,}

\newcommand{\be}{\begin{equation}}
	\newcommand{\bs}{\begin{split}}
		\newcommand{\es}{\end{split}}
	\newcommand{\ee}{\end{equation}}
\newcommand{\bee}{\begin{equation*}}
	\newcommand{\eee}{\end{equation*}}

\newcommand{\ef}{\eqref}

\begin{document}

\begin{center} \Large{ Uniform Regularity for  Incompressible MHD Equations in a Bounded Domain with Curved Boundary in 3D}
\end{center}
\centerline{Yingzhi Du \& Tao Luo}
\begin{abstract} For the initial boundary problem of the incompressible MHD equations in a bounded domain with general curved boundary in 3D with the general Navier-slip  boundary conditions for the velocity field and the perfect conducting condition for the magnetic field,   we establish the uniform regularity of conormal Sobolev norms and Lipschitz norms to addressing the anistropic regularity of tangential and normal directions , which enable us to prove the vanishing dissipation limit as the viscosity and the magnetic diffusion coefficients tend to zero. We overcome the difficulties caused by the  intricate interaction of  boundary curvature,  velocity field and magnetic fields and resolve the issue caused by the problem that the viscosity and the magnetic diffusion coefficients are not required to equal. 

\end{abstract}

		\section{Introduction}
		
		We consider the incompressible MHD system
		\begin{equation}
			\label{eq1-1}
			\left\{
			\begin{array}{lr}
				\pl_t u + u \cdot \triangledown u + \triangledown p = (\curl B) \times B + \mu(\epsilon) \triangle u,\\
				\pl_t B - \curl(u \times B)= - \nu( \epsilon ) \curl(\curl B),\\
				div u=0, \qquad div B = 0,   
			\end{array}
			\right.
		\end{equation}
		in a bounded  domain $\Omega$ of $\mathbb{R}^3$ with general curved smooth boundary $\partial \Omega$. In \eqref{eq1-1}, $u$ is the velocity, $p$ is the pressure, and $B$ is the magnetic field, $\mu(\epsilon) > 0$ is the inverse of the Reynolds number, $\nu(\epsilon) > 0$ is the inverse of the magnetic Reynolds number ($\epsilon>0$ is a small parameter).  We consider the Navier-slip boundary condition \cite{N1827} for the fluid velocity $u$, which Navier introduces to allow the fluid to slip on the boundary and thus is usually applied to deal with rough boundaries ( \cite{B2008}, \cite{GMAS}) (see also \cite{M2003} for a derivation via a hydrodynamic limit from the Maxwell boundary
condition of the Boltzmann equation). For magnetic field $B$, we consider  the perfectly conducting boundary condition \cite{G1991}. Therefore, the boundary conditions for system \eqref{eq1-1} we consider in this paper read
		\begin{equation}
			\label{eq1-2}
			\left\{
			\begin{array}{lr}
				u \cdot n =0 \quad , \quad (Su \cdot n)_{\tau} = - \alpha u_{\tau} \quad on \quad \pl \Omega,\\
				B \cdot n =0 \quad , \quad (\curl B) \times n = 0 \quad on \quad \pl \Omega.
			\end{array}
			\right.
		\end{equation}
	Here  $\alpha$ is a coefficient which measures the tendency of the fluid to slip on the boundary (we assume that $\alpha$ is a constant throughout this paper),  $n$ is the outward unit normal vector to $\partial\Omega$ and $S$ is the strain tensor defined as $Su=\frac{1}{2}(\triangledown u + \triangledown u^{t})$. For a vector $v$ on the boundary, $v_{\tau}$ means its tangential part thus is $v-(v\cdot n)n$.
		
We set the initial condition for system \eqref{eq1-1} as 
		\begin{equation}
			\label{eq1-3}
			(u,  B)(x,0)=(u_0, B_0)(x) , \qquad x\in \Omega.
		\end{equation}
		
		In this paper, we consider the fluid viscosity $\mu(\epsilon)$  and the  magnetic diffusion coefficient $\nu(\epsilon)$ in the form of 
		\begin{equation}			\label{munu}
		\begin{aligned}
& \mu(\epsilon) = \epsilon + m(\epsilon), \ \nu(\epsilon)=\epsilon+n(\epsilon), \ m(\epsilon)\ge 0, \ n(\epsilon)\ \ge 0,\\
& m(\epsilon)\to 0, n(\epsilon)\to 0, {\rm~as~} \epsilon \to 0, \\
&|m(\epsilon)-n(\epsilon)| \leq C\epsilon^{\frac{5}{4}}, \ {\rm for}\ 0<\epsilon<1, \end{aligned}
		\end{equation}
		where  $C$ is a constant independent of $\epsilon$. Therefore, we allow $\mu(\epsilon)\ne \nu(\epsilon)$, but their difference is of the order of $\epsilon^{\frac{5}{4}}$.

We aim to investigate certain  uniform regularity with respect to $\epsilon$  of solutions to the initial boundary problem \eqref{eq1-1}, \eqref{eq1-2}  and \eqref{eq1-3} so that we can pass the limit, as $\epsilon \to 0$,  to the inviscid solutions of ideal incompressible MHD system:
		\begin{equation}
			\label{idealMHD}
			\left\{
			\begin{array}{lr}
				\pl_t u + u \cdot \triangledown u + \triangledown p = (\curl B) \times B,\\
				\pl_t B - \curl(u \times B)= 0,\\
				div u=0, \qquad div B = 0, \qquad  x \in \Omega , \quad t>0,
			\end{array}
			\right.
		\end{equation}
		with boundary condition $(u \cdot n )|_{\pl \Omega}=0$ and $(B \cdot n )|_{\pl \Omega}=0.$ 
		
		It is  of fundamental importance and challenge  to investigate the behavior of solutions to nonlinear PDEs of fluids and MHD equations when some physical parameters, such as viscosity and magnetic diffusion coefficients,   tend to zero. This is in particular so in the presence of physical boundaries because solutions usually exhibit singular behavior 
near boundaries, i.e., strong or weak boundary layers. This makes it difficult to justify the limit when the above mentioned physical parameters go to zero. For the incompressible MHD system \eqref{eq1-1}, when the domain $\Omega$ takes the form $\mathbb{T}^2\times [0, 1]$, where $\mathbb{T}^2$ is the two dimensional torus, i.e.,  the domain is flat, the uniform $H^3$-regularity is 
obtained by Xiao, Xin and Wu in  \cite{XXW2009} so that the inviscid and zero magnetic diffusion limit can be justified for the boundary conditions 
\be\label{xxwbc} u\cdot n=\omega\times n=0, \ B\cdot n=\ (\nabla\times B)\times n=0, \ee
where $\omega=\nabla\times u$ is the vorticity.  As pointed out in \cite{XXW2009} (see also \cite{beirao} for Navier-Stokes equations), it is hardly  to extend the uniform $H^3$-estimates to the case when the boundary is not flat even for the boundary conditions \eqref{xxwbc}, which is  simpler than the general conditions \eqref{eq1-2},  some boundary integrals involving the boundary curvature  which may not be uniform bounded appear in the estimates.   
Note that the boundary condition for $u$ in \eqref{eq1-2} can be rewritten as:
		\begin{equation}\label{eq1-13}
			u\cdot n =0, \quad  \omega \times n= 2 ( S(n) u-\alpha u)_\tau 
		\end{equation}
		where $S(n)=\frac{1}{2}(\triangledown n + \triangledown n^{t})$ (the vector field $n$ defined on  $\partial\Omega$ has been extended to the whole domain $\Omega$). When the boundary is flat, $S(n)=0$, \eqref{eq1-13} reduces to 
		\begin{equation}\label{eq1-13a}
			u\cdot n =0, \quad  \omega \times n= -2 \alpha u_\tau. 
		\end{equation}
Therefore, even in the case that the boundary is flat, the boundary condition for the velocity $u$ in \ef{xxwbc} corresponds to $\alpha=0$ in condition \eqref{eq1-2}. It is a natural question to see if one may extend the vanishing dissipation results obtained in \cite{XXW2009} for the boundary conditions \eqref{xxwbc} for a flat boundary in the form of $\mathbb{T}^2\times [0, 1]$ to the case when the boundary is a general curved one under the general boundary conditions \eqref{eq1-2}. However,  one may not  in general  expect the uniform $H^3$-regularity  since the weak boundary layer may develop for this general case so that the $L^2$-norm of $\pl_n^2 u$ and $\pl_n^2 B$ may not be uniformly bounded. On should also note, the  uniform regularity and the asymptotic convergence  of the incompressible MHD equation with a different boundary condition for the magnetic field was proved in [6] as the dissipations tend to zero, also for a flat domain. 

For  the initial boundary value problem \eqref{eq1-1} with the initial condition \eqref{eq1-3} and general boundary condition \eqref{eq1-2} in a bounded domain in $\mathbb{R}^3$ with a curved boundary, the tangential and normal regularity are usually distinct, and the solutions are smoother in the tangential direction than in the normal direction  when $\epsilon$ is small. Therefore, instead of exploring the uniform $H^3$-regularity for solutions, it is natural to consider the problem in functional spaces distinguishing the regularity in the tangential and normal directions. 
For the incompressible Navier-Stokes equations with the general Navier-slip boundary condition, this was successfully achieved by Masmoudi and Rousset in  \cite{M2012} 
by  establishing  the uniform bounds in a conormal Sobolev space and uniform $L^{\infty}$-bound of one normal derivative.  This approach was also adopted by Masmoudi and Rousset in the study the vanishing viscosity problem  of incompressible Navier-Stokes equations with a free surface (\cite{M2017}).  However, when considering the MHD system \eqref{eq1-1} with boundary condition \eqref{eq1-2}, incorporating a magnetic field $B$ and addressing the issue of  inconsistency between the fluid viscosity  $\mu(\epsilon)$ and magnetic diffusion $\nu(\epsilon)$  introduce analytic difficulty and great  complexity. We will address this issue in detail after the statement of the  main theorems. 

We review some results closely related to the present work besides those mentioned above (\cite{M2012, XXW2009}).  For  incompressible Navier-Stokes equations with the boundary condition $n\cdot u=n\times (\nabla\times u)=0$, the uniform $H^3$ and $W^{k, p}$ regularity thus the vanishing viscosity limits were verified in \cite{XX2007}  and 
\cite{beirao, beirao1}.  For the problem with general Navier-slip boundary conditions, a matched asymptotic expansion with the boundary layer correction was analyzed by Iftimie and Sueur in \cite{I2011}, by which the convergence  to the solution of the  Euler solutions in $L^\infty (0, T; L^2)$ was justified. The results on the convergence rates from viscous solutions to the inviscid ones can be found in \cite{G2012} and \cite{XX2013}.   For the initial boundary value problem of Navier-Stokes equations with the Dirichlet boundary condition $u=0$ on the boundary, the situation is quite different involving Prandtl strong boundary layers, the vanishing viscosity limit was only justified for analytic data ( \cite{SC1998, SC1998-2} and \cite{zhangzf} ) or for the case of  the initial vorticity located away from the boundary (\cite{M2014}). 
				
For compressible isentropic Navier-Stokes equations with the Navier-slip boundary condition, a boundary
layer profile  in half plane was constructed in Wang- Williams (\cite{WW2013}), uniform regularity estimates were obtained by Paddick (\cite{paddick}) in the 3-D halfspace, and further by Wang-Xin-Yong (\cite{WX2015}) for the general 3D domain in which it was shown that the density boundary layers are weaker
than that of  the velocity. For the 
full compressible Navier-Stokes equations with Navier-slip boundary conditions in 3D,  the uniform
regularity of the solutions was obtained by Wang in \cite{wangyong} which also justifies the vanishing dissipation limit. For compressible Navier-Stokes equations coupled with the electric field in half space, Ju-Luo-Xu (\cite{JLX}) with the Navier-slip boundary condition for the velocity field and Dirichlet boundary condition for the electric potential, the mixed vanishing viscosity and Debby length limit were also verified.

We turn to the related results for  the incompressible MHD equations. For the problem in a 2D half plane with the non-slip Dirichlet boundry condition on the velocity field, Liu, Xie \& Yang established the well-posedness of the boundary layer equations in \cite{LXY2019}, and the vanishing viscosity and magnetic diffusion limit in \cite{LXY2019-2} for the non-degenerate tangential magnetic field. It is discovered in \cite{LXY2019} and \cite{LXY2019-2} that the non-degenerate tangential magnetic field induces a cancellation mechanism for the singular terms. Therefore, the mechanisms for the vanishing viscosity and magnetic diffusion limit for the problem studied in \cite{LXY2019} and \cite{LXY2019-2} and the one studied in this paper are different. Other 
related results on the vanishing dissipation limit for incompressible MHD equations under the non-slip Dirichlet boundary condition for the velocity field can be found in \cite{LW2020}, \cite{LXY2021}, \cite{LYZ} and \cite{GLW}. We also mention here the vanishing dissipation limit was proved  for the steady MHD equations in \cite{LYZ}  by the justification of  Prandtl expansions. 

				We consider $\Omega$ a bounded domain of $\mathbb{R}^3$ and presume the existence of a covering for $\Omega$ in the following manner
		\begin{equation}\label{eq1-4}
			\Omega \in \Omega_0 \cup^n_{i=1} \Omega_i,
		\end{equation}
		where $\bar{\Omega_0} \subset \Omega$ and. In each $\Omega_i$, we have a smooth function $\psi_i$ such that $\Omega \cap \Omega_i =$\\ $\{(x=(x_1,x_2,x_3),x_3>\psi_i(x_1,x_2))\} \cap \Omega_i$ and $\pl \Omega \cap \Omega_i = \{x_3=\psi_i(x_1,x_2)\} \cap \Omega_i$.
		We use $(Z_k)_{1 \leq k \leq 3}$ to denote the conormal derivatives (one can refer to \cite{M2012}), which is a finite set of generators of vector fields that are tangent to $\pl \Omega$, and we define
		$$H^m_{co}(\Omega)=\{f \in L^2(\Omega), \quad Z^If \in L^2(\Omega), \quad I \leq m \},$$
		where $I=(k_1,k_2,...,k_m)$, $Z^I=Z_{k_1}Z_{k_2} \cdots Z_{k_m}$.
		
		Then we set
		\begin{equation}\label{eq1-5}
			\left\| f \right\|^2_{m}= \sum_{|I| \leq m} \left\| Z^I f \right\|^2_{L^2},
		\end{equation}
		for a vector field $u$, and let
		\begin{equation}\label{eq1-6}
			\left\| u \right\|^2_{m}= \sum^3_{i=1} \sum_{|I| \leq m} \left\| Z^I u_i \right\|^2_{L^2},
		\end{equation}
		and
		\begin{equation}\label{eq1-7}
			\left\|u \right\|_{m,\infty} = \sum_{|I| \leq m} \left\|Z^I u \right\|_{L^{\infty}}.
		\end{equation}
		
	Using the partition of unity,  it is always possible to presume that each vector field finds its support in one of the subdomains $\Omega_i$. Furthermore, within $\Omega_0$, the norm $\Vert \cdot \Vert_m$ provides control over the standard $H^m$ norm. However, if the intersection of $\Omega_i$ and the boundary $\pl \Omega$ is not empty, the  control over the normal derivatives is missing.  The symbol $| \cdot |_{H^m}$ will be employed to denote the standard Sobolev norm of functions defined on the boundary $\pl \Omega$. It is important to recognize that this norm exclusively considers tangential derivatives.
		
		Let $C_k$ represent a positive constant that is independent of $\epsilon \in (0,1]$, relying solely on the $\mathcal{C}^k$-norm of the functions $\psi_i$. The boundary is given locally by $x_3=\psi(x_1,x_2)$ (here $i$ is omitted for simplicity), we can use the coordinates
		\begin{equation}\label{1-8}
			\Psi: (y,z) \mapsto (y, \psi(y)+z).
		\end{equation}
		Therefore, the vector fields $(\pl_{y^1}, \pl_{y^2},\pl_z)$ give a local basis. Note that, on the boundary, $\pl_{y^1}$ and $\pl_{y^2}$ are tangent to $\pl \Omega$ while $\pl_z$ is not a normal vector field. Therefore, we can consider vector fields that are compactly supported within $\Omega_i$ when defining the $\left\| \cdot \right\|_{m}$ norms:
		\begin{equation}\label{eq1-9}
			Z_i=\pl_{y^i}=\pl_i+\pl_i\psi\pl_z, \quad i=1,2, \quad Z_3=\varphi(z)(\pl_1\psi\pl_1 + \pl_2 \psi \pl_2 - \pl_z),
		\end{equation}
		we can choose $\varphi(z)=\frac{z}{z+1}$.
		
		In this paper, we denote by $\pl_j$, $j=1,2,3$ or $\triangledown$ the derivatives in the physical space. For the canonical basis of $\mathbb{R}^3$, we use $u_j$, that is, $u=u_1\pl_1+u_2\pl_2+u_3\pl_3$. For the basis $(\pl_{y^1}, \pl_{y^2}, \pl_z)$, we use $u^i$, that is
		\begin{equation}\label{eq1-10}
			u=u^1\pl_{y^1}+u^2\pl_{y^2}+u^3\pl_z.
		\end{equation}
		
		Then, the unit outward normal $n$ in the physical space will be given locally by
		\begin{equation}\label{eq1-11}
			n(x)=n(\Psi(y,z))=\frac{1}{\sqrt{1+|\triangledown \psi(y)|^2}}\left( \begin{array}{c}
				\pl_1 \psi(y)\\
				\pl_2 \psi(y)\\
				-1
			\end{array} \right) = \frac{-N(y)}{\sqrt{1+|\triangledown \psi(y)|^2}},
		\end{equation}
		and orthogonal projection $\Pi$ shall be
		\begin{equation}\label{eq1-12}
			\Pi(x)=\Pi(\Psi(y,z))u=u-[u\cdot n(\Psi(y,z))]n(\Psi(y,z)),
		\end{equation}
		which gives the orthogonal projection onto the tangent space of the boundary. Note that we establish $n$ and $\Pi$ in the entire $\Omega_i$ and independent of $x_3$.
		
Before we state our main theorems, we define
		\begin{equation}
			M_m(u, B)(t)=N_m(u, B)(t)+ M(u, B)(t),
		\end{equation}
		with
		\begin{equation}\label{N}
			N_m(u, B)(t) = \sup_{0 \leq \tau \leq t} \left\{ 1 + \left\| (u,B)(\tau) \right\|^2_{m}+ \left\| (\triangledown u, \triangledown B) (\tau) \right\|^2_{m-1} + \left\| (\triangledown u, \triangledown B)(\tau) \right\|^2_{{1,\infty}} \right\},
		\end{equation}
		and
		
		\begin{equation}\label{M}
			M(u, B)(t)= \sup_{0 \leq \tau \leq t} \left\{ \left\| (\pl_t u, \pl_t B) (\tau) \right\|^2_{4} + \left\| (\triangledown \pl_t u, \triangledown \pl_t B) (\tau)\right\|^2_{3}\right\}.
		\end{equation}

		The first theorem we prove in this paper is the uniform regularity of strong  solutions to the problem \eqref{eq1-1}, \eqref{eq1-2} and \eqref{eq1-3}. 
		
	\begin{thm}\label{thm1a}	
		Let $m$ be an integer satisfying $m \geq 6$, $\Omega$ be a $\mathcal{C}^{m+2}$ domain. Assume that $\mu(\epsilon)$ and $\nu(\epsilon)$ satisfy \eqref{munu} for $\epsilon\in (0, 1)$. Suppose that the initial data $(u_0, B_0)$ satisfy $M_m(u_0, B_0)<\infty$, $\nabla\cdot u_0=\nabla\cdot B_0=0$ in $\Omega$, $u_0\cdot n={B_0\cdot n} =0$ and $\pl_t u|_{t=0}\cdot n=\pl_t B |_{t=0}\cdot n =0$ on $\partial \Omega$. Then for $\epsilon\in (0,1)$, there exists ${T}>0$ independent of $\epsilon$ such that  the initial boundary value problem \eqref{eq1-1},  \eqref{eq1-2} and \eqref{eq1-3} admits a unique solution $(u^\epsilon, B^\epsilon)$ on the time interval  $[0,T]$  satisfying 
		\begin{equation}\label{T1a}
			\begin{split}
				&
				M_m(u^\epsilon, B^\epsilon)(T) + \epsilon \int^{T}_{0} ( \left\| (\triangledown u^\epsilon, \triangledown B^\epsilon)\right\|_{m}^2 + \left\| (\triangledown^2 u^\epsilon, \triangledown^2 B^\epsilon)\right\|_{m-1}^2 \, ) \mathrm{d} t
				\leq C , 
			\end{split}
		\end{equation}
		for some constant $C$ independent of $\epsilon$.
	    \end{thm}
    
For $0<\epsilon<1$, we denote the solution in the time interval $[0, T]$ obtained in Theorem \ref{thm1a} by $(u^\epsilon, B^\epsilon)$.  With the uniform regularity of $(u^\epsilon, B^\epsilon)$ stated in Theorem \ref{thm1a}  in hand, one may pass the limit as $\epsilon\to 0$ to the  solution of the inviscid system \eqref{idealMHD}. 

\begin{thm}\label{thm2}	
		Let $m$ be an integer satisfying $m \geq 6$, $\Omega$ be a $\mathcal{C}^{m+2}$ domain. Assume that $\mu(\epsilon)$ and $\nu(\epsilon)$ satisfy \eqref{munu} for $\epsilon\in (0, 1)$. Suppose that the initial data $(u_0, B_0)$ satisfy $M_m(0)<\infty$, $\nabla\cdot u_0=\nabla\cdot B_0=0$ in $\Omega$,  $\nabla\cdot u_0=\nabla\cdot B_0=0$ in $\Omega$, $u_0\cdot n={B_0\cdot n} =0$ and $\pl_t u|_{t=0}\cdot n=\pl_t B |_{t=0}\cdot n =0$ on $\partial \Omega$.  For $0<\epsilon<1$, let $(u^\epsilon, B^\epsilon)$ be the solution to  the initial boundary value problem \eqref{eq1-1},  \eqref{eq1-2} and \eqref{eq1-3}  on the time interval  $[0,T]$  satisfying \eqref{T1a}. The inviscid system \eqref{idealMHD} with the boundary 
		conditions $u\cdot n=B\cdot n=0$ on $\partial \Omega$ and the initial condition $(u, B)(x, 0)=(u_0, B_0)(x)$ for $x\in \Omega$ admits a unique solution in the time interval $[0, T]$ with $M_m(u, B)(T)<\infty$. Moreover, it holds that
		\begin{equation}\label{T1b}
		\lim_{\epsilon\to 0}\sup_{0<t\le T}(\|(u^\epsilon, B^\epsilon)-(u, B)\|_{L^2(\Omega)}+\|(u^\epsilon, B^\epsilon)-(u, B)\|_{L^\infty(\Omega)})=0. \end{equation}	\end{thm}
		
In the statements of the above two equations, the compatibility conditions  $\pl_t u|_{t=0}\cdot n=\pl_t B |_{t=0}\cdot n =0$ on $\partial \Omega$ are required. The initial data for $\pl_t u|_{t=0}$ and $ \pl_t B |_{t=0}$ are given by $u_0$ and $B_0$ via (1.1). \\
The key to the proof of Theorem \ref{thm1a} is to establish the uniform a priori estimates
 for smooth solutions, which will be given in Section \ref{sec3} (see Theorem \ref{thm3}). We outline the main steps in the proof and address issues caused by the strong coupling of velocity and magnetic fields, in particular, the issues when the fluid viscosity and magnetic diffusion coefficients  are not equal. In the proof of the a priori estimates, the first step is to derive  the conormal energy estimates, see Lemma \ref{lem2}, in which we obtain the estimate
\begin{equation}\label{eqL2a}
			\begin{split}
				&
				 \left\| (u,B) (t) \right\|^2_{m} +  \mu (\epsilon) \int_{0}^{t} \left\| \triangledown u \right\|^2_{m} \, \mathrm{d} \tau +\nu(\epsilon) \int_{0}^{t} \left\| \triangledown B \right\|^2_{m} \mathrm{d} \tau 				\\
				&\quad
				\lesssim  \left\| (u,B) (0) \right\|^2_{m} + \mu(\epsilon)^2 \delta \int^t_0 \left\| \triangledown^2 u \right\|^2_{m-1} \mathrm{d} \tau + \nu(\epsilon)^2 \delta \int^t_0 \left\| \triangledown^2 B \right\|^2_{m-1} \mathrm{d} \tau 
				\\
				&\quad \quad
				+ C_{\delta} C_{m+2} (1 + P(\left\| (u,B,Zu,ZB,\triangledown u,\triangledown B) \right\|^2_{L^{\infty}} ))\int_{0}^{t} P( \left\|(u,B)\right\|^2_{m} + \left\|(\triangledown u, \triangledown B)\right\|^2_{m-1} ) \mathrm{d}\tau
				\\
				&\quad \quad
				+  C_{m+2} \int^t_0 (\left\| \triangledown^2 q_1 \right\|^2_{m-1} + \frac{1}{\mu(\epsilon)} \left\| \triangledown q_2 \right\|^2_{m-1}) \mathrm{d} \tau + \delta \int_{0}^{t} \left\| \triangledown q_2 \right\|^2_{1} \mathrm{d} \tau, 
			\end{split}
		\end{equation}
		for any positive number $\delta>0$, $P(\cdot)$ is a polynomial function and $m \in \mathbb{N}$. $q_1$ and $q_2$ are solutions to the following problems respectively, 
		\begin{equation}\label{q1a}
			\triangle q_1 = -\triangledown u \cdot \triangledown u + \triangledown B \cdot \triangledown B , \quad x \in \Omega, \quad  \pl_n q_1 = (- u \cdot \triangledown u+ B \cdot \triangledown B) \cdot n , \quad x \in \pl \Omega, \end{equation}
			and 
		\begin{equation}\label{q2a}
			\triangle q_2 = 0 , \quad x \in \Omega, \quad  \pl_n q_2 = \mu(\epsilon) \triangle u \cdot n  , \quad x \in \pl \Omega.
\end{equation}



Next, due to the utilization of conormal Sobolev spaces, normal derivative control on the boundary is lacking. Consequently, we establish estimates for $\left\| \triangledown u\right\|^2_{m-1}$ and $\left\| \triangledown B\right\|^2_{m-1}$ ($m \in \mathbb{N}_{+}$) in Lemma \ref{lem4} via the estimates for $\eta=\chi \Pi(\curl u \times n + 2 ( \alpha u - S(n)u))$ with $S(n)=\frac{1}{2}(\triangledown n + (\triangledown n)^t)$ and $\chi (\curl B \times n)$ , $\chi$ is compactly supported in one of the $\Omega_j$ and with value one in a neighborhood of the boundary, by using the boundary conditions for $u$ and $B$ as follows, 

\begin{equation}\label{eqL4a}
	\begin{split}
		&
		\frac{1}{2} \left\| \eta (t) \right\|_{m-1}^2 + \frac{1}{2} \left\| \chi (\curl B \times n) (t) \right\|_{m-1}^2 + \frac{1}{2} \mu (\epsilon) \int^{t}_{0}  \left\| \triangledown \eta \right\|_{m-1}^2 \,  \mathrm{d} \tau + \frac{1}{2} \nu (\epsilon) \int^{t}_{0} ( \left\| \triangledown  \chi (\curl B \times n) )\right\|_{m-1}^2 \,  \mathrm{d} \tau
		\\
		&\quad
		\leq C_{m+2}( \frac{1}{2} \left\| \eta (0) \right\|_{m-1}^2 + \frac{1}{2} \left\| \chi (\curl B \times n) (0) \right\|_{m-1}^2 + \mu(\epsilon)^2 \delta \int^t_0 \left\| \triangledown^2 u \right\|_{m-1}^2 \,  \mathrm{d} \tau + \nu(\epsilon)^2 \delta \int^t_0 \left\| \triangledown^2 B \right\|_{m-1}^2 \,  \mathrm{d} \tau 
		\\
		&\quad \quad
		+ C_{\delta} (1 + \left\| (u,B,\triangledown u,\triangledown B, Z \triangledown u, Z \triangledown B) \right\|^2_{L^{\infty}} )\int^{t}_{0} ( \left\|(u,B)\right\|^2_{m} + \left\|(\triangledown u, \triangledown B)\right\|_{m-1}^2 + \left\| \triangledown q \right\|_{m-1}^2 \, ) \mathrm{d}\tau.
	\end{split}
\end{equation}



Compared with the problem of incompressible Navier-Stokes equations without coupling with magnetic fileds studied in \cite{M2012},  we have extra terms $\mu(\epsilon)^2 \delta \int^t_0 \left\| \triangledown^2 u \right\|^2_{m-1} \mathrm{d} \tau + \nu(\epsilon)^2 \delta \int^t_0 \left\| \triangledown^2 B \right\|^2_{m-1} \mathrm{d} \tau $ that need to be controlled on the right-hand side of \eqref{eqL2a} and \eqref{eqL4a}.  
The main idea for this is to apply the Hodge type estimates for vector fields (Proposition \ref{prop4}). The details for this can be found in Section 3.3.

  Motivated by the argument in \cite{M2012},  the  total pressure can be decomposed as two parts as in (1.26) and (1.27).  Using the elliptic estimates for   Neumann problems, we can have
  \begin{equation}\label{eqL5a}
  	\begin{split}
  		&
  		\int^t_0 \left\| \triangledown q_1 \right\|^2_{m-1} + \left\| \triangledown^2 q_1 \right\|^2_{m-1} \,  \mathrm{d} \tau \leq C_{m+2} (1+ P(\left\| (u, B , \triangledown u, \triangledown B) \right\|^2_{L^{\infty}}))
  		\\
  		&\qquad \qquad\qquad\qquad\qquad\qquad\qquad
  		\int^t_0 (\left\|(u, B)\right\|^2_m + \left\| (\triangledown u, \triangledown B)\right\|^2_{m-1} \, ) \mathrm{d} \tau,
  		\\
  		&
  		\int^t_0 \left\| \triangledown q_2 \right\|^2_{m-1} \,  \mathrm{d} \tau \leq C_{m+2} \mu(\epsilon) \int^t_0 ( \left\| u \right\|^2_m + \left\| \triangledown u \right\|^2_{m-1} \, ) \mathrm{d} \tau,
  	\end{split}
  \end{equation}
  where $m \geq 2$ and $P(\cdot)$ is a polynomial function.
  
  

  The most difficult part in the proof of Theorem 3.1 is to obtain the necessary $L^{\infty}$ estimates of the normal derivatives 
  which constitute the most challenging part of our proof for Theorem \ref{thm3}. Specifically, a primary challenge arises due to the presence of time-derivative-involved terms, stemming from the non-coincidence of $\mu(\epsilon)$ and $\nu(\epsilon)$. Precisely, we first get Lemma \ref{lem6} from applying Proposition \ref{prop2} directly such that for $m_0>1$, we have
  	\begin{equation}\label{eqL6a}
  	\left\| (u, B) \right\|^2_{2,\infty} \leq C_m( \left\| (\triangledown u, \triangledown B)\right\|^2_{m-1} + \left\| ( u, B )\right\|^2_{m})\leq C_m N_m(t), \, m \geq m_0+3.
  \end{equation}
  Then, we derive the estimate for $\left\|(\triangledown u, \triangledown B) \right\|^2_{1, \infty} $ in Lemma \ref{lem7} such that for $m \geq 6$ and every $\delta>0$
  \begin{equation}\label{eqL7a}
  	\left\| (\triangledown u, \triangledown B) \right\|^2_{1,\infty} \leq C_{m+2}(P(M_m(0)) + P(M_m(t)) \int^t_0 (C_{\delta} P(M_m(\tau))+ \delta \epsilon \left\| \triangledown^2 \pl_t B \right\|^2_{2} ) \, \mathrm{d} \tau),
  \end{equation}
  where the term $\epsilon\int_{0}^{t} \left\| \triangledown^2 \pl_t B \right\|^2_2 \mathrm{d} \tau$ shows up because of the inconsistency of $\mu(\epsilon)$ and $\nu(\epsilon)$. In order to control this term, we derive the vorticity equations solved by $\chi \Pi((\omega + 2(\alpha u -S(n)u) +\curl B) \times n)$ and $\chi \Pi((\omega + 2(\alpha u -S(n)u) - \curl B) \times n)$ which contain the term $(\nu(\epsilon)-\mu(\epsilon))\triangle B$. Therefore, according to Lemma \ref{lem8}, we need to estimate $\int_{0}^{t}\left\|(\nu(\epsilon)-\mu(\epsilon))\triangle \curl B\right\|^2_{1,\infty}\mathrm{d}\tau$. Since we have $\eqref{eq1-1}_2$ such that $\pl_t B - \curl (u \times B)=\nu(\epsilon)\triangle B$, we can rewrite this term as $\int_{0}^{t}\left\|\frac{(\nu(\epsilon)-\mu(\epsilon))}{\nu(\epsilon)} (\curl \pl_t B- \curl \curl (u \times B))\right\|^2_{1,\infty}\mathrm{d}\tau$ which forces us to deal with the estimate for $|\frac{(\nu(\epsilon)-\mu(\epsilon))}{\nu(\epsilon)}|^4 \int_{0}^{t} \left\| \triangledown^2 \pl_t B \right\|^2_{2} \mathrm{d}\tau$, which reduces to estimate $\epsilon\left\| \triangledown^2 \pl_t B \right\|^2_{2}$ due to the assumption $|m(\epsilon)-n(\epsilon)| \leq C\epsilon^{\frac{5}{4}}$. The estimate 
  of this new term due to the non-consistency of the coefficients of velocity viscosity and magnetic diffusion needs substantial work. For this, we take the time derivatives of \eqref{eq1-1} and \eqref{eq1-2}. It should be  remarked that in the step of obtaining the normal derivative estimates for the new system of 
	 $\pl_t(u, B)$, instead  of estimating   $\left\|\triangledown \pl_t u \right\|^2_{L^{\infty}}$ as usual which may not be controllable, we employ the local coordinate system ( \eqref{eqL7-3} in Lemma \ref{lem7}) to prevent the simultaneous appearance of normal and time derivatives in the $L^{\infty}$ estimates, which is not needed for the problem of the incompressible Navier-Stokes equations without coupling with magnetic fields, to obtain the estimate 
    \begin{equation}\label{eqL14a}
  	\left\| (\pl_t u, \pl_t B) \right\|^2_{1,\infty} \leq C_m( \left\| (\triangledown \pl_t u, \triangledown \pl_t B)\right\|^2_{m-1} + \left\| ( \pl_t u, \pl_t B )\right\|^2_{m}), \, m \geq m_0+2,
  \end{equation}
  for $m_0>1$.  This enable us to obtain the estimate, 
  \begin{equation}\label{eqL15a}
  	M(t) + \epsilon \int_0^t (\left\| (\triangledown \pl_t u, \triangledown \pl_t B) \right\|^2_4 + \left\| (\triangledown^2 \pl_t u, \triangledown^2 \pl_t B) \right\|^2_{3} \,) \mathrm{d} \tau \leq  C_{m+2}(P(M_m(0)) + P(M_m(t)) \int_0^t P(M_m(\tau)) \, \mathrm{d} \tau),
  \end{equation}
  where $m \geq 5$, which together with the previous obtained the estimates,  close the a whole a priori estimates.

		\section{Preliminaries}
		
		\subsection{Gagliardo-Nirenberg-Moser type inequality}
		
		\begin{prop}\label{prop1}
			For $u,v \in L^{\infty}(\Omega) \cap H^m_{co}(\Omega)$ with $m \in \mathbb{N}_{+}$ be an integer. It holds that
			\begin{equation}\label{eq2-1}
				\left\|Z^{\beta}u Z^{\gamma}v\right\|^2 \lesssim \left\| u\right\|^2_{L^{\infty}} \left\| v \right\|^2_{m} + \left\| v\right\|^2_{L^{\infty}} \left\| u \right\|^2_{m} \quad , \quad |\beta|+|\gamma|=m.
			\end{equation}
		\end{prop}
		
		One can refer \cite{G1990} and \cite{WX2015} for details and proof.

		\subsection{Anisotropic Sobolev embedding}
		
		Let $m_1 \geq 0$, $m_2 \geq 0$ be integers, $f \in H^{m_1}_{co} (\Omega) \cap H^{m_2}_{co} (\Omega)$ and $\triangledown f \in H^{m_2}_{co} (\Omega)$. We have the anisotropic Sobolev embedding:
		
		\begin{prop}\label{prop2}
			\begin{equation}\label{eq2-2}
				\left\|f\right\|^2_{L^\infty} \leq C(\left\|\triangledown f\right\|_{m_2} + \left\|f\right\|_{m_2}) \cdot \left\| f\right\|_{m_1},
			\end{equation}
			where $m_1+ m_2 \geq 3$.
		\end{prop}
		
		One can refer to \cite{M2017} for details and proof.

		\subsection{Trace estimate}
		
		\begin{prop}\label{prop3}
			Let $f \in H^{m_1}_{co} (\Omega) \cap H^{m_2}_{co} (\Omega)$ and $\triangledown f \in H^{m_2}_{co} (\Omega)$, we have
			
			\begin{equation}\label{eq2-3}
				|f|^2_{H^s(\pl \Omega)} \leq C(\left\| \triangledown f\right\|_{m_2} + \left\| f\right\|_{m_2}) \cdot \left\| f\right\|_{m_1},
			\end{equation}
			where $m_1+ m_2 \geq 2s \geq 0$ and $m_1$, $m_2$ are nonnegative integers.
		\end{prop}
		
		One can refer to \cite{M2017} for details and proof.
		
		\subsection{\color{red} Hodge type estimates}
		
		\begin{prop}\label{prop4}
			Let $u \in H^m(\Omega)$ be a vector-valued function, we have 
		
		\begin{equation}\label{eq2-4}
			\left\| u \right\|^2_{H^m} \leq C_{m+1} ( \left\| div u \right\|^2_{H^{m-1}} + \left\| \curl u \right\|^2_{H^{m-1}} + \left\|u \right\|^2_{H^{m-1}} + |u \cdot n|^2_{H^{m-\frac{1}{2}}(\pl \Omega)} ),
		\end{equation}
		where $m$ is any positive integer.
	    \end{prop}
		
		One can refer to \cite{WX2015} and \cite{CS2017} for details of this Proposition.

		\section{A priori Estimates}\label{sec3}
		
	\begin{thm}\label{thm3}
		Let $m$ be an integer satisfying $m \geq 6$, $\Omega$ be a $\mathcal{C}^{m+2}$ domain in $\mathbb{R}^3$. For very sufficiently smooth solution defined on $[0,T]$ of \eqref{eq1-1}, \eqref{eq1-2}, then the following a priori estimate holds
		\begin{equation}\label{T1}
			\begin{split}
				&
				M_m(t) + \epsilon \int^{t}_{0} ( \left\| (\triangledown u, \triangledown B)\right\|_{m}^2 + \left\| (\triangledown^2 u, \triangledown^2 B)\right\|_{m-1}^2 + \left\| (\triangledown \pl_t u, \triangledown \pl_t B) \right\|^2_4 + \left\| (\triangledown^2 \pl_t u, \triangledown^2 \pl_t B) \right\|^2_{3}\,) \mathrm{d} \tau
				\\
				&\quad
				\leq C_{m+2}( P(M_m(0))+ P(M_m(t)) \cdot \int^t_0 P(M_m(\tau)) \mathrm{d} \tau), \qquad \forall t \in [0,T],
			\end{split}
		\end{equation}
		where $P(\cdot)$ is a polynomial and $M_m(t)$ is defined as
		\begin{equation}
			M_m(t)=N_m(t)+ M(t),
		\end{equation}
		with
		\begin{equation}\label{N}
			N_m(t) = \sup_{0 \leq \tau \leq t} \left\{ 1 + \left\| (u,B)(\tau) \right\|^2_{m}+ \left\| (\triangledown u, \triangledown B) (\tau) \right\|^2_{m-1} + \left\| (\triangledown u, \triangledown B)(\tau) \right\|^2_{{1,\infty}} \right\},
		\end{equation}
		and
		
		\begin{equation}\label{M}
			M(t)= \sup_{0 \leq \tau \leq t} \left\{ \left\| (\pl_t u, \pl_t B) (\tau) \right\|^2_{4} + \left\| (\triangledown \pl_t u, \triangledown \pl_t B) (\tau)\right\|^2_{3}\right\}.
		\end{equation}
		\end{thm}

		\hspace*{\fill}\\
		\hspace*{\fill}

		\subsection{Conormal energy estimates}
		
		In this section, we first consider the a priori $L ^2$ energy estimates.
		
		\begin{lem}\label{lem1}
			For a smooth solution to \eqref{eq1-1} and \eqref{eq1-2}, it holds that for all $\epsilon \in (0,1]$
			\begin{equation}\label{eqL1}
				\begin{split}
					&
					\frac{1}{2} \left\| u (t) \right\|^2 + \frac{1}{2} \left\| B (t) \right\|^2 +  \frac{\mu(\epsilon)}{2C_2} \int^{t}_{0} \left\| \triangledown u\right\|^2 \,  \mathrm{d} \tau + \frac{\nu(\epsilon)}{2C_2} \int^{t}_{0}  \left\| \triangledown B\right\|^2 \,  \mathrm{d} \tau
					\\
					&\quad
					\leq \frac{1}{2} \left\| u (0) \right\|^2 + \frac{1}{2} \left\| B (0) \right\|^2 + C_2 \int^{t}_{0} ( \left\| u\right\|^2 + \left\| B \right\|^2 \, ) \mathrm{d} \tau.
				\end{split}
			\end{equation}
		\end{lem}

		\begin{proof}
	    Multiplying $\eqref{eq1-1}_1$ by $u$ yields that
		\begin{equation}\label{eqL1-1}
			\begin{split}
				&
				\int^t_0 \int_{\Omega} \pl_t u \cdot u + ( u \cdot \triangledown) u \cdot u + \triangledown p \cdot u \mathrm{d} x \mathrm{d} \tau
				= \int^t_0 \int_{\Omega} - \mu(\epsilon ) (\triangledown \times \omega) \cdot u + (\curl B) \times B \cdot u \, \mathrm{d}x \mathrm{d} \tau,
			\end{split}
		\end{equation}
		where we use the identity $ \triangledown div u - \triangledown \times \omega = \triangle u$ and $divu=0$.

		Integrating by parts leads to
		\begin{equation}\label{eqL1-2}
			\int^t_0 \int_{\Omega} \pl_t u \cdot u \mathrm{d} x \mathrm{d} \tau = \frac{1}{2} \left\| u (t) \right\|^2 - \frac{1}{2} \left\| u (0) \right\|^2,
		\end{equation}
		and
		\begin{equation}\label{eqL1-3}
			\begin{split}
				&
				\int^t_0 \int_{\Omega} ( u \cdot \triangledown) u \cdot u \mathrm{d} x \mathrm{d} \tau = \int^t_0 \int_{\Omega} \triangledown \cdot ( u \frac{|u|^2}{2}) - divu \frac{|u|^2}{2} \mathrm{d} x \mathrm{d} \tau 
				= \int^t_0 \int_{\pl \Omega} u \cdot n \frac{|u|^2}{2} \mathrm{d} \sigma \mathrm{d} \tau=0, 
			\end{split}
		\end{equation}
		where we use $divu=0$ and the boundary condition $\eqref{eq1-2}_1$ such that $(u \cdot n)|_{\pl \Omega}=0$.
		
		Similarly, we do integration by parts
		\begin{equation}\label{eqL1-4}
			\begin{split}
				&
				\int^t_0 \int_{\Omega} \triangledown p \cdot u \mathrm{d} x \mathrm{d} \tau = \int^t_0 \int_{\Omega} \triangledown \cdot ( p u) - divu \, p \mathrm{d} x \mathrm{d} \tau 
				= \int^t_0 \int_{\pl \Omega} u \cdot n \, p \mathrm{d} \sigma \mathrm{d} \tau=0, 
			\end{split}
		\end{equation}
		
		Next, we consider the first term on the right-hand side by using integration by parts
		\begin{equation}\label{eqL1-5}
			\begin{split}
				&
				- \mu(\epsilon) \int^t_0 \int_{\Omega} (\triangledown \times \omega) \cdot u \mathrm{d} x \mathrm{d} \tau = -\mu(\epsilon) \int^t_0 \left\| \omega \right\|^2 \mathrm{d} \tau - \int^t_0 \int_{\pl \Omega} \mu(\epsilon) (n \times \omega) \cdot u \mathrm{d} \sigma \mathrm{d} \tau
				\\
				&\quad
				\leq - \mu(\epsilon) \int^t_0 \left\| \omega \right\|^2 \mathrm{d} \tau + \int^t_0 \int_{\pl \Omega} \mu(\epsilon) |2 \Pi(\alpha u - S(n) u)| \, |u| \mathrm{d} \sigma \mathrm{d} \tau,
			\end{split}
		\end{equation}
		where we use the boundary condition \eqref{eq1-13} such that $(n \times \omega)|_{\pl \Omega}=(2 \Pi(\alpha u - S(n) u))|_{\pl \Omega}$. Note that we use the vector equalities from \cite{CM1993}
		$$\triangledown \cdot (\omega \times u) = u \cdot (\triangledown \times \omega) - \omega \cdot \curl u=u \cdot (\triangledown \times \omega) - \omega \cdot \omega,$$
		and
		$$\omega \times u \cdot n=n \times \omega \cdot u$$
		in (\ref{eqL1-5}).
		
		Then for the boundary term in \eqref{eqL1-5}, we apply Proposition 2.3 (trace estimate) and use $div u=0$
		\begin{equation}\label{eqL1-6}
			\begin{split}
				&
				- \mu(\epsilon) \int^t_0 \int_{\Omega} (\triangledown \times \omega) \cdot u \mathrm{d} x \mathrm{d} \tau \leq - \mu(\epsilon) \int^t_0 \left\| \omega \right\|^2 \mathrm{d} \tau + \int^t_0 \mu(\epsilon) ( \delta\left\| \triangledown u \right\|^2 + C_2 C_{\delta} \left\| u \right\|^2 )\mathrm{d} \tau
				\\
				&\quad
				\leq - \frac{\mu(\epsilon)}{2 C_2} \int^t_0 \left\| \triangledown u \right\|^2 \mathrm{d} \tau + C_2 \mu(\epsilon) \int^t_0 \left\| u \right\|^2 \mathrm{d} \tau,
			\end{split}
		\end{equation}
	    where we have Proposition 2.4 such that
	    \begin{equation}\label{eqL1-6-1}
	    	\begin{split}
	    		&
	    		- \mu(\epsilon) \left\| \triangledown u \right\|^2 \geq -\mu(\epsilon) \left\| u \right\|^2_{H^1}  \geq - C_2 \mu(\epsilon) (\left\| \omega \right\|^2 + \left\|div u \right\|^2 + \left\| u \right\|^2 + |u \cdot n |_{H^{\frac{1}{2}}(\pl \Omega)}),
	    	\end{split}
	    \end{equation}
		with boundary condition $(u \cdot n)|_{\pl \Omega}=0$ and we choose $\delta=\frac{1}{2C_2}$.
		
		Recall $\eqref{eq1-1}_2$, $div B=0$ and $div u=0$, we have the vector equality from \cite{CM1993}
		$$\curl (u \times B) = -B div u + u div B-(u \cdot \triangledown)B + (B \cdot \triangledown)u=-(u \cdot \triangledown)B + (B \cdot \triangledown)u .$$
		Then by $\triangle B = \triangledown div B - \curl \curl B= -\curl \curl B$
		\begin{equation}\label{eqL1-7}
			\begin{split}
				&
				\pl_t B- \curl (u \times B) = \pl_t B + (u \cdot \triangledown)B - (B \cdot \triangledown)u
				= - \nu(\epsilon) \curl (\curl B) = \nu(\epsilon) \triangle B.
			\end{split}
		\end{equation}
		
		Multiplying \eqref{eqL1-7} with $B$, we obtain
		\begin{equation}\label{eqL1-8}
			\begin{split}
				&
				\pl_t B \cdot B= \nu(\epsilon ) \triangle B \cdot B + (B \cdot \triangledown)u \cdot B - (u \cdot \triangledown)B \cdot B 
				\\
				&\quad
				= \nu(\epsilon) \triangle B \cdot B + \triangledown \cdot (B (u \cdot B) ) - (B \cdot \triangledown)B \cdot u - \triangledown \cdot (u \frac{1}{2} |B|^2).
			\end{split}
		\end{equation}

		Considering the vector equality from \cite{CM1993}, we have
		$$(\curl B) \times B= (B \cdot \triangledown)B - \frac{1}{2} \triangledown |B|^2.$$
		Then applying \eqref{eqL1-8} and boundary condition \eqref{eq1-2} to the last term of \eqref{eqL1-1}, we obtain
		\begin{equation}\label{eqL1-9}
			\begin{split}
				&
				\int^t_0 \int_{\Omega} (\curl B) \times B \cdot u \mathrm{d} x \mathrm{d} \tau = \int^t_0 \int_{\Omega} - \frac{1}{2} \triangledown (|B|^2) \cdot u +  (B \cdot \triangledown)B \cdot u \, \mathrm{d} x \mathrm{d} \tau
				\\
				&\quad
				= \int^t_0 \int_{\Omega} - \triangledown \cdot (\frac{1}{2}  |B|^2 u) + \frac{1}{2} |B|^2 divu +  (B \cdot \triangledown)B \cdot u \, \mathrm{d} x \mathrm{d} \tau
				= \int^t_0 \int_{\Omega} (B \cdot \triangledown)B \cdot u \, \mathrm{d} x \mathrm{d} \tau
				\\
				&\quad
				= \int^t_0 \int_{\Omega} - \pl_t B \cdot B + \nu(\epsilon ) \triangle B \cdot B + \triangledown \cdot (B (u \cdot B) ) - \triangledown \cdot (u \frac{1}{2} |B|^2) \, \mathrm{d} x \mathrm{d} \tau
				\\
				&\quad
				= \int^t_0 \int_{\Omega} - \pl_t B \cdot B + \nu(\epsilon ) \triangle B \cdot B \, \mathrm{d} x \mathrm{d} \tau,
			\end{split}
		\end{equation}
		where we do integration by parts as before.
		
		Similar to (\ref{eqL1-5}) and (\ref{eqL1-6}), we get by integration by parts and applying the boundary condition \eqref{eq1-2} such that $(n \times \curl B)|_{\pl \Omega}=0$
		\begin{equation}\label{eqL1-10}
			\begin{split}
				&
				\int^t_0 \int_{\Omega} - \pl_t B \cdot B + \nu(\epsilon ) \triangle B \cdot B \, \mathrm{d} x \mathrm{d} \tau = \int^t_0 \int_{\Omega} - \pl_t B \cdot B - \nu(\epsilon ) \triangledown \times \curl B \cdot B \, \mathrm{d} x \mathrm{d} \tau
				\\
				&\quad
				\leq \frac{1}{2} \left\| B(0) \right\|^2 - \frac{1}{2} \left\| B(t) \right\|^2 - \nu(\epsilon) \int^t_0 \left\| \triangledown \times B \right\|^2 \mathrm{d} \tau - \nu(\epsilon) \int^t_0 \int_{\pl \Omega} (n \times (\curl B)) \cdot B \, \mathrm{d} \sigma \mathrm{d} \tau
				\\
				&\quad
				\leq \frac{1}{2} \left\| B(0) \right\|^2 - \frac{1}{2} \left\| B(t) \right\|^2 - \frac{1}{C_2} \nu(\epsilon) \int^t_0 \left\| \triangledown B \right\|^2 \mathrm{d} \tau + C_2 \nu(\epsilon) \int_{0}^{t} \left\| B \right\|^2 \mathrm{d} \tau,
			\end{split}
		\end{equation}
		where the first equality is by $\triangle B = \triangledown div B - \triangledown \times \curl B$ and $div B =0$. For the inequality in the second line, we apply the vector equality $\triangledown \cdot (\curl B \times B) = B \cdot (\triangledown \times \curl B) - \curl B \cdot \curl B$ and Proposition 2.4 such that we have
		$$- \nu(\epsilon) \left\| \triangledown B \right\|^2 \geq -\nu(\epsilon) \left\| B \right\|^2_{H^1}  \geq - C_2 \nu(\epsilon) (\left\| \curl B \right\|^2 + \left\|div B \right\|^2 + \left\| B \right\|^2 + |B \cdot n |_{H^{\frac{1}{2}}(\pl \Omega)}).$$
		Note that we have the boundary condition \eqref{eq1-2} that $(B\cdot n)|_{\pl \Omega}=0$.
		
		Combining \eqref{eqL1-1}, \eqref{eqL1-2}, \eqref{eqL1-3}, \eqref{eqL1-4}, \eqref{eqL1-6} and \eqref{eqL1-9}, we get \eqref{eqL1} that ends the proof.
		
	    \end{proof}
		
		\hspace*{\fill}\\
		\hspace*{\fill}

		Recall the vector equality
		\begin{equation}\label{q-1}
			(\curl B) \times B= (B \cdot \triangledown)B - \frac{1}{2} \triangledown |B|^2.
		\end{equation}
		Then we can rewrite $\eqref{eq1-1}_1$ as
		\begin{equation}\label{q}
			\pl_t u + u \cdot \triangledown u + \triangledown (p+\frac{1}{2} |B|^2)  = (B \cdot \triangledown)B + \mu(\epsilon) \triangle u.
		\end{equation}
		
		We define $ q= (p+\frac{1}{2} |B|^2)$, and consider $q=q_1+q_2$ such that
		\begin{equation}\label{q1}
			\triangle q_1 = - \triangledown \cdot ((u \cdot \triangledown)u) + \triangledown \cdot ((B \cdot \triangledown)B)= -\triangledown u \cdot \triangledown u + \triangledown B \cdot \triangledown B \quad , \quad x \in \Omega,
		\end{equation}
		with $\triangledown u \cdot \triangledown u=\sum_{i,j=1,2,3}(\pl_i u_j \pl_j u_i)$, $\triangledown B \cdot \triangledown B=\sum_{i,j=1,2,3}(\pl_i B_j \pl_j B_i)$ and the boundary condition
		\begin{equation}\label{q1b}
			\pl_n q_1 = (- u \cdot \triangledown u+ B \cdot \triangledown B) \cdot n  \quad , \quad x \in \pl \Omega.
		\end{equation}
		
		And
		\begin{equation}\label{q2}
			\triangle q_2 = 0 \quad , \quad x \in \Omega,
		\end{equation}
		with the boundary condition
		\begin{equation}\label{q2b}
			\pl_n q_2 = \mu(\epsilon) \triangle u \cdot n  \quad , \quad x \in \pl \Omega.
		\end{equation}

		\begin{lem}\label{lem2}
		For a smooth solution to \eqref{eq1-1} and \eqref{eq1-2}, it holds that for all $\epsilon \in (0,1]$ and any positive number $\delta$, 
		\begin{equation}\label{eqL2}
			\begin{split}
				&
				\frac{1}{2} \left\| u (t) \right\|^2_{m} + \frac{1}{2} \left\| B (t) \right\|^2_{m} + \frac{1}{2C_2} \mu (\epsilon) \int_{0}^{t} \left\| \triangledown u \right\|^2_{m} \mathrm{d} \tau + \frac{1}{2C_2} \nu (\epsilon) \int_{0}^{t} \left\| \triangledown B\right\|^2_{m} \mathrm{d} \tau
				\\
				&\quad
				\leq \frac{1}{2} \left\| u (0) \right\|^2_{m} + \frac{1}{2} \left\| B (0) \right\|^2_{m} + \mu(\epsilon)^2 \delta \int^t_0 \left\| \triangledown^2 u \right\|^2_{m-1} \mathrm{d} \tau + \nu(\epsilon)^2 \delta \int^t_0 \left\| \triangledown^2 B \right\|^2_{m-1} \mathrm{d} \tau 
				\\
				&\quad \quad
				+ C_{\delta} C_{m+2} (1 + P(\left\| (u,B,Zu,ZB,\triangledown u,\triangledown B) \right\|^2_{L^{\infty}} ))\int_{0}^{t} P( \left\|(u,B)\right\|^2_{m} + \left\|(\triangledown u, \triangledown B)\right\|^2_{m-1} ) \mathrm{d}\tau
				\\
				&\quad \quad
				+  C_{m+2} \int^t_0 \left\| \triangledown^2 q_1 \right\|^2_{m-1} + \frac{1}{\mu(\epsilon)} \left\| \triangledown q_2 \right\|^2_{m-1} \mathrm{d} \tau + \delta \int_{0}^{t} \left\| \triangledown q_2 \right\|^2_{1} \mathrm{d} \tau.
			\end{split}
		\end{equation}
		where $P(\cdot)$ is a polynomial function and $m \in \mathbb{N}_{+}$. \end{lem}
		
		\begin{proof}
		The case for $m=0$ is already proved in Lemma \ref{lem1}. Assume that \eqref{eqL2} has been proved for $k \leq m-1$, we then need to prove that it holds for $k=m \geq 1$. Applying $Z^{\alpha}$ with $|\alpha|=m$ to $\eqref{q}$, multiplying it with $Z^{\alpha}u$, and integrating with respect to $x$ and $t$, one has
		\begin{equation}\label{eqL2-1}
			\begin{split}
				&
				\int^t_0 \int_{\Omega} Z^{\alpha} \pl_t u \cdot Z^{\alpha} u + Z^{\alpha} ((u \cdot \triangledown)  u ) \cdot Z^{\alpha} u  + Z^{\alpha} \triangledown q \cdot Z^{\alpha} u \, \mathrm{d} x \mathrm{d} \tau
				\\
				&\quad
				= \int^t_0 \int_{\Omega} Z^{\alpha} ((B \cdot \triangledown)B )\cdot Z^{\alpha} u - \mu (\epsilon )Z^{\alpha} \triangledown \times \omega \cdot Z^{\alpha} u  \, \mathrm{d} x \mathrm{d} \tau.
			\end{split}
		\end{equation}
		
		Integrating by parts, we get
		\begin{equation}\label{eqL2-2}
			\int^t_0 \int_{\Omega} Z^{\alpha} \pl_t u \cdot Z^{\alpha} u \, \mathrm{d} x \mathrm{d} \tau= \frac{1}{2} \left\| u(t) \right\|^2_{m} - \frac{1}{2} \left\| u(0) \right\|^2_{m},
		\end{equation}
		and by using Proposition \ref{prop1} and we have $divu =0$
		\begin{equation}\label{eqL2-3}
			\begin{split}
				&
				\int^t_0 \int_{\Omega} Z^{\alpha} ((u \cdot \triangledown ) u ) \cdot Z^{\alpha} u \, \mathrm{d} x \mathrm{d} \tau
				\\
				&\quad
				= \int^t_0 \int_{\Omega} (u \cdot \triangledown) Z^{\alpha} u \cdot Z^{\alpha} u + u \cdot [Z^{\alpha} , \triangledown ] u \cdot Z^{\alpha} u + \sum_{\beta+\gamma=\alpha, |\beta | \geq 1} \mathcal{C}_{\beta,\gamma} Z^{\beta}u \cdot Z^{\gamma}\triangledown u \cdot Z^{\alpha}u \, \mathrm{d} x \mathrm{d} \tau
				\\
				&\quad
				\geq \int^t_0 \int_{\Omega} (u \cdot \triangledown) Z^{\alpha} u \cdot Z^{\alpha} u \, \mathrm{d} x \mathrm{d} \tau - C_{m+1} (1+ \left\| (u, Zu, \triangledown u) \right\|^2_{L^{\infty}}) \int^t_0 (\left\| u \right\|^2_m + \left\| \triangledown u \right\|^2_{m-1} \, ) \mathrm{d} \tau
				\\
				&\quad
				= \int^t_0 \int_{\Omega} \triangledown \cdot( u \frac{1}{2} |Z^{\alpha} u|^2) - div u \frac{1}{2} |Z^{\alpha} u|^2 \, ) \mathrm{d} x \mathrm{d} \tau - C_{m+1} (1+ \left\| (u, Zu, \triangledown u) \right\|^2_{L^{\infty}}) \int^t_0 (\left\| u \right\|^2_m + \left\| \triangledown u \right\|^2_{m-1} \, ) \mathrm{d} \tau
				\\
				&\quad
				=- C_{m+1} (1+ \left\| (u, Zu, \triangledown u) \right\|^2_{L^{\infty}}) \int^t_0 ( \left\| u \right\|^2_m + \left\| \triangledown u \right\|^2_{m-1} \, ) \mathrm{d} \tau,
			\end{split}
		\end{equation}
		where we apply the boundary condition $\eqref{eq1-2}_1$ such that $u \cdot n =0$ on $\pl \Omega$.

		For the term involving $q$, we get
		\begin{equation}\label{eqL2-4}
			\int^t_0 \int_{\Omega} Z^{\alpha} \triangledown q \cdot Z^{\alpha} u \, \mathrm{d} x \mathrm{d} \tau \geq - \int^t_0 \left\| \triangledown^2 q_1 \right\|_{m-1} \left\| u \right\|_m \mathrm{d} \tau - | \int^t_0 \int_{\Omega} Z^{\alpha} \triangledown q_2 \cdot Z^{\alpha} u \, \mathrm{d} x \mathrm{d} \tau|.
		\end{equation}
		
		We consider $q_2$ term with $m=1$ first, we have
		\begin{equation}\label{q2_1}
			- | \int^t_0 \int_{\Omega} Z^{\alpha} \triangledown q_2 \cdot Z^{\alpha} u \, \mathrm{d} x \mathrm{d} \tau| \geq - \delta \int_{0}^{t} \left\| \triangledown q_2 \right\|^2_{1} \, \mathrm{d} \tau - C_{\delta} \int_{0}^{t} \left\| u \right\|^2_{1} \mathrm{d} \tau.
		\end{equation}
		
		Then considering the $q_2$ term with $m \geq 2$, we have the following by integration by parts
		\begin{equation}\label{eqL2-5}
			\begin{split}
				&
				-| \int^t_0 \int_{\Omega} Z^{\alpha} \triangledown q_2 \cdot Z^{\alpha} u \, \mathrm{d} x \mathrm{d} \tau| = -| \int^t_0 \int_{\Omega}  \triangledown Z^{\alpha} q_2 \cdot Z^{\alpha} u + [Z^{\alpha}, \triangledown] q_2 \cdot Z^{\alpha} u \, \mathrm{d} x \mathrm{d} \tau|
				\\
				&\quad
				\geq - | \int^t_0 \int_{\Omega}  \triangledown Z^{\alpha} q_2 \cdot Z^{\alpha} u \, \mathrm{d} x \mathrm{d} \tau|- C_{m+1} \int^t_0 \left\| \triangledown q_2 \right\|_{m-1} \left\| u \right\|_m \, \mathrm{d} \tau
				\\
				&\quad
				\geq - \int^t_0 \left\| \triangledown q_2 \right\|_{m-1} \left\| \triangledown Z^{\alpha} u \right\| \, \mathrm{d} \tau - | \int^t_0 \int_{\pl \Omega} Z^{\alpha} q_2 Z^{\alpha} u \cdot n \, \mathrm{d} \sigma \mathrm{d} \tau| - C_{m+1} \int^t_0 \left\| \triangledown q_2 \right\|_{m-1} \left\| u \right\|_m \, \mathrm{d} \tau.
			\end{split}
		\end{equation}
		
		For the boundary term, we do integration by parts again along the boundary since we have $\alpha_3 =0$. If $\alpha_3 \neq 0$, the boundary term will vanish. Therefore, for $|\beta|=m-1$
		\begin{equation}\label{eqL2-6}
			- | \int^t_0 \int_{\pl \Omega} Z^{\alpha} q_2 Z^{\alpha} u \cdot n \, \mathrm{d} \sigma \mathrm{d} \tau| \geq - C_2 \int^t_0 | Z^{\beta} q_2 |_{L^2(\pl \Omega)} | Z^{\alpha} u \cdot n |_{H^1(\pl \Omega)} \mathrm{d} \tau.
		\end{equation}
		
		Since $u \cdot n=0$ on the boundary, thus we have $Z^{\alpha}(u \cdot n)=0$ on the boundary for all $\alpha$,
		\begin{equation}\label{eqL2-7}
			| Z^{\alpha} u \cdot n|_{H^1(\pl \Omega)} = |Z^{\alpha}(u \cdot n) - \sum_{\beta+ \gamma= \alpha, | \beta| \geq 1} \mathcal{C}_{\beta, \gamma} Z^{\gamma} u \cdot Z^{\beta} n|_{H^1(\pl \Omega)} \leq C_{m+2} |u|_{H^m(\pl \Omega)}.
		\end{equation}
		
		Thus by Proposition \ref{prop3} (trace theorem), one has
		\begin{equation}\label{eqL2-8}
			\begin{split}
				&
				- | \int^t_0 \int_{\pl \Omega} Z^{\alpha} q_2 Z^{\alpha} u \cdot n \, \mathrm{d} \sigma \mathrm{d} \tau| \geq - C_{m+2} \int^t_0 \left\| \triangledown q_2 \right\|_{m-1} ( \left\| \triangledown Z^{\alpha} u \right\| + \left\| u \right\|_m) \mathrm{d} \tau.
			\end{split}
		\end{equation}
		
		Combining \eqref{eqL2-4} and \eqref{eqL2-8},
		\begin{equation}\label{eqL2-9}
			\int^t_0 \int_{\Omega} Z^{\alpha} \triangledown q \cdot Z^{\alpha} u \, \mathrm{d} x \mathrm{d} \tau \geq - C_{m+2} \int^t_0 ( \left\| \triangledown^2 q_1 \right\|^2_{m-1} + \frac{1}{\mu(\epsilon)} \left\| \triangledown q_2 \right\|^2_{m-1} + \left\| u \right\|^2_m \,) \mathrm{d} \tau - \frac{\mu(\epsilon)}{8} \int^t_0 \left\| \triangledown Z^{\alpha} u \right\|^2 \, \mathrm{d} \tau.
		\end{equation}

		Recall
		\begin{equation}
			\pl_t B + (u \cdot \triangledown)B - (B \cdot \triangledown)u = \nu( \epsilon) \triangle B.
		\end{equation}
		
		Applying it with $Z^{\alpha}$ and multiplying $Z^{\alpha} B$, we get
		\begin{equation}\label{eqL2-10}
			\begin{split}
				&
				Z^{\alpha}\pl_t B \cdot Z^{\alpha}B - \nu (\epsilon) Z^{\alpha} \triangle B \cdot Z^{\alpha} B
				\\
				&\quad
				=(B \cdot \triangledown)Z^{\alpha} u \cdot Z^{\alpha} B + B \cdot [ Z^{\alpha},\triangledown ] u \cdot Z^{\alpha} B - ( u \cdot \triangledown)Z^{\alpha} B \cdot Z^{\alpha} B - u \cdot [ Z^{\alpha},\triangledown ] B \cdot Z^{\alpha} B
				\\
				&\quad\quad
				+ \sum_{|\beta_1| \geq 1, \beta_1+\beta_2=\alpha} \mathcal{C}_{\beta_1,\beta_2} (Z^{\beta_1} B \cdot Z^{\beta_2} \triangledown)u \cdot Z^{\alpha} B- \sum_{|\beta_1| \geq 1 , \beta_1+\beta_2=\alpha} \mathcal{C}_{\beta_1,\beta_2} ( Z^{\beta_1} u \cdot Z^{\beta_2} \triangledown)B \cdot Z^{\alpha} B.
			\end{split}
		\end{equation}
		
		For the first term on the right-hand side of \eqref{eqL2-1}, we apply Proposition \ref{prop1} and \eqref{eqL2-10}
		\begin{equation}\label{eqL2-11}
			\begin{split}
				&
				\int^t_0 \int_{\Omega} Z^{\alpha} (B \cdot \triangledown B )\cdot Z^{\alpha} u \, \mathrm{d} x \mathrm{d} \tau
				\\
				&\quad
				= \int^t_0 \int_{\Omega} B \cdot \triangledown Z^{\alpha} B \cdot Z^{\alpha} u + B \cdot [ Z^{\alpha},\triangledown ] B \cdot Z^{\alpha} u + \sum_{|\beta_1| \geq 1 , \beta_1+\beta_2=\alpha} \mathcal{C}_{\beta_1,\beta_2} (Z^{\beta_1}B \cdot Z^{\beta_2} \triangledown)B \cdot Z^{\alpha} u \, \mathrm{d} x \mathrm{d} \tau
				\\
				&\quad
				\leq \int^t_0 \int_{\Omega} B \cdot \triangledown Z^{\alpha} B \cdot Z^{\alpha} u \, \mathrm{d} x \mathrm{d} \tau
				+ C_{m+2} (1+ \left\| (B, ZB, \triangledown B) \right\|^2_{L^{\infty}}) \int^t_0 ( \left\| u \right\|^2_{m}+ \left\| B \right\|^2_{m} + \left\| \triangledown B \right\|^2_{m-1} \, ) \mathrm{d} \tau
				\\
				&\quad
				= \int^t_0 \int_{\Omega} - Z^{\alpha}\pl_t B \cdot Z^{\alpha}B + \nu( \epsilon )Z^{\alpha} \triangle B \cdot Z^{\alpha} B + B \cdot \triangledown (Z^{\alpha} B \cdot Z^{\alpha} u) 
				\\
				&\quad\quad
				+ B \cdot [ Z^{\alpha},\triangledown ] u \cdot Z^{\alpha} B - ( u \cdot \triangledown)Z^{\alpha} B \cdot Z^{\alpha} B - u \cdot [ Z^{\alpha},\triangledown ] B \cdot Z^{\alpha} B
				\\
				&\quad\quad
				+ \sum_{|\beta_1| \geq 1, \beta_1+\beta_2=\alpha} \mathcal{C}_{\beta_1,\beta_2} (Z^{\beta_1} B \cdot Z^{\beta_2} \triangledown)u \cdot Z^{\alpha} B- \sum_{|\beta_1| \geq 1 , \beta_1+\beta_2=\alpha} \mathcal{C}_{\beta_1,\beta_2} ( Z^{\beta_1} u \cdot Z^{\beta_2} \triangledown)B \cdot Z^{\alpha} B\, \mathrm{d} x \mathrm{d} \tau
				\\
				&\quad\quad
				+ C_{m+2} (1+ \left\| (B, ZB, \triangledown B) \right\|^2_{L^{\infty}}) \int^t_0 ( \left\| u \right\|^2_{m}+ \left\| B \right\|^2_{m}+ \left\| \triangledown u \right\|^2_{m-1} + \left\| \triangledown B \right\|^2_{m-1}  \,) \mathrm{d} \tau.
			\end{split}
		\end{equation}
		
		We use Proposition \ref{prop1} and for the rest terms do integration by parts
		\begin{equation}\label{eqL2-12}
			\begin{split}
				&
				\int^t_0 \int_{\Omega} Z^{\alpha} (B \cdot \triangledown B )\cdot Z^{\alpha} u \, \mathrm{d} x \mathrm{d} \tau
				\\
				&\quad
				\leq \int^t_0 \int_{\Omega} - Z^{\alpha}\pl_t B \cdot Z^{\alpha}B + \nu( \epsilon) Z^{\alpha} \triangle B \cdot Z^{\alpha} B + B \cdot \triangledown (Z^{\alpha} B \cdot Z^{\alpha} u) - ( u \cdot \triangledown)Z^{\alpha} B \cdot Z^{\alpha} B \, \mathrm{d} x \mathrm{d} \tau
				\\
				&\quad\quad
				+ C_{m+2} (1+ \left\| (u,B,Zu, ZB,\triangledown u, \triangledown B) \right\|^2_{L^{\infty}}) \int^t_0 ( \left\| u \right\|^2_{m}+ \left\| B \right\|^2_{m}+ \left\| \triangledown u \right\|^2_{m-1} + \left\| \triangledown B \right\|^2_{m-1} \, ) \mathrm{d} \tau
				\\
				&\quad
				= \frac{1}{2} \left\| B(0) \right\|^2_{m} - \frac{1}{2} \left\| B(t) \right\|^2_{m} + \int^t_0 \int_{\Omega} \nu (\epsilon) Z^{\alpha} \triangle B \cdot Z^{\alpha} B \, \mathrm{d} x \mathrm{d} \tau
				\\
				&\quad\quad
				+ C_{m+2} (1+ \left\| (u,B,Zu, ZB,\triangledown u, \triangledown B) \right\|^2_{L^{\infty}}) \int^t_0 ( \left\| u \right\|^2_{m}+ \left\| B \right\|^2_{m}+ \left\| \triangledown u \right\|^2_{m-1} + \left\| \triangledown B \right\|^2_{m-1} \, ) \mathrm{d} \tau,
			\end{split}
		\end{equation}
		where we use $div u =0$, $div B=0$, and the boundary condition \eqref{eq1-2} such that $u \cdot n =0$ and $B \cdot n=0$ on the boundary $\pl \Omega$.
		
		Recall $div B =0$ and the vector equalities from \cite{CM1993}
		\begin{equation}\label{B}
			\triangle B = \triangledown div B - \triangledown \times (\curl B)= - \triangledown \times (\curl B),
		\end{equation}
		and
		$$\triangledown \cdot (Z^{\alpha}\curl B \times Z^{\alpha} B) = \triangledown \times Z^{\alpha} \curl B \cdot Z^{\alpha} B - \triangledown \times Z^{\alpha} B \cdot Z^{\alpha} \curl B.$$
		
		Thus, by integration by parts and Young's inequality we will get
		\begin{equation}\label{eqL2-13}
			\begin{split}
				&
				\int^t_0 \int_{\Omega} \nu (\epsilon) Z^{\alpha} \triangle B \cdot Z^{\alpha} B \, \mathrm{d} x \mathrm{d} \tau = - \int^t_0 \int_{\Omega} \nu (\epsilon) Z^{\alpha} \triangledown \times (\curl B) \cdot Z^{\alpha} B \, \mathrm{d} x \mathrm{d} \tau
				\\
				&\quad
				= - \nu (\epsilon) \int^t_0 \int_{\Omega}  \triangledown \times Z^{\alpha} \curl B \cdot Z^{\alpha} B + [Z^{\alpha}, \triangledown \times] \curl B \cdot Z^{\alpha} B \, \mathrm{d} x \mathrm{d} \tau
				\\
				&\quad
				= - \nu (\epsilon)  \int^t_0 \int_{\Omega} \triangledown \times Z^{\alpha} B \cdot Z^{\alpha} \curl B +  [Z^{\alpha}, \triangledown \times] \curl B \cdot Z^{\alpha} B \, \mathrm{d} x \mathrm{d} \tau - \nu (\epsilon) \int^t_0  \int_{\pl \Omega} n \times Z^{\alpha} \curl B \cdot Z^{\alpha} B \, \mathrm{d} \sigma \mathrm{d} \tau
				\\
				&\quad
				\leq \delta \nu (\epsilon)^2 \int^t_0  \left\| \triangledown^2 B \right\|^2_{m-1} \mathrm{d} \tau + C_{\delta} C_{m+1} \int^t_0 (\left\| B \right\|^2_{m}+ \left\| \triangledown B \right\|^2_{m-1}) \mathrm{d} \tau
				\\
				&\quad \quad
				- \nu (\epsilon) \int^t_0   \int_{\Omega} \triangledown \times Z^{\alpha} B \cdot Z^{\alpha} \curl B \, \mathrm{d} x \mathrm{d} \tau - \nu (\epsilon) \int^t_0  \int_{\pl \Omega} n \times Z^{\alpha} \curl B \cdot Z^{\alpha} B \, \mathrm{d} \sigma \mathrm{d} \tau.
			\end{split}
		\end{equation}

	    Since we have
	    \begin{equation}\label{eqL2-13_2}
	    	\begin{split}
	    		&
	    		- \nu (\epsilon) \int^t_0   \int_{\Omega} \triangledown \times Z^{\alpha} B \cdot Z^{\alpha} \curl B \, \mathrm{d} x \mathrm{d} \tau = - \nu (\epsilon) \int^t_0 \int_{\Omega} \triangledown \times Z^{\alpha} B \cdot \triangledown \times Z^{\alpha} B + \triangledown \times Z^{\alpha} B \cdot [Z^{\alpha}, \triangledown \times] B \, \mathrm{d} x \mathrm{d} \tau 
	    		\\
	    		&\quad
	    		\leq -\frac{3}{4} \nu(\epsilon) \int_{0}^{t} \left\| \triangledown \times Z^{\alpha} B \right\|^2 \mathrm{d} \tau + C_{m+1} \int^t_0 \left\| \triangledown B \right\|^2_{m-1} \mathrm{d} \tau,
	    	\end{split}
	    \end{equation}
	    thus we get
	    \begin{equation}\label{eqL2-13_3}
	    	\begin{split}
	    		&
	    		\int^t_0 \int_{\Omega} \nu (\epsilon) Z^{\alpha} \triangle B \cdot Z^{\alpha} B \, \mathrm{d} x \mathrm{d} \tau
	    		\\
	    		&\quad
	    		\leq -\frac{3}{4} \mu(\epsilon) \int_{0}^{t} \left\| \triangledown \times Z^{\alpha} B \right\|^2 \mathrm{d} \tau + \delta \nu (\epsilon)^2 \int^t_0  \left\| \triangledown^2 B \right\|^2_{m-1} \mathrm{d} \tau + C_{m+1}C_{\delta} \int^t_0 (\left\| B \right\|^2_{m}+ \left\| \triangledown B \right\|^2_{m-1}) \mathrm{d} \tau
	    		\\
	    		&\quad \quad
	    	    - \nu (\epsilon) \int^t_0  \int_{\pl \Omega} n \times Z^{\alpha} \curl B \cdot Z^{\alpha} B \, \mathrm{d} \sigma \mathrm{d} \tau.
	    	\end{split}
	    \end{equation}

		We have the boundary condition $\eqref{eq1-2}_2$, and here we consider $Z^{\alpha_3}_3$ with $\alpha_3=0$ since if $\alpha_3 \neq 0$, $(Z^{\alpha} B)|_{\pl \Omega}=0$, then we have on the boundary that 
		\begin{equation}\label{eqcurlB}
			n \times Z^{\alpha} \curl B =Z^{\alpha}(n \times \curl B)- \sum_{\beta+ \gamma=\alpha,\, |\beta| \geq 1} \mathcal{C}_{\beta,\gamma} Z^{\gamma}\curl B \times Z^{\beta}n = - \sum_{\beta+ \gamma=\alpha,\, |\beta| \geq 1} \mathcal{C}_{\beta,\gamma} Z^{\gamma}\curl B \times Z^{\beta}n.
		\end{equation}
		apply Proposition \ref{prop3} (trace theorem) and Young's inequality
		\begin{equation}\label{eqL2-14}
			\begin{split}
				&
				-\nu (\epsilon) \int^t_0  \int_{\pl \Omega} n \times Z^{\alpha} \curl B \cdot Z^{\alpha} B \, \mathrm{d} \sigma \mathrm{d} \tau \leq \nu (\epsilon) C_{m+2} \int^t_0 \sum_{|\beta| \leq m-1}  | Z^{\beta} \curl B|_{L^2(\pl \Omega)}|Z^{\alpha}B|_{L^2(\pl \Omega)} \, \mathrm{d} \tau
				\\
				&\quad
				\leq \nu (\epsilon) C_{m+2} \int^t_0 \sum_{|\beta| \leq m-1} ( \left\| \triangledown Z^{\beta} \curl B\right\|^{\frac{1}{2}}+ \left\|Z^{\beta} \curl B\right\|^{\frac{1}{2}}) \left\|Z^{\beta} \curl B\right\|^{\frac{1}{2}}
				\\
				&\quad\quad
				( \left\| \triangledown Z^{\alpha}B\right\|^{\frac{1}{2}} + \left\| Z^{\alpha}B \right\|^{\frac{1}{2}} ) \left\| Z^{\alpha}B \right\|^{\frac{1}{2}} \, \mathrm{d} \tau
				\\
				&\quad
				\leq \delta \nu (\epsilon)^2 \int^t_0  \left\| \triangledown^2 B \right\|^2_{m-1} \mathrm{d} \tau + \delta \nu (\epsilon) \int^t_0  \left\| \triangledown B \right\|^2_{m}\mathrm{d} \tau + C_{\delta} C_{m+2} \int^t_0 (\left\| B \right\|^2_{m}+ \left\| \triangledown B \right\|^2_{m-1}) \mathrm{d} \tau,
			\end{split}
		\end{equation}
	    where we consider $\triangledown Z^{\alpha} B = Z^{\alpha} \triangledown B + [\triangledown, Z^{\alpha}] B$ with $\left\|  [\triangledown, Z^{\alpha}] B \right\|^2 \leq C_{m+1} \left\|\triangledown B \right\|^2_{m-1}$ and $\triangledown Z^{\beta} \curl B = Z^{\beta} \triangledown \curl B + [\triangledown , Z^{\beta}] \curl B $ with $\left\| Z^{\beta} \triangledown \curl B \right\|^2 \leq \left\| \triangledown^2 B \right\|^2_{m-1} $ and $\left\| [\triangledown , Z^{\beta}] \curl B \right\|^2 \leq \left\| \triangledown^2 B \right\|^2_{m-2} $.

	    Apply Proposition \ref{prop4}, we have for $|\alpha|=m$ since $div B=0$
	   \begin{equation}\label{eqL2-14_2}
	   	\begin{split}
	   		&
	   		\delta \nu (\epsilon) \int^t_0  \left\| \triangledown B \right\|^2_{m}\mathrm{d} \tau = \delta \nu(\epsilon) \int^{t}_{0} ( \left\| \triangledown B \right\|^2_{m-1} +  \left\|  \triangledown Z^{\alpha} B + [Z^{\alpha} , \triangledown] B \right\|^2 \, ) \mathrm{d} \tau
	   		\\
	   		&\quad
	   		\leq \delta \nu(\epsilon) \int^{t}_{0} ( \left\|Z^{\alpha} B \right\|^2_{H^1} + C_{m+1} \left\|\triangledown B \right\|^2{m-1} \, ) \mathrm{d} \tau
	   		\\
	   		&\quad
	   		\leq \delta \nu(\epsilon) C_2 \int^{t}_{0} ( \left\| \triangledown \times Z^{\alpha} B \right\|^2 + \left\| div Z^{\alpha }B \right\|^2+ \left\| Z^{\alpha} B \right\|^2 + | Z^{\alpha} B \cdot n |^2_{H^{\frac{1}{2}}(\pl \Omega)} + C_{m+1} \left\| \triangledown B \right\|^2_{m-1} \, ) \mathrm{d} \tau
	   		\\
	   		&\quad
	   		\leq \delta \nu(\epsilon) C_2 \int^{t}_{0} ( \left\| \triangledown \times Z^{\alpha} B \right\|^2 + | Z^{\alpha} B \cdot n |^2_{H^{\frac{1}{2}}(\pl \Omega)} + C_{m+1} \left\| \triangledown B \right\|^2_{m-1} + \left\| B \right\|^2_m \, ) \mathrm{d} \tau,
	   	\end{split}
	   \end{equation}
	    for the boundary term since we have the boundary condition \eqref{eq1-2} such that $(B \cdot n)|_{\pl \Omega} =0$ and if $\alpha_3 \neq 0$ it vanished, thus one has
	    \begin{equation}\label{eqL2-14_3}
	    	\begin{split}
	    		&
	    		\int_{0}^{t} | Z^{\alpha} B \cdot n |^2_{H^{\frac{1}{2}}(\pl \Omega)} \mathrm{d} \tau \leq \int_{0}^{t} |\sum_{\beta+ \gamma= \alpha, | \beta| \geq 1} \mathcal{C}_{\beta, \gamma} Z^{\gamma} B \cdot Z^{\beta} n|^2_{H^{\frac{1}{2}}(\pl \Omega)} \mathrm{d} \tau
	    		\\
	    		&\quad
	    		\leq C_{m+2} \int_{0}^{t} |\sum_{|\beta|\leq m-1} Z^{\beta} B |^2_{H^{\frac{1}{2}}(\pl \Omega)} \mathrm{d} \tau \leq  C_{m+2} \int_{0}^{t}( \left\|\triangledown B\right\|^2_{m-1} + \left\|B \right\|^2_{m} \,)\mathrm{d} \tau,
	    	\end{split}
	    \end{equation}
	    where the last inequality we apply Proposition \ref{prop3} (trace theorem).

		Thus, by combining the above, one has
		\begin{equation}\label{eqL2-15}
			\begin{split}
				&
				\int^t_0 \int_{\Omega} Z^{\alpha} (B \cdot \triangledown B )\cdot Z^{\alpha} u \, \mathrm{d} x \mathrm{d} \tau
				\\
				&\quad
				\leq \frac{1}{2} \left\| B(0) \right\|^2_{m} - \frac{1}{2} \left\| B(t) \right\|^2_{m} - \frac{1}{2} \nu (\epsilon) \int^t_0  \left\| \triangledown \times Z^{\alpha} B \right\|^2 \,  \mathrm{d} \tau + \delta \nu(\epsilon)^2 \int^t_0  \left\| \triangledown^2 B \right\|^2_{m-1} \,  \mathrm{d} \tau
				\\
				&\quad\quad
				+ C_{\delta} C_{m+2} (1+ \left\| (u,B,Zu, ZB,\triangledown u, \triangledown B) \right\|^2_{L^{\infty}}) \int^t_0 ( \left\| u \right\|^2_{m}+ \left\| B \right\|^2_{m} + \left\| \triangledown u \right\|^2_{m-1} + \left\| \triangledown B \right\|^2_{m-1} \, ) \mathrm{d} \tau.
			\end{split}
		\end{equation}

		For the last term in \eqref{eqL2-1}, we first do integration by parts and apply Proposition \ref{prop1} similar to \eqref{eqL2-13} by using $div u=0$ and the vector equality from \cite{CM1993} such that
		$$\triangledown \cdot (Z^{\alpha} \omega \times Z^{\alpha} u) = \triangledown \times Z^{\alpha} \omega \cdot Z^{\alpha} u - \triangledown \times Z^{\alpha} u \cdot Z^{\alpha} \curl u,$$
		thus we have
		\begin{equation}\label{eqL2-16}
			\begin{split}
				&
				- \int^t_0 \int_{\Omega} \mu (\epsilon) Z^{\alpha} (\triangledown \times \omega) \cdot Z^{\alpha} u \, \mathrm{d} x \mathrm{d} \tau = - \mu (\epsilon) \int^t_0 \int_{\Omega}  \triangledown \times Z^{\alpha} \omega \cdot Z^{\alpha} u + [Z^{\alpha}, \triangledown \times] \omega \cdot Z^{\alpha} u \, \mathrm{d} x \mathrm{d} \tau
				\\
				&\quad
				= - \mu (\epsilon) \int^t_0 \int_{\Omega} \triangledown \times Z^{\alpha} u \cdot Z^{\alpha} \omega +  [Z^{\alpha}, \triangledown \times] \omega \cdot Z^{\alpha} u \, \mathrm{d} x \mathrm{d} \tau - \mu (\epsilon) \int^t_0 \int_{\pl \Omega} n \times Z^{\alpha} \omega \cdot Z^{\alpha} u \, \mathrm{d} \sigma \mathrm{d} \tau
				\\
				&\quad
				\leq \delta \mu (\epsilon)^2 \int^t_0 \left\| \triangledown^2 u \right\|^2_{m-1} \mathrm{d} \tau + C_{\delta} C_{m+2} \int^t_0 ( \left\| u \right\|^2_{m}+ \left\| \triangledown u \right\|^2_{m-1} \,) \mathrm{d} \tau
				\\
				&\quad \quad
				- \mu( \epsilon) \int_0^t \int_{\Omega} \triangledown \times Z^{\alpha} u \cdot Z^{\alpha} \omega \, \mathrm{d} x \mathrm{d} \tau - \mu (\epsilon) \int^t_0 \int_{\pl \Omega} n \times Z^{\alpha} \omega \cdot Z^{\alpha} u \, \mathrm{d} \sigma \mathrm{d} \tau.
			\end{split}
		\end{equation}
	
	    Since we have
	    \begin{equation}\label{eqL2-16_2}
	    	\begin{split}
	    		&
	    		- \mu( \epsilon) \int_0^t \int_{\Omega} \triangledown \times Z^{\alpha} u \cdot Z^{\alpha} \omega \, \mathrm{d} x \mathrm{d} \tau = - \nu (\epsilon) \int^t_0 \int_{\Omega} \triangledown \times Z^{\alpha} u \cdot \triangledown \times Z^{\alpha} u + \triangledown \times Z^{\alpha} u \cdot [Z^{\alpha}, \triangledown \times] u \, \mathrm{d} x \mathrm{d} \tau 
	    		\\
	    		&\quad
	    		\leq -\frac{3}{4} \mu(\epsilon) \int_{0}^{t} \left\| \triangledown \times Z^{\alpha} u \right\|^2 \mathrm{d} \tau + C_{m+1} \int^t_0 \left\| \triangledown u \right\|^2_{m-1} \mathrm{d} \tau,
    		\end{split}
    	\end{equation}
        thus we get
        \begin{equation}\label{eqL2-16_3}
        	\begin{split}
        		&
        		- \int^t_0 \int_{\Omega} \mu (\epsilon) Z^{\alpha} (\triangledown \times \omega) \cdot Z^{\alpha} u \, \mathrm{d} x \mathrm{d} \tau
        		\\
        		&\quad
        		\leq  \delta \mu (\epsilon)^2 \int^t_0 \left\| \triangledown^2 u \right\|^2_{m-1} \mathrm{d} \tau + C_{\delta} C_{m+2} \int^t_0 (\left\| u \right\|^2_{m}+ \left\| \triangledown u \right\|^2_{m-1} \,) \mathrm{d} \tau -\frac{3}{4} \mu(\epsilon) \int_{0}^{t} \left\| \triangledown \times Z^{\alpha} u \right\|^2 \mathrm{d} \tau
        		\\
        		&\quad \quad
        		- \mu (\epsilon) \int^t_0 \int_{\pl \Omega} n \times Z^{\alpha} \omega \cdot Z^{\alpha} u \, \mathrm{d} \sigma \mathrm{d} \tau.
        	\end{split}
        \end{equation}

		Consider the boundary condition $\eqref{eq1-13}$ such that $ n \times \omega = 2 \Pi(\alpha u - S(n) u)$ on the boundary and since if $\alpha_3 \neq 0$, $Z^{\alpha}u$ will vanish on the boundary, thus we have on the boundary that
		$$n \times Z^{\alpha} \omega= Z^{\alpha}(n \times \omega) - \sum_{\beta+ \gamma= \alpha, | \beta| \geq 1} \mathcal{C}_{\beta, \gamma} Z^{\beta} n \times Z^{\gamma} \omega= Z^{\alpha}(2 \Pi (\alpha u - S(n)u))- \sum_{\beta+ \gamma= \alpha, | \beta| \geq 1} \mathcal{C}_{\beta, \gamma} Z^{\beta} n \times Z^{\gamma} \omega$$
		then we apply it and Proposition \ref{prop3} (trace estimates)
		\begin{equation}\label{eqL2-17}
			\begin{split}
				&
				- \mu (\epsilon) \int^t_0 \int_{\pl \Omega} n \times Z^{\alpha} \omega \cdot Z^{\alpha} u \, \mathrm{d} \sigma \mathrm{d} \tau \leq \mu (\epsilon) \int^t_0 | n \times Z^{\alpha} \omega|_{L^2(\pl \Omega)} \cdot |Z^{\alpha} u|_{L^2(\pl \Omega)} \, \mathrm{d} \tau
				\\
				&\quad
				\leq \mu (\epsilon) C_{m+2} \int^t_0 \sum_{|\beta| \leq m-1}(| Z^{\beta}\omega |_{L^2(\pl \Omega)}+ | Z^{\alpha} u |_{L^2(\pl \Omega)} )|Z^{\alpha}u|_{L^2(\pl \Omega)} \, \mathrm{d} \tau
				\\
				&\quad
				\leq \mu (\epsilon) C_{m+2} \int^t_0 ( \sum_{|\beta| \leq m-1} ( \left\| \triangledown Z^{\beta} \omega \right\|^{\frac{1}{2}} + \left\|Z^{\beta} \omega \right\|^{\frac{1}{2}} ) \left\| Z^{\beta} \omega \right\|^{\frac{1}{2}} + ( \left\| \triangledown Z^{\alpha} u \right\|^{\frac{1}{2}} + \left\| Z^{\alpha} u \right\|^{\frac{1}{2}} ) \left\|Z^{\alpha} u \right\|^{\frac{1}{2}} )
				\\
				&\quad\quad
				( \left\| \triangledown Z^{\alpha} u \right\|^{\frac{1}{2}} + \left\| Z^{\alpha} u \right\|^{\frac{1}{2}} ) \left\| Z^{\alpha} u \right\|^{\frac{1}{2}}
				\\
				&\quad
				\leq \delta \mu (\epsilon)^2 \int^t_0  \left\| \triangledown^2 u \right\|^2_{m-1} \mathrm{d} \tau + \delta \mu( \epsilon) \int^t_0  \left\| \triangledown u \right\|^2_m\mathrm{d} \tau + C_{\delta} C_{m+2} \int^t_0 (\left\| u \right\|^2_{m}+ \left\| \triangledown u \right\|^2_{m-1}\,) \mathrm{d} \tau.
			\end{split}
		\end{equation}
		
		Similar to \eqref{eqL2-14_2} and \eqref{eqL2-14_3}, we use Proposition \ref{prop4} and the boundary condition $(u \cdot n)|_{\pl \Omega}=0$, we have
		\begin{equation}\label{eqL2-18}
			\begin{split}
				&
				\delta \mu( \epsilon) \int^t_0  \left\| \triangledown u \right\|^2_m\mathrm{d} \tau \leq \delta \mu(\epsilon) C_2 \int_{0}^{t} \left\| \triangledown \times Z^{\alpha} u \right\|^2 \mathrm{d} \tau + C_{m+2} \int_{0}^{t} ( \left\| u \right\|^2_m + \left\| \triangledown u \right\|^2_{m-1} \,) \mathrm{d} \tau.
			\end{split}
		\end{equation}

		Then we obtain
		\begin{equation}\label{eqL2-19}
			\begin{split}
				&
				- \int^t_0 \int_{\Omega} \mu( \epsilon) Z^{\alpha} (\triangledown \times \omega) \cdot Z^{\alpha} u \, \mathrm{d} x \mathrm{d} \tau
				\\
				&\quad
				\leq  \delta \mu(\epsilon)^2 \int^t_0 \left\| \triangledown^2 u \right\|^2_{m-1} \mathrm{d} \tau + C_{\delta} C_{m+2} \int^t_0 ( \left\| u \right\|^2_{m}+ \left\| \triangledown u \right\|^2_{m-1}) \mathrm{d} \tau
				- \frac{1}{2} \mu (\epsilon) \int^t_0 \left\| \triangledown \times Z^{\alpha} u \right\|^2 \mathrm{d} \tau.
			\end{split}
		\end{equation}
		
		We combine \eqref{eqL2-1} \eqref{eqL2-2}, \eqref{eqL2-3}, \eqref{eqL2-9}, \eqref{eqL2-15}, and \eqref{eqL2-18} to get
		\begin{equation}\label{eqL2-20}
			\begin{split}
				&
				\frac{1}{2} \left\| u (t) \right\|^2_{m} + \frac{1}{2} \left\| B (t) \right\|^2_{m} + \frac{1}{2} \mu (\epsilon) \int_{0}^{t} \left\| \triangledown \times Z^{\alpha} u \right\|^2 \mathrm{d} \tau + \frac{1}{2} \nu (\epsilon) \int_{0}^{t} \left\| \triangledown \times Z^{\alpha} B\right\|^2 \mathrm{d} \tau
				\\
				&\quad
				\leq \frac{1}{2} \left\| u (0) \right\|^2_{m} + \frac{1}{2} \left\| B (0) \right\|^2_{m} + \mu(\epsilon)^2 \delta \int^t_0 \left\| \triangledown^2 u \right\|^2_{m-1} \mathrm{d} \tau + \nu(\epsilon)^2 \delta \int^t_0 \left\| \triangledown^2 B \right\|^2_{m-1} \mathrm{d} \tau 
				\\
				&\quad \quad
				+ C_{\delta} C_{m+2} (1 + P(\left\| (u,B,Zu,ZB,\triangledown u,\triangledown B) \right\|^2_{L^{\infty}} ))\int_{0}^{t} P( \left\|(u,B)\right\|^2_{m} + \left\|(\triangledown u, \triangledown B)\right\|^2_{m-1} ) \mathrm{d}\tau
				\\
				&\quad \quad
				+  C_{m+2} \int^t_0 ( \left\| \triangledown^2 q_1 \right\|^2_{m-1} + \frac{1}{\mu(\epsilon)} \left\| \triangledown q_2 \right\|^2_{m-1} \,) \mathrm{d} \tau + \delta \int_{0}^{t} \left\| \triangledown q_2 \right\|^2_{1} \mathrm{d} \tau.
			\end{split}
		\end{equation}
		
		We use Proposition \ref{prop4} again for the terms $ \left\| \triangledown \times Z^{\alpha} u \right\|^2$ and $\left\| \triangledown \times Z^{\alpha} B\right\|^2$
		\begin{equation}\label{eqL2-21}
			\begin{split}
				&
				\frac{1}{2} \mu (\epsilon) \int^{t}_{0}  \left\| \triangledown \times Z^{\alpha} u \right\|^2 \,  \mathrm{d} \tau + \frac{1}{2} \nu (\epsilon) \int^{t}_{0}  \left\| \triangledown \times Z^{\alpha} B\right\|^2 \,  \mathrm{d} \tau
				\\
				&\quad
				\geq \frac{1}{2C_2} \mu (\epsilon) \int^{t}_{0} ( \left\|  Z^{\alpha} \triangledown u + [\triangledown , Z^{\alpha}] u \right\|^2 \, ) \mathrm{d} \tau + \frac{1}{2C_2} \nu (\epsilon) \int^{t}_{0}  \left\| Z^{\alpha} \triangledown  B + [\triangledown , Z^{\alpha}] B\right\|^2 \,  \mathrm{d} \tau
				\\
				&\quad\quad
				-\frac{1}{2} \mu (\epsilon) \int^{t}_{0} ( \left\|  div Z^{\alpha} u \right\|^2 + \left\|u\right\|^2_m \, ) \mathrm{d} \tau - \frac{1}{2} \nu (\epsilon) \int^{t}_{0} ( \left\| div Z^{\alpha} B \right\|^2 + \left\|B\right\|^2_m \, ) \mathrm{d} \tau
				\\
				&\quad\quad
				- \frac{1}{2} \mu (\epsilon) \int_{0}^{t} | Z^{\alpha} u \cdot n|_{H^{\frac{1}{2}}(\pl \Omega)}^2 \,  \mathrm{d} \tau - \frac{1}{2} \nu (\epsilon) \int_{0}^{t} | Z^{\alpha} B \cdot n |_{H^{\frac{1}{2}}(\pl \Omega)}^2 \,  \mathrm{d} \tau,
			\end{split}
		\end{equation}
		with $divu=0$, $div B=0$. We have for the boundary term that $(u \cdot n)|_{\pl \Omega}=0$, $(B \cdot n)|_{\pl \Omega}=0$ and $\alpha_3=0$ since if $\alpha_3 \neq 0$, it will vanish, thus we apply Proposition \ref{prop3} (trace theorem) similarly to \eqref{eqL2-14_3} to get
		\begin{equation}\label{eqL2-22}
			\begin{split}
				&
				\frac{1}{2} \mu (\epsilon) \int^{t}_{0}  \left\| \triangledown \times Z^{\alpha} u \right\|^2 \,  \mathrm{d} \tau + \frac{1}{2} \nu (\epsilon) \int^{t}_{0}  \left\| \triangledown \times Z^{\alpha} B\right\|^2 \,  \mathrm{d} \tau
				\\
				&\quad
				\geq \frac{1}{2C_2} \mu (\epsilon) \int^{t}_{0}  \left\|  Z^{\alpha} \triangledown u \right\|^2 \,  \mathrm{d} \tau + \frac{1}{2C_2} \nu (\epsilon) \int^{t}_{0}  \left\| Z^{\alpha} \triangledown B \right\|^2 \,  \mathrm{d} \tau
				- C_{m+2} \int^{t}_{0} ( \left\| ( u, B ) \right\|^2_m + \left\| (\triangledown u , \triangledown B) \right\|^2_{m-1} \, ) \mathrm{d} \tau.
			\end{split}
		\end{equation}

		Combining \eqref{eqL2-20} and \eqref{eqL2-22}, we end the proof.
		
	 \end{proof}

		\hspace*{\fill}\\
		\hspace*{\fill}

		\subsection{Normal Derivative Estimates}
		
		In this section, we need to estimate $\left\| \triangledown u \right\|_{m-1}$ and $\left\| \triangledown B \right\|_{m-1}$, then it remains to estimate $\left\| \chi \triangledown u \right\|_{m-1}$ and $\left\| \chi \triangledown B \right\|_{m-1}$ with $\chi$ is compactly supported in one of the $\Omega_j$ and with value one in a neighborhood of the boundary. We have $\left\| \chi \pl_{y^i} u \right\|_{m-1} \leq C \left\| u \right\|_{m}$ and $\left\| \chi \pl_{y^i} B \right\|_{m-1} \leq C \left\| B \right\|_{m}$ for $i=1,2$. So we remain to estimate $\left\| \chi \pl_n u \right\|_{m-1}$ and $\left\| \chi \pl_n B \right\|_{m-1}$.
		
		Note that we can rewrite $div u$ and $div B$ as following
		\begin{equation}\label{n1}
			\left\{
			\begin{array}{lr}
				div u = \pl_n u \cdot n + (\Pi \pl_{y^1} u )_1 + (\Pi \pl_{y^2} u)_2 =0,\\
				div B = \pl_n B \cdot n + (\Pi \pl_{y^1} B )_1 + (\Pi \pl_{y^2} B)_2 =0,
			\end{array}
			\right.
		\end{equation}
		and we also have that
		\begin{equation}
			\label{n2}
			\left\{
			\begin{array}{lr}
				\pl_n u = (\pl_n u \cdot n ) n + \Pi (\pl_n u),\\
				\pl_n B = (\pl_n B \cdot n ) n + \Pi (\pl_n B).
			\end{array}
			\right.
		\end{equation}
		
		Thus, we will get
		\begin{equation}
			\begin{split}
				&
				\left\| \chi \pl_n u \right\|_{m-1} \leq \left\| \chi \pl_n u \cdot n \right\|_{m-1} + \left\| \chi \Pi (\pl_n u) \right\|_{m-1}
				\leq  C_m ( \left\| \chi \Pi (\pl_n u) \right\|_{m-1} + \left\|u \right\|_{m} ).
			\end{split}
		\end{equation}
		
		Similarly,
		\begin{equation}
			\left\| \chi \pl_n B \right\|_{m-1} \leq  C_m ( \left\| \chi \Pi (\pl_n B) \right\|_{m-1} + \left\|B \right\|_{m} ).
		\end{equation}
		
		We now define
		\begin{equation}\label{eta}
			\eta = \chi ( \omega \times n + 2 \Pi(\alpha u - S(n) u) ) = \chi ( \Pi(\omega \times n) + 2 \Pi(\alpha u - S(n) u) ) .
		\end{equation}
		
		Recall \eqref{eq1-13} and define $K=2(\alpha I -S(n))$ with $I$ is the identity matrix, one has
		\begin{equation}
			u\cdot n =0, \quad n \times \omega = 2 \Pi(\alpha u - S(n) u)= \Pi (Ku),
		\end{equation}
		and we will have $\eta|_{\pl \Omega}=0$.
		
		Since $\omega \times n = (\triangledown u - (\triangledown u)^t) \cdot n$, thus $\eta$ can be rewritten as (one can refer to \cite{M2012} and \cite{XZ2015} for details)
		\begin{equation}
			\eta = \chi ( \Pi (\pl_n u) - \Pi( \triangledown (u \cdot n)) + \Pi (( \triangledown n)^t \cdot u) + \Pi (Ku)),
		\end{equation}
		which shows that
		\begin{equation}\label{nu}
			\left\| \chi \Pi (\pl_n u) \right\|_{m-1} \leq C_{m-1} ( \left\| \eta \right\|_{m-1} + \left\| u \right\|_{m}).
		\end{equation}
		
		Similarly, since we have $\eqref{eq1-2}_2$ such that $\curl B \times n =0$ on $\pl \Omega$, thus we get
		\begin{equation}
			\chi (\curl B \times n) = \chi ( \Pi (\pl_n B) - \Pi( \triangledown (B \cdot n)) + \Pi (( \triangledown n)^t \cdot B) ),
		\end{equation}
		with $\chi (\curl B \times n)|_{\pl \Omega}=0$ and it yields
		\begin{equation}\label{nB}
			\left\| \chi \Pi (\pl_n B) \right\|_{m-1} \leq C_{m-1} ( \left\| \chi (\curl B \times n) \right\|_{m-1} + \left\| B \right\|_{m}).
		\end{equation}
		
		Therefore, we remain to estimate $\left\| \eta \right\|_{m-1}$ and $\left\| \chi (\curl B \times n) \right\|_{m-1}$ in order to get the normal derivative estimates.
		
		\hspace*{\fill}\\
		\hspace*{\fill}

		\begin{lem}\label{lem3}
			For a smooth solution to \eqref{eq1-1} and \eqref{eq1-2}, it holds that for all $\epsilon \in (0,1]$
			\begin{equation}\label{eqL3}
				\begin{split}
					&
					\frac{1}{2} \left\| \eta (t) \right\|^2 + \frac{1}{2} \left\| \chi (\curl B \times n) (t) \right\|^2 +  \frac{3}{4} \mu (\epsilon) \int^{t}_{0}  \left\| \triangledown \eta \right\|^2 \,  \mathrm{d} \tau + \frac{3}{4} \nu (\epsilon )\int^{t}_{0}  \left\| \triangledown ( \chi (\curl B \times n) )\right\|^2 \,  \mathrm{d} \tau
					\\
					&\quad
					\leq \frac{1}{2} \left\| \eta (0) \right\|^2 + \frac{1}{2} \left\| \chi (\curl B \times n) (0) \right\|^2 + \mu(\epsilon)^2 \delta \int^t_0 \left\| \triangledown^2 u \right\|^2 \,  \mathrm{d} \tau + \nu(\epsilon)^2 \delta \int^t_0 \left\| \triangledown^2 B \right\|^2 \,  \mathrm{d} \tau 
					\\
					&\quad \quad
					+ C_{\delta} C_3(1 + \left\| (u,B,\triangledown u,\triangledown B) \right\|^2_{L^{\infty}} )\int_{0}^{t} (  \left\|(u,B)\right\|^2_{1} + \left\|(\triangledown u, \triangledown B)\right\|^2+ \left\| \triangledown q \right\|^2\, ) \mathrm{d}\tau.
				\end{split}
			\end{equation}
		\end{lem}

		\begin{proof}
		Taking $\curl$ of \eqref{q}, we get
		\begin{equation}\label{eqL3-2}
			\pl_t \omega + (u \cdot \triangledown) \omega - (\omega \cdot \triangledown)u = (B \cdot \triangledown) \curl B - (\curl B \cdot \triangledown)B + \mu (\epsilon) \triangle \omega.
		\end{equation}
		
		Recall we have
		\begin{equation}
			\pl_t u + u \cdot \triangledown u + \triangledown q = (B\cdot \triangledown)B + \mu(\epsilon) \triangle u.
		\end{equation}
		
		By the definition of $\eta$, we then have
		\begin{equation}\label{eqL3-3}
			\pl_t \eta + (u \cdot \triangledown) \eta - (B \cdot \triangledown) ( \chi  (\curl B \times n)) - \mu (\epsilon) \triangle \eta  = \chi F_1 + \chi F_2 + F_3 + F_4=F,
		\end{equation}
		where
		\begin{equation}\label{eqF1}
			F_1= (\omega \cdot \triangledown)u \times n - (\curl B \cdot \triangledown)B \times n - \Pi(K \triangledown q) + \Pi(K((B \cdot \triangledown) B)),
		\end{equation}
		\begin{equation}\label{eqF2}
			\begin{split}
				&
				F_2= - 2 \sum^2_{i=1} \mu (\epsilon) \pl_i \omega \times \pl_i n - \mu (\epsilon) \omega \times \triangle n + \sum_{i=1}^2 u_i \omega \times \pl_{y^i} n 
				\\
				&\quad
				- \sum^2_{i=1} B_i \curl B \times \pl_{y^i} n + \sum^2_{i=1} u_i \pl_{y^i} (\Pi K) u - 2 \mu (\epsilon) \sum^2_{i=1} \pl_i (\Pi K) \pl_i u,
			\end{split}
		\end{equation}
		\begin{equation}\label{eqF3}
			\begin{split}
				&
				F_3= \sum^2_{i=1} u_i \pl_{y^i} \chi \cdot ( \omega \times n + \Pi(Ku)) + u \cdot N \pl_z \chi \cdot ( \omega \times n + \Pi(Ku))
				\\
				&\quad
				-2 \sum^3_{i=1} \mu (\epsilon) \pl_i \chi \cdot \pl_i ( \omega \times n + \Pi(Ku)) - \mu( \epsilon) \triangle \chi \cdot( \omega \times n + \Pi(Ku))
				\\
				&\quad
				- \sum^2_{i=1} B_i \pl_{y^i} \chi  \cdot (\curl B \times n) - B \cdot N \pl_z \chi \cdot (\curl B \times n),
			\end{split}
		\end{equation}
		and
		\begin{equation}\label{eqF4}
			F_4= \mu (\epsilon )\chi \triangle (\Pi K) u.
		\end{equation}

		Multiplying \eqref{eqL3-3} with $\eta$ and integrating with respect to $x$ and $t$, we obtain
		\begin{equation}\label{eqL3-4}
			\begin{split}
				&
				\int_{0}^{t} \int_{\Omega} \pl_t \eta \cdot \eta + (u \cdot \triangledown )\eta \cdot \eta - (B \cdot \triangledown )(\chi (\curl B \times n) \cdot \eta - \mu (\epsilon) \triangle \eta \cdot \eta \, \mathrm{d}x \mathrm{d} \tau 
				= \int_{0}^{t} \int_{\Omega} F \cdot \eta \, \mathrm{d} x \mathrm{d} \tau.
			\end{split}
		\end{equation}
		
		Integrating by parts, one has
		\begin{equation}\label{eqL3-5}
			\int_{0}^{t} \int_{\Omega} \pl_t \eta \cdot \eta \, \mathrm{d}x \mathrm{d} \tau = \frac{1}{2} \left\| \eta(t) \right\|^2 - \frac{1}{2} \left\| \eta(0) \right\|^2,
		\end{equation}
		and by the vetor equality such that
		$$div (\triangledown \eta \cdot \eta)= \triangle \eta \cdot \eta + \triangledown \eta : \triangledown \eta $$
		with $\triangledown \eta : \triangledown \eta = \sum_{i=1,2,3} \triangledown \eta_i \cdot \triangledown \eta_i$, thus we have
		\begin{equation}\label{eqL3-6}
			\int_{0}^{t} \int_{\Omega} - \mu (\epsilon )\triangle \eta \cdot \eta \, \mathrm{d}x \mathrm{d} \tau = \int_{0}^{t} \int_{\Omega} - \mu (\epsilon) \triangledown \cdot( \triangledown \eta \cdot \eta)+ \mu (\epsilon)| \triangledown \eta|^2 \, \mathrm{d}x \mathrm{d} \tau = \mu (\epsilon) \int^t_0 \left\| \triangledown \eta \right\|^2 \mathrm{d} \tau,
		\end{equation}
		where we apply the boundary condition $\eta|_{\pl \Omega}=0$ by the definition of $\eta$.
		
		Integrating by parts and applying the boundary condition $\eqref{eq1-2}_1$ such that $u \cdot n=0$ on $\pl \Omega$
		\begin{equation}\label{eqL3-7}
			\int_{0}^{t} \int_{\Omega}  (u \cdot \triangledown )\eta \cdot \eta \, \mathrm{d}x \mathrm{d} \tau = \int_{0}^{t} \int_{\Omega}  \triangledown \cdot (u  \frac{1}{2}|\eta|^2)- div u \frac{1}{2}|\eta|^2 \, \mathrm{d}x \mathrm{d} \tau =0,
		\end{equation}
		where we use $div u =0$.
		
		Similarly, by $B \cdot n=0$ on $\pl \Omega$ and $div B =0$, we obtain
		\begin{equation}\label{eqL3-8}
			\begin{split}
				&
				\int_{0}^{t} \int_{\Omega} - (B \cdot \triangledown )(\chi (\curl B \times n)) \cdot \eta \, \mathrm{d}x \mathrm{d} \tau
				\\
				&\quad
				= \int_{0}^{t} \int_{\Omega} -  \triangledown \cdot (B (\chi (\curl B \times n)) \cdot \eta)+ div B (\chi (\curl B \times n)) \cdot \eta + (B \cdot \triangledown) \eta \cdot (\chi (\curl B \times n)) \, \mathrm{d}x \mathrm{d} \tau
				\\
				&\quad
				=\int_{0}^{t} \int_{\Omega} (B \cdot \triangledown)\eta \cdot (\chi (\curl B \times n)) \, \mathrm{d}x \mathrm{d} \tau.
			\end{split}
		\end{equation}
		
		Recall $\eqref{eq1-1}_2$
		\begin{equation}
			\pl_t B - \curl(u \times B)= \nu (\epsilon) \triangle B
		\end{equation}
		
		Taking $\curl$ of it, one has
		\begin{equation}\label{eqL3-9}
			\pl_t \curl B - (B \cdot \triangledown) \omega + (u \cdot \triangledown)\curl B - ( \omega \cdot \triangledown)B + (\curl B \cdot \triangledown)u= \nu (\epsilon) \triangle \curl B,
		\end{equation}
		and thus,
		\begin{equation}\label{eqL3-10}
			\pl_t ( \chi (\curl B \times n)) - (B \cdot \triangledown) \eta + (u \cdot \triangledown)( \chi (\curl B \times n))- \nu (\epsilon) \triangle ( \chi (\curl B \times n)) = G,
		\end{equation}
		where $G= G_1+ G_2 +G_3$ with
		\begin{equation}\label{G1}
			G_1= \chi ( \omega \cdot \triangledown )B \times n - \chi (\curl B \cdot \triangledown)(u \times n) - (B \cdot \triangledown)(\chi \Pi(Ku)),
		\end{equation}
		\begin{equation}\label{G2}
			\begin{split}
				&
				G_2= - \nu (\epsilon) \triangle \chi \cdot \curl B \times n - 2 \nu (\epsilon) \sum^3_{i=1} \pl_i \chi \cdot \pl_i(\curl B \times n) - \sum^2_{i=1} B_i \pl_{y^i} \chi \cdot (\omega \times n)
				\\
				&\quad
				- B \cdot N \pl_z \chi \cdot (\omega \times n) + \sum^2_{i=1} u_i \pl_{y^i} \chi \cdot (\curl B \times n) + u \cdot N \pl_z \chi \cdot (\curl B \times n),
			\end{split}
		\end{equation}
		and
		\begin{equation}\label{G3}
			\begin{split}
				&
				G_3= -\nu (\epsilon) \chi ( \curl B \times \triangle n) - 2 \nu (\epsilon )\chi \sum^2_{i=1} \pl_i \curl B \times \pl_i \times n 
				- \sum^2_{i=1} B_i \chi \omega \times \pl_{y^i} n + \sum^2_{i=1} u_i \chi \omega \times \pl_{y^i} n.
			\end{split}
		\end{equation}
		
		Putting \eqref{eqL3-10} into \eqref{eqL3-8}, we get
		\begin{equation}\label{eqL3-11}
			\begin{split}
				&
				\int_{0}^{t} \int_{\Omega} (B \cdot \triangledown) \eta \cdot (\chi (\curl B \times n)) \, \mathrm{d}x \mathrm{d} \tau
				\\
				&\quad
				= \int_{0}^{t} \int_{\Omega} \pl_t (\chi (\curl B \times n)) \cdot \chi (\curl B \times n) + (u \cdot \triangledown)(\chi (\curl B \times n)) \cdot \chi (\curl B \times n) 
				\\
				&\quad\quad
				- \nu (\epsilon) \triangle \chi (\curl B \times n) \cdot \chi (\curl B \times n) - G \cdot \chi (\curl B \times n) \, \mathrm{d}x \mathrm{d} \tau.
			\end{split}
		\end{equation}
		
		First, we have
		\begin{equation}\label{eqL3-12}
			\begin{split}
				&
				\int_{0}^{t} \int_{\Omega} \pl_t (\chi (\curl B \times n)) \cdot \chi (\curl B \times n) \, \mathrm{d}x \mathrm{d} \tau
				= \frac{1}{2} \left\| \chi ( \curl B \times n)(t) \right\|^2 - \frac{1}{2} \left\| \chi ( \curl B \times n)(0) \right\|^2.
			\end{split}
		\end{equation}
		
		Integrating by parts and applying boundary conditions $\eqref{eq1-2}_1$ such that $(u\cdot n)|_{\pl \Omega} =0$
		\begin{equation}\label{eqL3-13}
			\begin{split}
				&
				\int_{0}^{t} \int_{\Omega} (u \cdot \triangledown)(\chi (\curl B \times n)) \cdot \chi (\curl B \times n) \, \mathrm{d}x \mathrm{d} \tau
				\\
				&\quad
				= \int_{0}^{t} \int_{\Omega} \triangledown \cdot (u \frac{1}{2} |\chi (\curl B \times n)|^2) - divu \frac{1}{2} |\chi (\curl B \times n)|^2 \, \mathrm{d}x \mathrm{d} \tau =0.
			\end{split}
		\end{equation}
		
		Since we have $(\chi (\curl B \times n) )|_{\pl \Omega}= (\curl B \times n) |_{\pl \Omega}=0$ by $\eqref{eq1-2}_2$ and do integration by parts same as \eqref{eqL3-6}
		\begin{equation}\label{eqL3-14}
			\begin{split}
				&
				\int_{0}^{t} \int_{\Omega} - \nu (\epsilon) \triangle \chi (\curl B \times n) \cdot \chi (\curl B \times n) \, \mathrm{d}x \mathrm{d} \tau
				= \nu (\epsilon) \int_{0}^{t}  \left\| \triangledown (\chi (\curl B \times n)) \right\|^2 \, \mathrm{d} \tau.
			\end{split}
		\end{equation}
		
		Applying Proposition \ref{prop1}, we get
		\begin{equation}\label{eqL3-15}
			\begin{split}
				&
				\int^t_0 \left\| G_1 \right\|^2_{m-1} \mathrm{d}\tau = \int^t_0 \left\| \chi ( \omega \cdot \triangledown )B \times n - \chi (\curl B \cdot \triangledown)(u \times n) - (B \cdot \triangledown)(\chi \Pi(Ku)) \right\|^2_{m-1} \mathrm{d}\tau
				\\
				&\quad
				\geq - C_{m+2} ( 1+ \left\| ( u,B, \triangledown u, \triangledown B) \right\|^2_{L^{\infty}} ) \int^t_0 (\left\| u \right\|^2_m + \left\| B \right\|^2_m + \left\| \triangledown u \right\|^2_{m-1} + \left\| \triangledown B \right\|^2_{m-1} \,) \mathrm{d} \tau.
			\end{split}
		\end{equation}
		Similarly, we have $G_2$ supported away from the boundary
		\begin{equation}\label{eqL3-16}
			\begin{split}
				&
				\int^t_0 \left\| G_2 \right\|^2_{m-1} \mathrm{d}\tau \geq - C_{m+2} \nu(\epsilon)^2 \int^t_0 \left\| \triangledown^2 B \right\|^2_{m-1} \, \mathrm{d} \tau
				\\
				&\quad
				- C_{m+2} ( 1+ \left\| ( u,B, \triangledown u, \triangledown B) \right\|^2_{L^{\infty}} ) \int^t_0 ( \left\| u \right\|^2_m + \left\| B \right\|^2_m + \left\| \triangledown u \right\|^2_{m-1} + \left\| \triangledown B \right\|^2_{m-1} \,) \mathrm{d} \tau.
			\end{split}
		\end{equation}
		and
		\begin{equation}\label{eqL3-17}
			\begin{split}
				&
				\int^t_0 \left\| G_3 \right\|^2_{m-1}\mathrm{d}\tau \geq - C_{m+2} \nu(\epsilon)^2 \int^t_0 \left\| \chi \triangledown^2 B \right\|^2_{m-1} \, \mathrm{d} \tau
				\\
				&\quad
				- C_{m+2} ( 1+ \left\| ( u,B, \triangledown u) \right\|^2_{L^{\infty}} ) \int^t_0 (\left\| u \right\|^2_m + \left\| B \right\|^2_m + \left\| \triangledown u \right\|^2_{m-1} \, )\mathrm{d} \tau.
			\end{split}
		\end{equation}
		
		Therefore, here we choose $m=1$
		\begin{equation}\label{eqL3-18}
			\begin{split}
				&
				\int_{0}^{t} \int_{\Omega} - (B \cdot \triangledown )(\chi (\curl B \times n)) \cdot \eta \, \mathrm{d}x \mathrm{d} \tau
				\\
				&\quad
				\geq \frac{1}{2} \left\| \chi ( \curl B \times n)(t) \right\|^2 - \frac{1}{2} \left\| \chi ( \curl B \times n)(0) \right\|^2
				\\
				&\quad\quad
				+ \nu (\epsilon) \int_{0}^{t}  \left\| \triangledown (\chi (\curl B \times n)) \right\|^2 \, \mathrm{d} \tau - \delta \nu( \epsilon)^2 \int^t_0 \left\| \chi \triangledown^2 B \right\|^2 \, \mathrm{d} \tau
				\\
				&\quad\quad
				- C_{\delta}C_{3} ( 1+ \left\| ( u,B, \triangledown u, \triangledown B) \right\|^2_{L^{\infty}} ) \int^t_0 (\left\| u \right\|^2_1 + \left\| B \right\|^2_1 + \left\| \triangledown u \right\|^2 + \left\| \triangledown B \right\|^2 \,) \mathrm{d} \tau .
			\end{split}
		\end{equation}

		Lastly, we consider the term involving $F$. We apply Proposition 2.1
		\begin{equation}\label{eqL3-19}
			\begin{split}
				&
				\int^t_0 \left\| \chi F_1 \right\|^2_{m-1} \mathrm{d}\tau \leq \int^t_0 \left\| F_1 \right\|^2_{m-1}\mathrm{d}\tau
				\\
				&\quad
				\leq C_{m+1}((1+ \left\|( B, \triangledown u, \triangledown B ) \right\|^2_{L^{\infty}} ) \int_{0}^{t} ( \left\| B \right\|^2_{m-1} + \left\| \triangledown u \right\|^2_{m-1}+ \left\| \triangledown B \right\|^2_{m-1} \mathrm{d}\tau + \left\| \triangledown q \right\|^2_{m-1} \,) \mathrm{d}\tau,
			\end{split}
		\end{equation}
		\begin{equation}\label{eqL3-20}
			\begin{split}
				&
				\int^t_0 \left\| \chi F_2 \right\|^2_{m-1} \mathrm{d}\tau \leq \int^t_0 \left\| F_2 \right\|^2_{m-1} \mathrm{d}\tau \leq C_{m+1} \mu(\epsilon)^2 \int^t_0 \left\| \chi \triangledown^2 u \right\|^2_{m-1} \, \mathrm{d} \tau
				\\
				&\quad
				+ C_{m+2}(1+ \left\|(u, B, \triangledown u, \triangledown B ) \right\|^2_{L^{\infty}} ) \int_{0}^{t} (\left\| u \right\|^2_{m} + \left\| B \right\|^2_{m-1}+ \left\| \triangledown u \right\|^2_{m-1}  + \left\| \triangledown B \right\|^2_{m-1} \,)\mathrm{d}\tau,
			\end{split}
		\end{equation}
		and $F_3$ supported away from the boundary
		\begin{equation}\label{eqL3-21}
			\begin{split}
				&
				\int^t_0 \left\| F_3 \right\|^2_{m-1}\mathrm{d}\tau  \leq C_{m+1} \mu (\epsilon)^2 \int^t_0 \left\| \triangledown^2 u \right\|^2_{m-1} \, \mathrm{d} \tau
				\\
				&\quad
				+ C_{m+2}(1+ \left\|(u, B, \triangledown u, \triangledown B ) \right\|^2_{L^{\infty}} ) \int_{0}^{t} ( \left\| u \right\|^2_{m-1} + \left\| B \right\|^2_{m-1} + \left\| \triangledown u \right\|^2_{m-1}  + \left\| \triangledown B \right\|^2_{m-1} \,)\mathrm{d}\tau.
			\end{split}
		\end{equation}
		
		For $F_4$, we need to do integration by parts, we will get for $|\alpha|=m-1$
		\begin{equation}\label{eqL3-22}
			\begin{split}
				&
				\int^t_0 \int_{\Omega} Z^{\alpha} F_4 \cdot Z^{\alpha} \eta \, \mathrm{d} x \mathrm{d} \tau = \int^t_0 \int_{\Omega} \mu (\epsilon) \chi Z^{\alpha} \triangle (\Pi K) u \cdot Z^{\alpha} \eta \, \mathrm{d} x \mathrm{d} \tau
				\\
				&\quad
				\leq \delta \mu (\epsilon)^2 \int^t_0 \left\| \triangledown \eta \right\|^2_{m-1} \, \mathrm{d} \tau + \delta \mu (\epsilon)^2 \int^t_0 \left\| \triangledown \eta \right\|^2_{m-2} \, \mathrm{d} \tau + C_{\delta} C_{m+2} \int_{0}^{t} ( \left\| u \right\|^2_{m-1} + \left\| \triangledown u \right\|^2_{m-1} \,) \mathrm{d} \tau.
			\end{split}
		\end{equation}
		
		Consequently, substituting these estimates into \eqref{eqL3-4} and combining it with \eqref{eqL3-5}, \eqref{eqL3-6}, \eqref{eqL3-7}, \eqref{eqL3-18}, we complete the proof of Lemma \ref{lem3}.
		
	\end{proof}

		\hspace*{\fill}\\
		\hspace*{\fill}

		\begin{lem}\label{lem4}
			For a smooth solution to \eqref{eq1-1} and \eqref{eq1-2}, it holds that for all $\epsilon \in (0,1]$ 
			\begin{equation}\label{eqL4}
				\begin{split}
					&
					\frac{1}{2} \left\| \eta (t) \right\|_{m-1}^2 + \frac{1}{2} \left\| \chi (\curl B \times n) (t) \right\|_{m-1}^2 + \frac{1}{2} \mu (\epsilon) \int^{t}_{0}  \left\| \triangledown \eta \right\|_{m-1}^2 \,  \mathrm{d} \tau + \frac{1}{2} \nu (\epsilon) \int^{t}_{0} ( \left\| \triangledown  \chi (\curl B \times n) )\right\|_{m-1}^2 \,  \mathrm{d} \tau
					\\
					&\quad
					\leq C_{m+2}( \frac{1}{2} \left\| \eta (0) \right\|_{m-1}^2 + \frac{1}{2} \left\| \chi (\curl B \times n) (0) \right\|_{m-1}^2 + \mu(\epsilon)^2 \delta \int^t_0 \left\| \triangledown^2 u \right\|_{m-1}^2 \,  \mathrm{d} \tau + \nu(\epsilon)^2 \delta \int^t_0 \left\| \triangledown^2 B \right\|_{m-1}^2 \,  \mathrm{d} \tau 
					\\
					&\quad \quad
					+ C_{\delta} (1 + \left\| (u,B,\triangledown u,\triangledown B, Z \triangledown u, Z \triangledown B) \right\|^2_{L^{\infty}} )\int^{t}_{0} ( \left\|(u,B)\right\|^2_{m} + \left\|(\triangledown u, \triangledown B)\right\|_{m-1}^2 + \left\| \triangledown q \right\|_{m-1}^2 \, ) \mathrm{d}\tau,
				\end{split}
			\end{equation}
			for every $m \in \mathbb{N}_+$.
		\end{lem}
		
		\begin{proof}
		
		The case for $m=1$ has already been proved in Lemma \ref{lem3}, we shall assume \eqref{eqL4} is proved for $k \leq m-2$ and then prove it holds for $k=m-1$ for $m \geq 2$. Let $|\alpha|=m-1$, multiplying $Z^{\alpha} \eqref{eqL3-3}$ with $Z^{\alpha} \eta$ and integrating it with respect to $x$ and $t$. We have
		\begin{equation}\label{eqL4-1}
			\begin{split}
				&
				\int_{0}^{t} \int_{\Omega} Z^{\alpha}\pl_t \eta \cdot Z^{\alpha}\eta + Z^{\alpha}((u \cdot \triangledown )\eta) \cdot Z^{\alpha} \eta - Z^{\alpha}(B \cdot \triangledown )(\chi (\curl B \times n) )\cdot Z^{\alpha}\eta 
				\\
				&\quad
				- \mu (\epsilon) Z^{\alpha} \triangle \eta \cdot Z^{\alpha} \eta \, \mathrm{d}x \mathrm{d} \tau = \int_{0}^{t} \int_{\Omega} Z^{\alpha}F \cdot Z^{\alpha} \eta \, \mathrm{d} x \mathrm{d} \tau.
			\end{split}
		\end{equation}
		
		Firstly, we have
		\begin{equation}\label{eqL4-2}
			\int_{0}^{t} \int_{\Omega} Z^{\alpha}\pl_t \eta \cdot Z^{\alpha}\eta \, \mathrm{d}x \mathrm{d} \tau = \frac{1}{2} \left\| \eta (t) \right\|^2_{m-1} - \frac{1}{2} \left\| \eta (0) \right\|^2_{m-1}.
		\end{equation}
		
		Recall we have
		$$div (Z^{\alpha}\triangledown \eta \cdot Z^{\alpha} \eta)= [\triangledown \cdot , Z^{\alpha}] \triangledown \eta \cdot Z^{\alpha} \eta + Z^{\alpha} \triangle \eta \cdot Z^{\alpha} \eta + \triangledown Z^{\alpha} \eta : Z^{\alpha} \triangledown \eta $$
		with $\triangledown Z^{\alpha} \eta : Z^{\alpha}  \triangledown \eta = \sum_{i=1,2,3} \triangledown Z^{\alpha} \eta_i \cdot Z^{\alpha} \triangledown \eta_i$.
		Integrating by parts and applying the boundary condition such that $Z^{\alpha} \eta |_{\pl \Omega}=0$ for all $\alpha$
		\begin{equation}\label{eqL4-3}
			\begin{split}
				&
				\int_{0}^{t} \int_{\Omega} - \mu( \epsilon )Z^{\alpha} \triangle \eta \cdot Z^{\alpha} \eta \, \mathrm{d}x \mathrm{d} \tau
				\\
				&\quad
				= \mu (\epsilon)\int_{0}^{t} \int_{\Omega} - \triangledown \cdot (Z^{\alpha} \triangledown \eta Z^{\alpha} \eta) + [ \triangledown , Z^{\alpha}] \triangledown \eta \cdot Z^{\alpha} \eta + Z^{\alpha} \triangledown \eta : \triangledown Z^{\alpha} \eta \, \mathrm{d}x \mathrm{d} \tau
				\\
				&\quad
				= \mu (\epsilon)\int_{0}^{t} \int_{\Omega} [ \triangledown , Z^{\alpha}] \triangledown \eta \cdot Z^{\alpha} \eta + \triangledown Z^{\alpha} \eta : \triangledown Z^{\alpha} \eta - [\triangledown, Z^{\alpha}] \eta : \triangledown Z^{\alpha} \eta \, \mathrm{d}x \mathrm{d} \tau
				\\
				&\quad
				\geq \frac{3}{4} \mu (\epsilon) \int^t_0 \left\| \triangledown Z^{\alpha} \eta \right\|^2 \, \mathrm{d} \tau - C_{m} \mu (\epsilon) \int^t_0 \left\| \triangledown Z^{\beta} \eta \right\|^2 \, \mathrm{d} \tau + \mu(\epsilon)\int_{0}^{t} \int_{\Omega} [ \triangledown , Z^{\alpha}] \triangledown \eta \cdot Z^{\alpha} \eta  \, \mathrm{d}x \mathrm{d} \tau,
			\end{split}
		\end{equation}
		where $|\beta|=m-2$.
		
		In the local basis, we can write as
		\begin{equation}
			\begin{split}
				&
				\pl_j= \mathcal{\beta}^1_j \pl_{y^1} + \mathcal{\beta}^2_j \pl_{y^2} + \mathcal{\beta}^3_j \pl_{z}.
			\end{split}
		\end{equation}
		
		For the last term in \eqref{eqL4-3}, we need to do integration by parts
		\begin{equation}\label{eqL4-4}
			\begin{split}
				&
				\mu (\epsilon) \int_{0}^{t} \int_{\Omega} [ \triangledown , Z^{\alpha}] \triangledown \eta \cdot Z^{\alpha} \eta  \, \mathrm{d}x \mathrm{d} \tau = \mu (\epsilon)\int_{0}^{t} \int_{\Omega} \sum_{|\gamma | \leq m-2} C_{1\gamma} \pl_z Z^{\gamma} \triangledown \eta \cdot Z^{\alpha} \eta + C_{2\gamma} Z_y Z^{\gamma} \triangledown \eta \cdot Z^{\alpha} \eta  \, \mathrm{d}x \mathrm{d} \tau
				\\
				&\quad
				\geq - C_{m+1} \mu( \epsilon) \int^t_0 \left\| \triangledown \eta \right\|_{m-1}^2 \, \mathrm{d} \tau - \frac{1}{4} \mu (\epsilon) \int^t_0 \left\| \triangledown Z^{\alpha} \eta \right\|^2 \, \mathrm{d} \tau - C_{m+1} \int^t_0 ( \left\| u \right\|^2_{m} + \left\| \triangledown u \right\|^2_{m-1} \,) \mathrm{d} \tau.
			\end{split}
		\end{equation}
		
		Therefore, we have
		\begin{equation}\label{eqL4-5}
			\begin{split}
				&
				\int_{0}^{t} \int_{\Omega} - \mu (\epsilon )Z^{\alpha} \triangle \eta \cdot Z^{\alpha} \eta \, \mathrm{d}x \mathrm{d} \tau
				\\
				&\quad
				\geq \frac{1}{2}\mu (\epsilon) \int^t_0 \left\| \triangledown Z^{\alpha} \eta \right\|^2 \, \mathrm{d} \tau - C_{m+1} \mu (\epsilon) \int^t_0 \left\| \triangledown \eta \right\|_{m-1}^2 \, \mathrm{d} \tau - C_{m+1} \int^t_0 ( \left\| u \right\|^2_m + \left\| \triangledown u \right\|^2_{m-1} \,) \mathrm{d} \tau.
			\end{split}
		\end{equation}

		Integrating by parts, one obtains
		\begin{equation}\label{eqL4-6}
			\begin{split}
				&
				\int_{0}^{t} \int_{\Omega} Z^{\alpha}((u \cdot \triangledown )\eta) \cdot Z^{\alpha} \eta \, \mathrm{d}x \mathrm{d} \tau = \int_{0}^{t} \int_{\Omega} \sum_{\beta+ \gamma=\alpha, |\beta| \geq 1} \mathcal{C}_{\beta,\gamma} (Z^{\beta} u \cdot Z^{\gamma} \triangledown )\eta \cdot Z^{\alpha} \eta
				\\
				&\quad\quad
				+ (u \cdot \triangledown)Z^{\alpha}\eta \cdot Z^{\alpha}\eta + u \cdot [Z^{\alpha}, \triangledown] \eta \cdot Z^{\alpha} \eta \, \mathrm{d}x \mathrm{d} \tau
				\\
				&\quad
				= \int_{0}^{t} \int_{\Omega} \sum_{\beta+ \gamma=\alpha, |\beta| \geq 1} \mathcal{C}_{\beta,\gamma} (Z^{\beta} u \cdot Z^{\gamma} \triangledown )\eta \cdot Z^{\alpha} \eta + u \cdot [Z^{\alpha}, \triangledown] \eta \cdot Z^{\alpha} \eta \, \mathrm{d}x \mathrm{d} \tau,
			\end{split}
		\end{equation}
		where we used the boundary condition $\eqref{eq1-2}_1$ such that $u \cdot n=0$ on $\pl \Omega$.
		
		We can rewrite for $z>0$ and all $\beta$, $\gamma$
		\begin{equation}\label{eqL4-7}
			\begin{split}
				&
				(Z^{\beta} u \cdot Z^{\gamma} \triangledown )\eta = Z^{\beta} u_1 Z^{\gamma} \pl_{y^1}\eta + Z^{\beta} u_2 Z^{\gamma} \pl_{y^2}\eta + \frac{Z^{\beta} u \cdot N}{\varphi(z)} \varphi(z) Z^{\gamma} \pl_{z}\eta
				\\
				&\quad
				=Z^{\beta} u_1 Z^{\gamma} \pl_{y^1}\eta + Z^{\beta} u_2 Z^{\gamma} \pl_{y^2}\eta + \sum_{|\tilde{\beta}| + |\tilde{\gamma}| \leq m-1, |\tilde{\gamma}| \neq m-1} \mathcal{C}_{\tilde{\beta}, \tilde{\gamma}} Z^{\tilde{\beta}}\frac{ u \cdot N}{\varphi(z)}  Z^{\tilde{\gamma}} \varphi(z) \pl_{z}\eta,
			\end{split}
		\end{equation}
		and apply boundary condition $\eqref{eq1-2}_1$ again. Then we apply Proposition \ref{prop1} to get
		\begin{equation}\label{eqL4-8}
			\begin{split}
				&
				\int_{0}^{t} \int_{\Omega} \sum_{\beta+ \gamma=\alpha, |\beta| \geq 1} \mathcal{C}_{\beta,\gamma} (Z^{\beta} u \cdot Z^{\gamma} \triangledown )\eta \cdot Z^{\alpha} \eta + u \cdot [Z^{\alpha}, \triangledown] \eta \cdot Z^{\alpha} \eta \, \mathrm{d}x \mathrm{d} \tau
				\\
				&\quad \quad
				\geq -C_{m+2} (1+ \left\| (u, Zu, \triangledown u ,Z \triangledown u) \right\|^2_{L^{\infty}} ) \int^t_0 ( \left\| u \right\|^2_{m} + \left\| \triangledown u \right\|^2_{m-1} \,) \mathrm{d} \tau.
			\end{split}
		\end{equation}

		Since by the boundary condition $\eqref{eq1-2}_2$ such that $(B \cdot n)|_{\pl \Omega}=0$, we can also rewrite
		\begin{equation}
			\begin{split}
				&
				Z^{\beta} B \cdot Z^{\gamma} \triangledown (\chi ( \curl B \times n))
				=Z^{\beta} B_1 Z^{\gamma} \pl_{y^1}\chi ( \curl B \times n) + Z^{\beta} B_2 Z^{\gamma} \pl_{y^2}\chi ( \curl B \times n)
				\\
				&\quad
				+ \sum_{|\tilde{\beta}| + |\tilde{\gamma}| \leq m-1, |\tilde{\gamma}| \neq m-1} \mathcal{C}_{\tilde{\beta}, \tilde{\gamma}}Z^{\tilde{\beta}} \frac{ B \cdot N}{\varphi(z)} Z^{\tilde{\gamma}} \varphi(z) \pl_{z}\chi ( \curl B \times n),
			\end{split}
		\end{equation}
	    for $z>0$ and all $\beta$, $\gamma$.

		Thus, we do integration by parts and apply Proposition \ref{prop1} to get by $(B \cdot n)|_{\pl \Omega}=0$
		\begin{equation}\label{eqL4-9}
			\begin{split}
				&
				\int_{0}^{t} \int_{\Omega} - Z^{\alpha}((B \cdot \triangledown )(\chi (\curl B \times n) )\cdot Z^{\alpha}\eta \, \mathrm{d}x \mathrm{d} \tau
				\\
				&\quad
				= \int_{0}^{t} \int_{\Omega} - \sum_{\beta+ \gamma= \alpha, | \beta| \geq 1} \mathcal{C}_{\beta,\gamma} Z^{\beta} B \cdot Z^{\gamma} \triangledown (\chi ( \curl B \times n)) \cdot Z^{\alpha} \eta
				\\
				&\quad\quad
				- (B \cdot \triangledown) Z^{\alpha} (\chi (\curl B \times n)) \cdot Z^{\alpha} \eta + (B \cdot [\triangledown, Z^{\alpha}] (\chi (\curl B \times n)) \cdot Z^{\alpha} \eta \, \mathrm{d}x  \mathrm{d} \tau
				\\
				&\quad
				\geq \int_{0}^{t} \int_{\Omega} - \triangledown \cdot (B Z^{\alpha}\chi (\curl B \times n) \cdot Z^{\alpha} \eta) + div B Z^{\alpha}\chi (\curl B \times n) \cdot Z^{\alpha} \eta
				\\
				&\quad\quad
				+ (B \cdot \triangledown)Z^{\alpha}\eta \cdot Z^{\alpha}(\chi (\curl B \times n)) \, \mathrm{d}x  \mathrm{d} \tau
				\\
				&\quad\quad
				-C_{m+2} (1+ \left\| (B, ZB, \triangledown B ,Z \triangledown B) \right\|^2_{L^{\infty}} ) \int^t_0 (\left\| u \right\|^2_{m} + \left\| B \right\|^2_{m}+\left\| \triangledown u \right\|^2_{m-1} +\left\| \triangledown B \right\|^2_{m-1}\,) \mathrm{d} \tau
				\\
				&\quad
				=\int_{0}^{t} \int_{\Omega} (B \cdot \triangledown)Z^{\alpha}\eta \cdot Z^{\alpha}(\chi (\curl B \times n)) \, \mathrm{d}x  \mathrm{d} \tau
				\\
				&\quad\quad
				-C_{m+2} (1+ \left\| (B, ZB, \triangledown B ,Z \triangledown B) \right\|^2_{L^{\infty}} ) \int^t_0 (\left\| u \right\|^2_{m} + \left\| B \right\|^2_{m}+\left\| \triangledown u \right\|^2_{m-1} +\left\| \triangledown B \right\|^2_{m-1}\,) \mathrm{d} \tau.
			\end{split}
		\end{equation}

		Recall \eqref{eqL3-10}
		\begin{equation}
			\pl_t ( \chi (\curl B \times n)) - (B \cdot \triangledown) \eta + (u \cdot \triangledown)( \chi (\curl B \times n))- \nu (\epsilon) \triangle ( \chi (\curl B \times n)) = G,
		\end{equation}
		by applying $Z^{\alpha}$, we will get
		\begin{equation}\label{eqL4-10}
			\begin{split}
				&
				\pl_t Z^{\alpha}( \chi (\curl B \times n))+ (u \cdot \triangledown)Z^{\alpha}\chi (\curl B \times n) + (u \cdot [Z^{\alpha},\triangledown])\chi (\curl B \times n)
				\\
				&\quad
				+ \sum_{\beta+\gamma=\alpha, |\beta|\geq 1} \mathcal{C}_{\beta, \gamma} (Z^{\beta} u \cdot Z^{\gamma} \triangledown) \chi (\curl B \times n)+ \sum_{\beta+\gamma=\alpha, |\beta|\geq 1} \mathcal{C}_{\beta, \gamma} (Z^{\beta} B \cdot Z^{\gamma} \triangledown)\eta 
				\\
				&\quad
				- (B \cdot \triangledown)Z^{\alpha}\eta - (B \cdot [Z^{\alpha},\triangledown])\eta  
				- \nu (\epsilon) Z^{\alpha} \triangle ( \chi (\curl B \times n)) = Z^{\alpha} G.
			\end{split}
		\end{equation}
		
		Therefore applying Proposition \ref{prop1} and the boundary condition as we did above, we get
		\begin{equation}\label{eqL4-11}
			\begin{split}
				&
				\int_{0}^{t} \int_{\Omega} (B \cdot \triangledown)Z^{\alpha}\eta \cdot Z^{\alpha}(\chi (\curl B \times n)) \, \mathrm{d}x  \mathrm{d} \tau
				\\
				&\quad
				\geq \int_{0}^{t} \int_{\Omega} \pl_t Z^{\alpha}( \chi (\curl B \times n)) \cdot Z^{\alpha}( \chi (\curl B \times n))
				\\
				&\quad\quad
				+ (u \cdot \triangledown)Z^{\alpha}\chi (\curl B \times n) \cdot Z^{\alpha}( \chi (\curl B \times n)) 
				\\
				&\quad\quad
				- \nu (\epsilon) Z^{\alpha} \triangle ( \chi (\curl B \times n)) \cdot Z^{\alpha}( \chi (\curl B \times n)) - Z^{\alpha} G \cdot Z^{\alpha}( \chi (\curl B \times n)) \, \mathrm{d}x  \mathrm{d} \tau
				\\
				&\quad
				-C_{m+2} (1+ \left\| (u,B,Zu, ZB,\triangledown u, \triangledown B ,Z \triangledown u, Z \triangledown B) \right\|^2_{L^{\infty}} ) \int^t_0 (\left\| u \right\|^2_{m} + \left\| B \right\|^2_{m}+\left\| \triangledown u \right\|^2_{m-1} +\left\| \triangledown B \right\|^2_{m-1}\,) \mathrm{d} \tau.
			\end{split}
		\end{equation}

		We immediately get
		\begin{equation}\label{eqL4-12}
			\begin{split}
				&
				\int_{0}^{t} \int_{\Omega} \pl_t Z^{\alpha}( \chi (\curl B \times n)) \cdot Z^{\alpha}( \chi (\curl B \times n)) \, \mathrm{d}x  \mathrm{d} \tau
				\\
				&\quad
				= \frac{1}{2} \left\| \chi (\curl B \times n) (t) \right\|^2_{m-1} - \frac{1}{2} \left\| \chi (\curl B \times n) (0) \right\|^2_{m-1}.
			\end{split}
		\end{equation}
		
		Integrating by parts and using boundary condition $\eqref{eq1-2}_1$ such that $u \cdot n =0$ on $\pl \Omega$
		\begin{equation}\label{eqL4-13}
			\begin{split}
				&
				\int_{0}^{t} \int_{\Omega} (u \cdot \triangledown)Z^{\alpha}\chi (\curl B \times n) \cdot Z^{\alpha}( \chi (\curl B \times n)) \, \mathrm{d}x  \mathrm{d} \tau
				\\
				&\quad
				= \int_{0}^{t} \int_{\Omega} \triangledown \cdot (u \frac{1}{2} |Z^{\alpha}\chi (\curl B \times n)|^2) - divu \frac{1}{2} |Z^{\alpha}\chi (\curl B \times n)|^2 \, \mathrm{d}x  \mathrm{d} \tau = 0 .
			\end{split}
		\end{equation}
		
		Recall we have
		$$div (Z^{\alpha}\triangledown  \chi (\curl B \times n) \cdot Z^{\alpha}  \chi (\curl B \times n))= [\triangledown \cdot , Z^{\alpha}] \triangledown  \chi (\curl B \times n) \cdot Z^{\alpha}  \chi (\curl B \times n)$$
		$$+ Z^{\alpha} \triangle  \chi (\curl B \times n) \cdot Z^{\alpha}  \chi (\curl B \times n) + \triangledown Z^{\alpha} \chi (\curl B \times n) : Z^{\alpha} \triangledown  \chi (\curl B \times n) $$
		with $\triangledown Z^{\alpha}  \chi (\curl B \times n) : Z^{\alpha}  \triangledown  \chi (\curl B \times n) = \sum_{i=1,2,3} \triangledown Z^{\alpha}  \chi (\curl B \times n)_i \cdot Z^{\alpha} \triangledown  \chi (\curl B \times n)_i$.

		Integrating by parts and since $(Z^{\alpha} \chi (\curl B \times n) )| _{\pl \Omega}=0$ for all $\alpha$, we obtain
		\begin{equation}\label{eqL4-14}
			\begin{split}
				&
				\int_{0}^{t} \int_{\Omega} - \nu (\epsilon) Z^{\alpha} \triangle ( \chi (\curl B \times n)) \cdot Z^{\alpha}( \chi (\curl B \times n)) \, \mathrm{d}x  \mathrm{d} \tau
				\\
				&\quad
				= \nu (\epsilon) \int_{0}^{t} \int_{\Omega} Z^{\alpha} \triangledown ( \chi (\curl B \times n)) : Z^{\alpha}  \triangledown( \chi (\curl B \times n))
				\\
				&\quad\quad
				+  [\triangledown, Z^{\alpha}] ( \chi (\curl B \times n)) : Z^{\alpha} \triangledown ( \chi (\curl B \times n))
				\\
				&\quad\quad
				+ [\triangledown \cdot,Z^{\alpha}] \triangledown ( \chi (\curl B \times n)) \cdot Z^{\alpha}( \chi (\curl B \times n)) \, \mathrm{d}x  \mathrm{d} \tau
				\\
				&\quad
				\geq \frac{3}{4} \nu (\epsilon) \int^t_0 \left\| Z^{\alpha} \triangledown ( \chi (\curl B \times n)) \right\|^2 \mathrm{d} \tau - C_{m} \nu( \epsilon )\int^t_0 \left\| \triangledown ( \chi (\curl B \times n)) \right\|_{m-2}^2 \mathrm{d} \tau
				\\
				&\quad\quad
				+ \nu (\epsilon) \int_{0}^{t} \int_{\Omega} \sum_{|\beta| \leq m-2} \mathcal{C}_{1 \beta} \pl_z Z^{\beta} \pl_z ( \chi (\curl B \times n)) \cdot Z^{\alpha}( \chi (\curl B \times n)) 
				\\
				&\quad\quad\quad
				+ \mathcal{C}_{2 \beta} Z_y Z^{\beta} Z_y ( \chi (\curl B \times n)) \cdot Z^{\alpha}( \chi (\curl B \times n)) \, \mathrm{d} x \mathrm{d} \tau.
			\end{split}
		\end{equation}
	
		Integrating by parts for the last two terms, we then obtain
		\begin{equation}\label{eqL4-15}
			\begin{split}
				&
				\int_{0}^{t} \int_{\Omega} - \nu (\epsilon) Z^{\alpha} \triangle ( \chi (\curl B \times n)) \cdot Z^{\alpha}( \chi (\curl B \times n)) \, \mathrm{d}x  \mathrm{d} \tau
				\\
				&\quad
				\geq \frac{1}{2} \nu( \epsilon) \int^t_0 \left\| Z^{\alpha} \triangledown ( \chi (\curl B \times n)) \right\|^2 \mathrm{d} \tau - C_{m+2} \int^t_0 ( \left\| B \right\|^2_{m} + \left\| \triangledown B \right\|^2_{m-1} \,) \mathrm{d} \tau
				\\
				&\quad\quad
				- C_{m+1} \nu (\epsilon) \int^t_0 \left\| \triangledown ( \chi (\curl B \times n)) \right\|^2_{m-2} \mathrm{d} \tau.
			\end{split}
		\end{equation}
		
		We already estimate $\left\|F\right\|^2_{m-1}$ and $\left\|G\right\|^2_{m-1}$ in the proof of Lemma \ref{lem3}. Combining them with \eqref{eqL4-1}, \eqref{eqL4-2}, \eqref{eqL4-5}, \eqref{eqL4-6}, \eqref{eqL4-8}, \eqref{eqL4-9}, \eqref{eqL4-11}, \eqref{eqL4-12}, \eqref{eqL4-13}, \eqref{eqL4-15}, we get
		\begin{equation}\label{eqL4-16}
			\begin{split}
				&
				\frac{1}{2} \left\| \eta (t) \right\|_{m-1}^2 + \frac{1}{2} \left\| \chi (\curl B \times n) (t) \right\|_{m-1}^2 + \frac{1}{2} \mu (\epsilon) \int_{0}^{t} \left\| \triangledown \eta \right\|_{m-1}^2 \mathrm{d} \tau + \frac{1}{2} \nu (\epsilon) \int_{0}^{t} \left\| \triangledown ( \chi (\curl B \times n) )\right\|_{m-1}^2 \mathrm{d} \tau
				\\
				&\quad
				\leq C_{m+1} \mu (\epsilon) \int^t_0 \left\| \triangledown \eta \right\|^2_{m-2} \, \mathrm{d} \tau + C_{m+1} \nu (\epsilon) \int^t_0 \left\| \triangledown ( \chi (\curl B \times n)) \right\|^2_{m-2} \mathrm{d} \tau
				\\
				&\quad\quad
				+\frac{1}{2} \left\| \eta (0) \right\|_{m-1}^2 + \frac{1}{2} \left\| \chi (\curl B \times n) (0) \right\|_{m-1}^2 + \mu(\epsilon)^2 \delta \int^t_0 \left\| \triangledown^2 u \right\|^2 \mathrm{d} \tau + \nu(\epsilon)^2 \delta \int^t_0 \left\| \triangledown^2 B \right\|^2 \mathrm{d} \tau 
				\\
				&\quad \quad
				+ C_{\delta} C_{m+2}(1 + \left\| (u,B,\triangledown u,\triangledown B, Z \triangledown u, Z \triangledown B) \right\|^2_{L^{\infty}} )\int_{0}^{t} (\left\|(u,B)\right\|^2_{m} + \left\|(\triangledown u, \triangledown B)\right\|_{m-1}^2 + \left\| \triangledown q \right\|_{m-1}^2 \,)\mathrm{d}\tau.
			\end{split}
		\end{equation}
		
		By induction, we then complete the proof.
		
	\end{proof}

		\hspace*{\fill}\\
		\hspace*{\fill}

		\subsection{Pressure estimates}
		
		\begin{lem}\label{lem5}
			For $m \geq 2$, we have the following estimates for $q$:
			\begin{equation}\label{eqL5}
				\begin{split}
					&
					\int^t_0 \left\| \triangledown q_1 \right\|^2_{m-1} + \left\| \triangledown^2 q_1 \right\|^2_{m-1} \,  \mathrm{d} \tau \leq C_{m+2} (1+ P(\left\| (u, B , \triangledown u, \triangledown B) \right\|^2_{L^{\infty}}))
					\\
					&\qquad \qquad\qquad\qquad\qquad\qquad\qquad
					\int^t_0 (\left\|(u, B)\right\|^2_m + \left\| (\triangledown u, \triangledown B)\right\|^2_{m-1} \, ) \mathrm{d} \tau,
					\\
					&
					\int^t_0 \left\| \triangledown q_2 \right\|^2_{m-1} \,  \mathrm{d} \tau \leq C_{m+2} \mu(\epsilon) \int^t_0 ( \left\| u \right\|^2_m + \left\| \triangledown u \right\|^2_{m-1} \, ) \mathrm{d} \tau.
				\end{split}
			\end{equation}
		\end{lem}

		\begin{proof}
		Recall
		\begin{equation}
			\triangle q_1 = - \triangledown \cdot ((u \cdot \triangledown)u) + \triangledown \cdot ((B \cdot \triangledown)B)= -\triangledown u \cdot \triangledown u + \triangledown B \cdot \triangledown B \quad , \quad x \in \Omega,
		\end{equation}
		with $\triangledown u \cdot \triangledown u=\sum_{i,j=1,2,3}(\pl_i u_j \pl_j u_i)$, $\triangledown B \cdot \triangledown B=\sum_{i,j=1,2,3}(\pl_i B_j \pl_j B_i)$ and the boundary condition
		\begin{equation}
			\pl_n q_1 = (- u \cdot \triangledown u+ B \cdot \triangledown B) \cdot n  \quad , \quad x \in \pl \Omega,
		\end{equation}
		
		and
		\begin{equation}
			\triangle q_2 = 0 \quad , \quad x \in \Omega,
		\end{equation}
		with the boundary condition
		\begin{equation}
			\pl_n q_2 = \mu(\epsilon) \triangle u \cdot n  \quad , \quad x \in \pl \Omega,
		\end{equation}
		
		From standard elliptic regularity results with Neumann boundary conditions \cite{M2012}, we get that
		\begin{equation}\label{eqL5-1}
			\begin{split}
				&
				\int^t_0 \left\| \triangledown q_1 \right\|^2_{m-1} + \left\| \triangledown^2 q_1 \right\|^2_{m-1} \, \mathrm{d} \tau \leq C_{m+1} \int^t_0 ( \left\| \triangledown u \cdot \triangledown u \right\|^2_{m-1} + \left\| \triangledown B \cdot \triangledown B \right\|^2_{m-1} 
				\\
				&\qquad \qquad\qquad\qquad
				+ \left\| u \cdot \triangledown u \right\|^2 + \left\| B \cdot \triangledown B \right\|^2 + | (u \cdot \triangledown u) \cdot n |^2_{H^{m-\frac{1}{2}}(\pl \Omega)} + | (B \cdot \triangledown B) \cdot n |^2_{H^{m-\frac{1}{2}}(\pl \Omega)} \,) \mathrm{d} \tau.
			\end{split}
		\end{equation}
		
		Since $u \cdot n=0$ on the boundary, we note that
		\begin{equation}\label{eqL5-2}
			(u \cdot \triangledown u) \cdot n = - (u \cdot \triangledown n) \cdot u \, , \qquad x \in \pl \Omega,
		\end{equation}
		and hence by Proposition \ref{prop3} (trace theorem)
		\begin{equation}\label{eqL5-3}
			\int^t_0 | (u \cdot \triangledown u) \cdot n |^2_{H^{m-\frac{1}{2}}(\pl \Omega)} \mathrm{d} \tau \leq C_{m+2} \int^t_0 ( \left\| \triangledown u \cdot \triangledown u \right\|^2_{m-1} + \left\| u \cdot \triangledown u \right\|^2 + \left\| u \otimes u \right\|^2_{m} \,)\mathrm{d} \tau.
		\end{equation}
		
		Then by Proposition \ref{prop1}, we obtain
		\begin{equation}\label{eqL5-4}
			\int^t_0 | (u \cdot \triangledown u) \cdot n |^2_{H^{m-\frac{1}{2}}(\pl \Omega)} \mathrm{d} \tau \leq C_{m+2} (1+ \left\| (u, \triangledown u) \right\|^2_{L^{\infty}}) \int^t_0 ( \left\| u \right\|^2_{m} + \left\|\triangledown u \right\|^2_{m-1} \,) \mathrm{d} \tau.
		\end{equation}
		
		Similarly, since we also have $B \cdot n=0$ on the boundary, we use the same arguments as \eqref{eqL5-2}
		\begin{equation}\label{eqL5-5}
			\int^t_0 | (B \cdot \triangledown B) \cdot n |^2_{H^{m-\frac{1}{2}}(\pl \Omega)} \mathrm{d} \tau \leq C_{m+2} (1+ \left\| (B, \triangledown B) \right\|^2_{L^{\infty}}) \int^t_0 ( \left\| B \right\|^2_{m} + \left\|\triangledown B \right\|^2_{m-1} \,) \mathrm{d} \tau.
		\end{equation}
		
		Therefore, we have
		\begin{equation}\label{eqL5-6}
			\begin{split}
				&
				\int^t_0 (\left\| \triangledown q_1 \right\|^2_{m-1} + \left\| \triangledown^2 q_1 \right\|^2_{m-1} \,) \mathrm{d} \tau \leq C_{m+2} (1+ \left\| (u,B,\triangledown u, \triangledown B) \right\|^2_{L^{\infty}}) 
				\int^t_0 ( \left\| (u,B) \right\|^2_{m} + \left\|(\triangledown u ,\triangledown B) \right\|^2_{m-1} \,) \mathrm{d} \tau.
			\end{split}
		\end{equation}
		
		It remains to estimate $q_2$. By using the regularity for the Neumann problem \cite{M2012}, we get that for $m \geq 2$,
		\begin{equation}\label{eqL5-7}
			\int^t_0 \left\| \triangledown q_2 \right\|^2_{m-1} \mathrm{d} \tau \leq \mu(\epsilon) C_{m} \int^t_0 | \triangle u \cdot n |^2_{H^{m-\frac{3}{2}}(\pl \Omega)}.
		\end{equation}
		
		To estimate the right-hand side, we shall use the Navier boundary condition $\eqref{eqL1-2}_1$. Since
		\begin{equation}\label{eqL5-8}
			2 \triangle u \cdot n = \triangledown \cdot (Su \, n) - \sum_{j} ( Su \pl_j n)_j,
		\end{equation}
		we first get that
		\begin{equation}\label{eqL5-9}
			| \triangle u \cdot n |^2_{H^{m-\frac{3}{2}}(\pl \Omega)} \lesssim | \triangledown \cdot (Su \, n)|^2_{H^{m-\frac{3}{2}}(\pl \Omega)} + C_{m+1} | \triangledown u|^2_{H^{m-\frac{3}{2}}(\pl \Omega)},
		\end{equation}
		and, hence, since we have $\eqref{n1}_1$ and $\eqref{n2}_1$
		\begin{equation}\label{eqL5-10}
			| \triangledown u|^2_{H^{m-\frac{3}{2}}(\pl \Omega)} \lesssim | \Pi(u \cdot \triangledown n)- \Pi(\alpha u)|^2_{H^{m-\frac{3}{2}}(\pl \Omega)} + | \sum_{i=1,2} (\Pi\pl_{y^i} u)^i |^2_{H^{m-\frac{3}{2}}(\pl \Omega)} + |u|^2_{H^{m-\frac{1}{2}}(\pl \Omega)},
		\end{equation}
		finally, we get
		\begin{equation}\label{eqL5-11}
			| \triangle u \cdot n |^2_{H^{m-\frac{3}{2}}(\pl \Omega)} \lesssim | \triangledown \cdot (Su \, n)|^2_{H^{m-\frac{3}{2}}(\pl \Omega)} + C_{m+1} | u|^2_{H^{m-\frac{1}{2}}(\pl \Omega)}.
		\end{equation}
		
		We can rewrite
		\begin{equation}\label{eqL5-12}
			\triangledown \cdot (Su \, n) = \pl_n (Su \, n) \cdot n + ( \Pi \pl_{y^1} (Su \, n) )^1 + ( \Pi \pl_{y^2} (Su \, n) )^2. 
		\end{equation}
		
		Therefore, we have
		\begin{equation}\label{eqL5-13}
			\begin{split}
				&
				| \triangledown \cdot (Su \, n)|^2_{H^{m-\frac{3}{2}}(\pl \Omega)} \lesssim | \pl_n (Su \, n) \cdot n|^2_{H^{m-\frac{3}{2}}(\pl \Omega)} 
				+ C_{m+1}( | \Pi (Su \, n)|^2_{H^{m-\frac{1}{2}}(\pl \Omega)} + | \triangledown u |^2_{H^{m-\frac{3}{2}}(\pl \Omega)})
				\\
				&\quad
				\lesssim | \pl_n (Su \, n) \cdot n|^2_{H^{m-\frac{3}{2}}(\pl \Omega)}+ C_{m+1} | u |^2_{H^{m-\frac{1}{2}}(\pl \Omega)},
			\end{split}
		\end{equation}
		where we consider the boundary condition $\eqref{eqL1-2}_1$ such that $\Pi(Su n)+ \Pi(\alpha u)=0$ on the boundary and \eqref{eqL5-10}.
		
		Finally, we need to estimate the first term on the right-hand side in the above inequality by \eqref{eqL5-10}
		\begin{equation}\label{eqL5-14}
			\begin{split}
				&
				| \pl_n (Su \, n) \cdot n|^2_{H^{m-\frac{3}{2}}(\pl \Omega)} \lesssim  | \pl_n (\pl_n u \cdot n)|^2_{H^{m-\frac{3}{2}}(\pl \Omega)} + C_{m+1}| \triangledown u |^2_{H^{m-\frac{3}{2}}(\pl \Omega)}
				\\
				&\quad
				\lesssim | \pl_n (\pl_n u \cdot n)|^2_{H^{m-\frac{3}{2}}(\pl \Omega)} + C_{m+1}| u |^2_{H^{m-\frac{1}{2}}(\pl \Omega)}
				\\
				&\quad
				\lesssim | \Pi \pl_n u |^2_{H^{m-\frac{1}{2}}(\pl \Omega)} + C_{m+1}| u |^2_{H^{m-\frac{1}{2}}(\pl \Omega)},
			\end{split}
		\end{equation}
		the second inequality used $\pl_n u \cdot n= - (\Pi \pl_{y^1} u )^1 - (\Pi \pl_{y^2} u )^2 $ as $div u =0$.
		
		By the boundary condition $\eqref{eqL1-2}_1$ such that $\Pi \pl_n u = \Pi( u \cdot \triangledown n) - \Pi(\alpha u)$, one has
		\begin{equation}\label{eqL5-15}
			| \Pi \pl_n u |^2_{H^{m-\frac{1}{2}}(\pl \Omega)} \lesssim C_{m+2} | u |^2_{H^{m-\frac{1}{2}}(\pl \Omega)}. 
		\end{equation}
		
		Thus, we obtain
		\begin{equation}\label{eqL5-16}
			\int^t_0 \left\| \triangledown q_2 \right\|^2_{m-1} \, \mathrm{d} \tau \leq C_{m+2} \mu(\epsilon) \int^t_0 ( \left\| u \right\|^2_m + \left\| \triangledown u \right\|^2_{m-1} \,) \mathrm{d} \tau,
		\end{equation}
		which ends the proof of Lemma \ref{lem5}.
		
	\end{proof}

		\hspace*{\fill}\\
		\hspace*{\fill}

		Combine the above lemmas, we still need to consider the terms $\mu(\epsilon) \int_{0}^{t} \left\| \triangledown^2 u \right\|^2_{m-1} \mathrm{d} \tau$ and 
		\\$\nu(\epsilon) \int_{0}^{t} \left\| \triangledown^2 B \right\|^2_{m-1} \mathrm{d} \tau$, which do not have estimates for now. We apply Proposition \ref{prop4} to get for $|\alpha| \leq m-1$
		\begin{equation}\label{tri2-u-1}
			\begin{split}
				&
				\mu (\epsilon) \int^t_0 \left\| \triangledown^2 u \right\|^2_{m-1} \,  \mathrm{d} \tau = \mu (\epsilon) \int^{t}_{0} \left\| \triangledown^2 Z^{\alpha} u + [Z^{\alpha}, \triangledown^2] u \right\| \mathrm{d} \tau
				\\
				&\quad
				\leq \mu(\epsilon) \int^{t}_{0} ( \left\| Z^{\alpha} u \right\|^2_{H^2} + C_{m+1} (\left\| \triangledown^2 u \right\|^2_{m-2} + \left\|\triangledown u \right\|^2_{m-2})) \mathrm{d} \tau
				\\
				&\quad
				\leq C_{m+1} \mu(\epsilon) \int^{t}_{0} ( \left\| \curl Z^{\alpha} u \right\|_{H^1}^2 + \left\| div Z^{\alpha} u \right\|_{H^1}^2 + \left\|Z^{\alpha} u \right\|^2_{H^1} + |Z^{\alpha} u \cdot n |_{H^{\frac{3}{2}}(\pl \Omega)} + \left\| \triangledown^2 u \right\|^2_{m-2}+ \left\| \triangledown u \right\|^2_{m-2} \, ) \mathrm{d} \tau
				\\
				&\quad
				\leq C_{m+1} \mu (\epsilon) \int^t_0 ( \left\|\triangledown^2 u \right\|^2_{m-2} + \left\| \triangledown u \right\|^2_{m} + \left\| \triangledown \eta \right\|^2_{m-1} + \left\| u \right\|^2_{m} + |Z^{\alpha} u \cdot n |^2_{H^{\frac{3}{2}}(\pl \Omega)} \, ) \mathrm{d} \tau,
			\end{split}
		\end{equation}
	    where we use $div u =0$. For $\left\|\triangledown^2 u \right\|^2_{m-2}$, we write it as $\triangledown^2 Z^{\beta} u + [Z^{\beta} , \triangledown^2]u$ for $|\beta|\leq m-2$ and then consider the commutator as we did above
	    \begin{equation}\label{tri2-u-3}
	    	\begin{split}
	    		&
	    		\mu (\epsilon) \int^t_0 \left\| \triangledown^2 u \right\|^2_{m-2} \, \mathrm{d} \tau 
	    		\leq C_{m} \mu (\epsilon) \int^t_0 ( \left\|\triangledown^2 u \right\|^2_{m-3} + \left\| \triangledown u \right\|^2_{m-1} + \left\| \triangledown \eta \right\|^2_{m-2} + \left\| u \right\|^2_{m-1} + |Z^{\beta} u \cdot n |^2_{H^{\frac{3}{2}}(\pl \Omega)} \,) \mathrm{d} \tau.
	    	\end{split}
	    \end{equation}
        
        Repeat until $\left\|\triangledown^2 u \right\|^2\leq \left\|u\right\|^2_{H^2} \leq C_3 (\left\|\curl u \right\|^2_{H^1} + \left\|div u \right\|^2_{H^1}+\left\|u \right\|^2+|u\cdot n |^2_{H^{\frac{3}{2}}(\pl \Omega)})$ with $div u=0$ and $(u \cdot n)|_{
        \pl \Omega}=0$, we will obtain
        \begin{equation}\label{tri2-u-4}
        	\begin{split}
        		&
        		\mu (\epsilon) \int^t_0 \left\| \triangledown^2 u \right\|^2_{m-1} \mathrm{d} \tau
        		\leq C_{m+1} \mu (\epsilon) \int^t_0 (\left\| \triangledown u \right\|^2_{m} + \left\| \triangledown \eta \right\|^2_{m-1} + \left\| u \right\|^2_{m} + |Z^{\alpha} u \cdot n |^2_{H^{\frac{3}{2}}(\pl \Omega)} \,) \mathrm{d} \tau
        	\end{split}
        \end{equation}
	    
	    For the boundary term, we apply boundary condition \eqref{eq1-2} that is $Z^{\alpha}(u \cdot n)=0$ for all $\alpha$ on the boundary and applying Proposition \ref{prop3} (trace theorem)
	    \begin{equation}\label{tri2-u-2}
	    	\begin{split}
	    		&
	    		\int^t_0 |Z^{\alpha} u \cdot n |_{H^{\frac{3}{2}}(\pl \Omega)} \mathrm{d} \tau \leq C_{m+2} \int_{0}^{t} \sum_{|\beta| \leq m-2} |Z^{\beta} u |^2_{H^{\frac{3}{2}}} \mathrm{d} \tau
	    		\\
	    		&\quad
	    		\leq C_{m+2} \int_{0}^{t} \sum_{|\beta| \leq m-2}(\left\| \triangledown Z^{\beta} u \right\|_1+\left\| Z^{\beta} u \right\|_1)\left\| Z^{\beta} u \right\|_2 \, \mathrm{d} \tau
	    		\leq C_{m+2} \int^{t}_{0} ( \left\| \triangledown u \right\|^2_{m-1} + \left\| u \right\|_m \, ) \mathrm{d} \tau.
	    	\end{split}
	    \end{equation}
        
        Then, one obtains
        \begin{equation}\label{tri2-u}
        	\mu (\epsilon) \int^t_0 \left\| \triangledown^2 u \right\|^2_{m-1} \mathrm{d} \tau \leq C_{m+2} \mu (\epsilon) \int^t_0 (\left\| \triangledown u \right\|^2_{m} + \left\| \triangledown \eta \right\|^2_{m-1} + \left\| u \right\|^2_{m}\,) \mathrm{d} \tau.
        \end{equation}
	    
		Similarly,
		\begin{equation}\label{tri2-B}
			\nu (\epsilon) \int^t_0 \left\| \triangledown^2 B \right\|^2_{m-1} \mathrm{d} \tau \leq C_{m+2} \nu (\epsilon) \int^t_0 ( \left\| \triangledown B \right\|^2_{m} + \left\| \triangledown (\chi (\curl B \times n) \right\|^2_{m-1} + \left\| B \right\|^2_{m} \,) \mathrm{d} \tau,
		\end{equation}
		since $div u=0$ and $div B =0$, all the terms on the right-hand side have estimates as we already proved above.
		
		For now, we already have
		\begin{equation}\label{allbeforeLinfinity}
			\begin{split}
				&
				\left\|(u,B)\right\|^2_m + \left\|(\triangledown u, \triangledown B) \right\|^2_{m-1} + \epsilon \int^{t}_{0} ( \left\| (\triangledown u, \triangledown B) \right\|^2_{m} + \left\|(\triangledown \eta, \triangledown (\chi (\curl B \times n))) \right\|^2_{m-1} \, ) \mathrm{d} \tau
				\\
				&\quad
				\leq C_{m+2}( P( N_m(0)) + P(\left\|(u,B, Zu, ZB, \triangledown u, \triangledown B, Z \triangledown u, Z \triangledown B) \right\|^2_{L^{\infty}} )\int_{0}^{t}P( N_m(t) )\mathrm{d} \tau ),
			\end{split}
		\end{equation}
		thus we next need to estimate $L^{\infty}$ terms.

		\hspace*{\fill}\\
		\hspace*{\fill}

		\subsection{$L^{\infty}$ estimates}
		
		In order to close the estimates, we remain to prove the bounds for the $L^{\infty}$ norms.
		
		Recall $M_m(t)=N_m(t)+M(t)$ with
		\begin{equation}
			N_m(t) = \sup_{0 \leq \tau \leq t} \left\{ 1+ \left\| (u,B)(\tau) \right\|^2_{m}+ \left\| (\triangledown u, \triangledown B) (\tau) \right\|^2_{m-1} + \left\| (\triangledown u, \triangledown B)(\tau) \right\|^2_{{1,\infty}} \right\},
		\end{equation}
		and
		\begin{equation}
			M(t)= \sup_{0 \leq \tau \leq t} \left\{ \left\| (\pl_t u, \pl_t B) (\tau) \right\|^2_{4} + \left\| (\triangledown \pl_t u, \triangledown \pl_t B) (\tau)\right\|^2_{3}\right\}.
		\end{equation}
		
		\hspace*{\fill}

		\begin{lem}\label{lem6}
		For every $m_0 >1$, it holds that 
			\begin{equation}\label{eqL6}
				\left\| (u, B) \right\|^2_{2,\infty} \leq C_m( \left\| (\triangledown u, \triangledown B)\right\|^2_{m-1} + \left\| ( u, B )\right\|^2_{m})\leq C_m N_m(t), \, m \geq m_0+3.
			\end{equation}
		\end{lem}

		\begin{proof}
			We apply Proposition \ref{prop2} directly.
		\end{proof}

		\hspace*{\fill}\\
		\hspace*{\fill}

		\begin{lem}\label{lem7}
		For $m \geq 6$, we have the following estimate 
			\begin{equation}\label{eqL7}
				\left\| (\triangledown u, \triangledown B) \right\|^2_{1,\infty} \leq C_{m+2}(P(M_m(0)) + P(M_m(t)) \int^t_0 (C_{\delta} P(M_m(\tau))+ \delta \epsilon \left\| \triangledown^2 \pl_t B \right\|^2_{2} ) \, \mathrm{d} \tau).
			\end{equation}
		\end{lem}

		\begin{proof}
		
		Away from the boundary, we can use the classical isotropic Sobolev embedding \cite{WX2015} such that
		\begin{equation}\label{eqL7-1}
			\left\|( \chi \triangledown u, \chi \triangledown B) \right\|^2_{1,\infty} \lesssim \left\| u \right\|^2_{m}, \, m \geq 4.
		\end{equation}
		
		Therefore, we only have to estimate $\left\| \chi_i \triangledown u \right\|^2_{1,\infty}$ and $\left\| \chi_i \triangledown B \right\|^2_{1,\infty}$, $ i>0$ since we use a partition of unity subordinate to the covering \eqref{eq1-4}. For convenience, we denote $\chi_i$ by $\chi$. In order to take the Laplacian in a convenient form, we use the local parametrization in the neighborhood of the boundary is given by a normal geodesic system \cite{M2012}. Here, we denote
		\begin{equation}\label{eqL7-3}
			\Psi^n(y,z)= \left( \begin{array}{c}
				y\\
				\psi(y)
			\end{array} \right) - z n(y)=x,
		\end{equation}
		where n is the unit outward normal such that
		\begin{equation}\label{eqL7-4}
			n(y)= \frac{1}{\sqrt{1+|\triangledown \psi (y)|^2}}\left( \begin{array}{c}
				\pl_1\psi (y)\\
				\pl_2 \psi(y)\\
				-1
			\end{array} \right).
		\end{equation}
		As before, one can extend $n$ and $\Pi$ in the interior by setting
		\begin{equation}\label{eqL7-5}
			n(\Psi^n(y,z))=n(y), \, \Pi(\Psi^n(y,z))=\Pi(y).
		\end{equation}
		
		Now on the associated local basis $(\pl_{y^1}, \pl_{y^2}, \pl_z)$ of $\mathbb{R}^3$, we have $\pl_z=\pl_n$ and
		\begin{equation}\label{eqL7-6}
			(\pl_{y^i})|_{\Psi^n(y,z)} \cdot (\pl_z)|_{\Psi^n(y,z)}=0.
		\end{equation}
		
		The scalar product on $\mathbb{R}^3$ induces in the coordinate system the Riemannian metric $g$ with the form
		\begin{equation}\label{eqL7-7}
			g(y,z)= \left( \begin{matrix}
				\Tilde{g}(y,z)  & 0\\
				0 & 1
			\end{matrix} \right).
		\end{equation}
		
		Therefore, we get the Laplacian in the coordinate system as
		\begin{equation}\label{eqL7-8}
			\triangle f = \pl_{zz} f + \frac{1}{2} \pl_z ( \ln |g|) \pl_z f + \triangle_{\Tilde{g}} f,
		\end{equation}
		where $|g|$ denotes the determinant of the matrix $g$, and $\triangle_{\Tilde{g}}$ is defined by
		\begin{equation}\label{eqL7-9}
			\triangle_{\Tilde{g}} f = \frac{1}{\sqrt{| \Tilde{g}|}} \sum_{i,j=1,2} \pl_{y^i} (\Tilde{g}^{ij} | \Tilde{g}|^{\frac{1}{2}} \pl_{y^j} f),
		\end{equation}
		which only involves the tangential derivatives and ${\Tilde{g}^{ij}}$ is the inverse matrix to $\widetilde{g}$.
		
		Next, from \eqref{n1} and Lemma 6, we first get
		\begin{equation}\label{eqL7-10}
			\begin{split}
				&
				\left\| \chi \triangledown u \right\|^2_{1,\infty} \leq C_3( \left\| \chi \Pi \pl_n u \right\|^2_{1,\infty} + \left\| u \right\|^2_{2,\infty})
				\leq C_3( \left\| \chi \Pi \pl_n u \right\|^2_{1,\infty} + \left\| u \right\|^2_{m} + \left\| \triangledown u \right\|^2_{m-1} ),
			\end{split}
		\end{equation}
		and
		\begin{equation}\label{eqL7-11}
			\begin{split}
				&
				\left\| \chi \triangledown B \right\|^2_{1,\infty} \leq C_3( \left\| \chi \Pi \pl_n B \right\|^2_{1,\infty} + \left\| B \right\|^2_{2,\infty})
				\leq C_3( \left\| \chi \Pi \pl_n B \right\|^2_{1,\infty} + \left\| B \right\|^2_{m} + \left\| \triangledown B \right\|^2_{m-1} ).
			\end{split}
		\end{equation}
		
		Recall
		\begin{equation}
			\Pi (\omega \times n)= \Pi( \pl_n u - \triangledown (u \cdot n) + \triangledown n^t \cdot u),
		\end{equation}
		and
		\begin{equation}
			\Pi (\curl B \times n)= \Pi( \pl_n B - \triangledown (B \cdot n) + \triangledown n^t \cdot B).
		\end{equation}
		
		Therefore, we find that
		\begin{equation}
			\left\| \chi \Pi \pl_n u \right\|^2_{1,\infty} \leq C_3( \left\| \chi \Pi (\omega \times n) \right\|^2_{1,\infty} + \left\| u \right\|^2_{2,\infty}),
		\end{equation}
		and
		\begin{equation}
			\left\| \chi \Pi \pl_n B \right\|^2_{1,\infty} \leq C_3( \left\| \chi \Pi (\curl B \times n) \right\|^2_{1,\infty} + \left\| B \right\|^2_{2,\infty}),
		\end{equation}
		
		Thus, we only need to estimate $\left\| \chi \Pi \pl_n u \right\|^2_{1,\infty}$ and $\left\| \chi \Pi \pl_n B \right\|^2_{1,\infty}$ in order to close a priori estimates.
		
		In the support of $\chi$, set
		\begin{equation}
			\Tilde{\omega}(y,z)=\omega(\Psi^n(y,z)), \, \widetilde{\curl B} = \curl B(\Psi^n(y,z)), \, (\Tilde{u}, \Tilde{B})(y,z)=(u, B)(\Psi^n(y,z)).
		\end{equation}
		
		It follows from \eqref{eq1-1} and \eqref{eqL3-2} that
		\begin{equation}\label{eqL7-12}
			\begin{split}
				&
				\pl_t \Tilde{u} +  \Tilde{u}^1 \pl_{y^1} \Tilde{u} +  \Tilde{u}^2 \pl_{y^2} \Tilde{u}+  \Tilde{u} \cdot n \pl_{z} \Tilde{u} -\Tilde{B}^1 \pl_{y^1} \Tilde{B} -  \Tilde{B}^2 \pl_{y^2} \Tilde{B} - \Tilde{B} \cdot n \pl_{z} \Tilde{B}
				\\
				&\quad
				= \mu (\epsilon) (\pl_{zz} \Tilde{u} + \frac{1}{2} \pl_z (\ln |g|)\pl_z \Tilde{u} + \triangle_{\Tilde{g}} \Tilde{u}) - (\triangledown q) \circ \Psi^n,
			\end{split}
		\end{equation}
		and
		\begin{equation}\label{eqL7-13}
			\begin{split}
				&
				\pl_t \Tilde{\omega} +  \Tilde{u}^1 \pl_{y^1} \Tilde{\omega} +  \Tilde{u}^2 \pl_{y^2} \Tilde{\omega}+  \Tilde{u} \cdot n \pl_z \Tilde{\omega} -\Tilde{B}^1 \pl_{y^1} \widetilde{\curl B} -  \Tilde{B}^2 \pl_{y^2} \widetilde{\curl B} - \Tilde{B} \cdot n \pl_z \widetilde{\curl B}
				\\
				&\quad
				= \mu (\epsilon) (\pl_{zz} \Tilde{\omega} + \frac{1}{2} \pl_z (\ln |g|)\pl_z \Tilde{\omega} + \triangle_{\Tilde{g}} \Tilde{\omega}) + \Tilde{F^1},
			\end{split}
		\end{equation}
		where
		\begin{equation}\label{eqL7-14}
			\Tilde{F^1}= ((\omega \cdot \triangledown)u) \circ \Psi^n + ((\curl B \cdot \triangledown)B) \circ \Psi^n .
		\end{equation}
		
		We define
		\begin{equation}\label{eqL7-15}
			\Tilde{\eta}(y,z)= \chi (\Tilde{\omega} \times n + \Pi (K \Tilde{u})),
		\end{equation}
		where $K=2(\alpha I - S(n))$ with $I$ is the identity matrix, and thus we have $\Tilde{\eta}(y,0)=0$ by the boundary condition \eqref{eq1-13}.
		
		Due to \eqref{eqL7-12} and \eqref{eqL7-13}, $\Tilde{\eta}$ solves the equations
		\begin{equation}\label{eqL7-16}
			\begin{split}
				&
				\pl_t \Tilde{\eta} +  \Tilde{u}^1 \pl_{y^1} \Tilde{\eta} +  \Tilde{u}^2 \pl_{y^2} \Tilde{\eta}+  \Tilde{u} \cdot n \pl_z \Tilde{\eta} 
				\\
				&\quad
				-\Tilde{B}^1 \pl_{y^1} (\chi (\widetilde{\curl B} \times n)) -  \Tilde{B}^2 \pl_{y^2} (\chi (\widetilde{\curl B} \times n)) - \Tilde{B} \cdot n \pl_{z} (\chi (\widetilde{\curl B} \times n)) - \mu \epsilon \pl_{zz} \Tilde{\eta}
				\\
				&\quad
				= \mu (\epsilon) \frac{1}{2} \pl_z (\ln |g|)\pl_z \Tilde{\eta} + \chi \Pi \Tilde{H_1} \times n + \chi \Pi (K \Tilde{H_2}) + H_3+ \chi H_4,
			\end{split}
		\end{equation}
		with
		\begin{equation}\label{H1}
			H_1= (\omega \cdot \triangledown)u - (\curl B \cdot \triangledown)B, \quad \Tilde{H_1}= H_1 \circ \Psi^n
		\end{equation}
		\begin{equation}\label{H2}
			H_2= (B \cdot \triangledown) B - \triangledown q, \quad \Tilde{H_2}= H_2 \circ \Psi^n
		\end{equation}
		\begin{equation}\label{H3}
			\begin{split}
				&
				H_3= (\Tilde{u} \cdot \triangledown)\chi \cdot (\Tilde{\omega} \times n + \Pi(K \Tilde{u})) - (\Tilde{B} \cdot \triangledown)\chi \cdot (\widetilde{\curl B} \times n)
				\\
				&\quad
				-\mu (\epsilon)(\pl_{zz} \chi + 2 \pl_z \chi \pl_z + \frac{1}{2} \pl_z(\ln |g|) \pl_z \chi)(\Tilde{\omega} \times n + \Pi(K \Tilde{u})),
			\end{split}
		\end{equation}
		and
		\begin{equation}\label{H4}
			\begin{split}
				&
				H_4= (\Tilde{u}^1 \pl_{y^1} \Pi+ \Tilde{u}^2 \pl_{y^2} \Pi) \cdot (K \Tilde{u}) + \Tilde{\omega} \times (\Tilde{u}^1 \pl_{y^1} n + \Tilde{u}^2 \pl_{y^2} n)
				\\
				&\quad
				- \widetilde{\curl B} \times (\Tilde{B}^1 \pl_{y^1} n + \Tilde{B}^2 \pl_{y^2} n)
				+\Pi( \Tilde{u} \cdot \triangledown )K \Tilde{u} + \mu (\epsilon) \triangle_{\Tilde{g}} \Tilde{\omega} \times n + \mu(\epsilon) \triangle_{\Tilde{g}} \Pi (K \Tilde{u}).
			\end{split}
		\end{equation}
		
		Recall \eqref{eqL3-9}
		\begin{equation}
			\pl_t \curl B - (B \cdot \triangledown) \omega + (u \cdot \triangledown)\curl B - ( \omega \cdot \triangledown)B + (\curl B \cdot \triangledown)u= \nu (\epsilon) \triangle \curl B.
		\end{equation}
		
		Thus Similar to \eqref{eqL7-16}, we have
		\begin{equation}\label{eqL7-17}
			\begin{split}
				&
				\pl_t \chi( \widetilde{\curl B} \times n) +  \Tilde{u}^1 \pl_{y^1} \chi( \widetilde{\curl B} \times n) +  \Tilde{u}^2 \pl_{y^2} \chi( \widetilde{\curl B} \times n)+  \Tilde{u} \cdot n \pl_z \chi( \widetilde{\curl B} \times n) 
				\\
				&\quad
				-\Tilde{B}^1 \pl_{y^1} \Tilde{\eta} -  \Tilde{B}^2 \pl_{y^2} \Tilde{\eta} - \Tilde{B} \cdot n \pl_z \Tilde{\eta} - \nu (\epsilon) \pl_{zz} \chi( \widetilde{\curl B} \times n)
				\\
				&\quad
				= \nu (\epsilon) \frac{1}{2} \pl_z (\ln |g|)\pl_z \chi( \widetilde{\curl B} \times n) + \chi \Pi \Tilde{H_5} \times n + \chi \Pi (K \Tilde{H_6}) + H_7+ \chi H_8,
			\end{split}
		\end{equation}
		with
		\begin{equation}\label{H5}
			H_5= (\omega \cdot \triangledown)B - (\curl B \cdot \triangledown)u, \quad \Tilde{H_5}= H_5 \circ \Psi^n
		\end{equation}
		\begin{equation}\label{H6}
			H_6= -(B \cdot \triangledown)u, \quad \Tilde{H_6}= H_6 \circ \Psi^n
		\end{equation}
		\begin{equation}\label{H7}
			\begin{split}
				&
				H_7= (\Tilde{u} \cdot \triangledown)\chi \cdot (\widetilde(\curl B) \times n - (\Tilde{B} \cdot \triangledown)\chi \cdot (\Tilde{\omega} \times n + \Pi (K \Tilde{u}))
				\\
				&\quad
				-\nu (\epsilon)(\pl_{zz} \chi + 2 \pl_z \chi \pl_z + \frac{1}{2} \pl_z(\ln |g|) \pl_z \chi) \widetilde{\curl B }\times n,
			\end{split}
		\end{equation}
		and
		\begin{equation}\label{H8}
			\begin{split}
				&
				H_8= -(\Tilde{B}^1 \pl_{y^1} \Pi+ \Tilde{B}^2 \pl_{y^2} \Pi) \cdot (K \Tilde{u}) + \widetilde{\curl B} \times (\Tilde{u}^1 \pl_{y^1} n + \Tilde{u}^2 \pl_{y^2} n)
				\\
				&\quad
				-\Tilde{\omega} \times (\Tilde{B}^1 \pl_{y^1} n + \Tilde{B}^2 \pl_{y^2} n)
				+\Pi( \Tilde{B} \cdot \triangledown ) K \Tilde{u} + \nu (\epsilon) \triangle_{\Tilde{g}} \widetilde{\curl B} \times n.
			\end{split}
		\end{equation}
		
		Define
		\begin{equation}\label{eqL7-18}
			\Tilde{f}= \Tilde{\eta} + \chi \widetilde{\curl B} \times n, \qquad \Tilde{h}= \Tilde{\eta} - \chi \widetilde{\curl B} \times n,
		\end{equation}
		we will immediately get by \eqref{eqL7-16} and \eqref{eqL7-17}
		\begin{equation}\label{eqL7-19}
			\begin{split}
				&
				\pl_t \Tilde{f} +  (\Tilde{u}^1 - \Tilde{B}^1) \pl_{y^1} \Tilde{f} +  (\Tilde{u}^2 - \Tilde{B}^2) \pl_{y^2} \Tilde{f}+  (\Tilde{u} - \Tilde{B}) \cdot n \pl_{z} \Tilde{f} - \mu (\epsilon) \pl_{zz} \Tilde{f} 
				\\
				&\quad
				= \mu (\epsilon) \frac{1}{2} \pl_z (\ln |g|)\pl_z\Tilde{f}+ \chi \Pi (\Tilde{H_1}+\Tilde{H_5}) \times n + \chi \Pi (K (\Tilde{H_2}+ \Tilde{H_6})) + (H_3+H_7)+ \chi(H_4+ H_8)
				\\
				&\quad\quad
				+ (\nu(\epsilon) - \mu(\epsilon)) ( \pl_{zz} (\chi \widetilde{\curl B} \times n) + \frac{1}{2} \pl_z (\ln |g|)\pl_z (\chi \widetilde{\curl B} \times n)),
			\end{split}
		\end{equation}
		and
		\begin{equation}\label{eqL7-20}
			\begin{split}
				&
				\pl_t \Tilde{h} +  (\Tilde{u}^1 + \Tilde{B}^1) \pl_{y^1} \Tilde{h} +  (\Tilde{u}^2 + \Tilde{B}^2) \pl_{y^2} \Tilde{h}+  (\Tilde{u} + \Tilde{B}) \cdot n \pl_{z} \Tilde{h} - \mu( \epsilon )\pl_{zz} \Tilde{h}
				\\
				&\quad
				= \mu (\epsilon) \frac{1}{2} \pl_z (\ln |g|)\pl_z\Tilde{h}+ \chi \Pi (\Tilde{H_1}-\Tilde{H_5}) \times n + \chi \Pi (K (\Tilde{H_2}- \Tilde{H_6})) + (H_3-H_7)+ \chi(H_4- H_8)
				\\
				&\quad\quad
				- (\nu(\epsilon) - \mu(\epsilon)) ( \pl_{zz} (\chi \widetilde{\curl B} \times n) + \frac{1}{2} \pl_z (\ln |g|)\pl_z (\chi \widetilde{\curl B} \times n)).
			\end{split}
		\end{equation}
		
		We know that $\triangle_{\Tilde{g}}$ involves only tangential derivatives and the derivatives of $\chi$ are compactly away from the boundary, we have the following estimates for $m \geq 5$
		\begin{equation}\label{eqL7-21}
			\left\| \chi \Pi \Tilde{H_1} \times n \right\|^2_{1,\infty} \leq C_3( \left\|  u \right\|^4_{{2,\infty}} + \left\|  B \right\|^4_{{2,\infty}} + \left\| \triangledown u \right\|^4_{1,\infty} + \left\| \triangledown B \right\|^4_{1,\infty}) \leq C_2 P(N_m(t)) ,
		\end{equation}
		\begin{equation}\label{eqL7-22}
			\left\| \chi \Pi \Tilde{H_5} \times n \right\|^2_{1,\infty} \leq C_3( \left\|  u \right\|^4_{{2,\infty}} + \left\|  B \right\|^4_{{2,\infty}} + \left\| \triangledown u \right\|^4_{1,\infty} + \left\| \triangledown B \right\|^4_{1,\infty}) \leq C_2 P(N_m(t)) ,
		\end{equation}
		\begin{equation}\label{eqL7-23}
			\left\| \chi \Pi K \Tilde{H_2} \times n \right\|^2_{1,\infty} \leq C_3( \left\|  B \right\|^4_{{2,\infty}} + \left\| \triangledown u \right\|^4_{1,\infty} + \left\| \Pi \triangledown q \right\|^2_{1,\infty}) \leq C_3 ( P(N_m(t)) + \left\| \Pi \triangledown q \right\|^2_{1,\infty}) ,
		\end{equation}
		\begin{equation}\label{eqL7-24}
			\left\| \chi \Pi K \Tilde{H_6} \times n \right\|^2_{1,\infty} \leq C_3( \left\|  u \right\|^4_{{2,\infty}} + \left\|  B \right\|^4_{{2,\infty}} + \left\| \triangledown u \right\|^4_{1,\infty}) \leq C_2 P(N_m(t)) ,
		\end{equation}
		\begin{equation}\label{eqL7-25}
			\left\| H_3 \right\|^2_{1,\infty} \leq C_3( \left\|  u \right\|^4_{{2,\infty}} + \left\|  B \right\|^4_{{2,\infty}} + \mu(\epsilon)^2 \left\| \triangledown u \right\|^2_{{2,\infty}}) \leq C_3 (P(N_m(t))+ \mu(\epsilon)^2 \left\| \triangledown u \right\|^2_{{2,\infty}}) ,
		\end{equation}
		\begin{equation}\label{eqL7-26}
			\left\| H_7 \right\|^2_{1,\infty} \leq C_3( \left\|  u \right\|^4_{{2,\infty}} + \left\|  B \right\|^4_{{2,\infty}} + \nu(\epsilon)^2 \left\| \triangledown B \right\|^2_{{2,\infty}}) \leq C_3 (P(N_m(t)) +\nu(\epsilon)^2 \left\| \triangledown B \right\|^2_{{2,\infty}})  ,
		\end{equation}
		\begin{equation}\label{eqL7-27}
			\begin{split}
				&
				\left\| \chi H_4 \right\|^2_{1,\infty} \leq C_4( \left\|  u \right\|^4_{{2,\infty}} + \left\|  B \right\|^4_{{2,\infty}} + \left\| \triangledown u \right\|^4_{1,\infty} + \left\| \triangledown  B \right\|^4_{1,\infty} + \mu(\epsilon)^2 (\left\| u \right\|^2_{{3,\infty}} +\left\| \triangledown u \right\|^2_{{3,\infty}} ) ) 
				\\
				&\quad
				\leq C_4 (P(N_m(t)) +\mu(\epsilon)^2 \left\| \triangledown u \right\|^2_{{3,\infty}} ),
			\end{split}
		\end{equation}
		\begin{equation}\label{eqL7-28}
			\begin{split}
				&
				\left\| \chi H_8 \right\|^2_{1,\infty} \leq C_4( \left\|  u \right\|^4_{{2,\infty}} + \left\|  B \right\|^4_{{2,\infty}} + \left\| \triangledown u \right\|^4_{1,\infty} + \left\| \triangledown  B \right\|^4_{1,\infty} + \nu(\epsilon)^2 (\left\| B \right\|^2_{{3,\infty}} +\left\| \triangledown B \right\|^2_{{3,\infty}} ) ) 
				\\
				&\quad
				\leq C_4 (P(N_m(t)) +\nu(\epsilon)^2 \left\| \triangledown B \right\|^2_{{3,\infty}} ).
			\end{split}
		\end{equation}
		
		Define
		\begin{equation}\label{eqL7-29}
			\begin{split}
				&
				\Tilde{F}_1 = \chi \Pi (\Tilde{H_1}+\Tilde{H_5}) \times n + \chi \Pi (K (\Tilde{H_2}+ \Tilde{H_6})) + (H_3+H_7)+ \chi(H_4+ H_8) ,
				\\
				&
				\Tilde{F}_2 = \chi \Pi (\Tilde{H_1}-\Tilde{H_5}) \times n + \chi \Pi (K (\Tilde{H_2}- \Tilde{H_6})) + (H_3-H_7)+ \chi(H_4- H_8) ,
			\end{split}
		\end{equation}
		then we have
		\begin{equation}\label{eqL7-30}
			\begin{split}
				&
				\left\| (\Tilde{F}_1 ,\Tilde{F}_2)  \right\|^2_{1,\infty} \leq C_4( P(N_m(t)) + \left\| \Pi \triangledown q \right\|^2_{1,\infty} + \mu(\epsilon)^2 \left\| \triangledown u \right\|^2_{{3,\infty}}+ \nu(\epsilon)^2 \left\| \triangledown B \right\|^2_{{3,\infty}} ).
			\end{split}
		\end{equation}
		
		In order to deal with the terms $\pl_z ( \ln |g|)\pl_z \Tilde{f}$ and $\pl_z ( \ln |g|)\pl_z \Tilde{h}$ in \eqref{eqL7-19} and \eqref{eqL7-20} ,we set
		\begin{equation}\label{eqL7-31}
			\Tilde{f} = \frac{1}{|g|^{\frac{1}{4}}} f = \gamma f, \qquad \Tilde{h} = \frac{1}{|g|^{\frac{1}{4}}} h = \gamma h.
		\end{equation}
		
		Note that we have
		\begin{equation}\label{eqL7-32}
			\begin{split}
				&
				\left\| (\Tilde{f}, \Tilde{h})  \right\|^2_{1,\infty} \leq C_3\left\| (f,h) \right\|^2_{1,\infty}, \qquad \left\| (f,h) \right\|^2_{1,\infty} \leq C_3 \left\| (\Tilde{f}, \Tilde{h}) \right\|^2_{1,\infty},
			\end{split}
		\end{equation}
		moreover, $f$ and $h$ solve the following equations respectively
		\begin{equation}\label{eqL7-33}
			\begin{split}
				&
				\pl_t f +  (\Tilde{u}^1 - \Tilde{B}^1) \pl_{y^1} f +  (\Tilde{u}^2 - \Tilde{B}^2) \pl_{y^2} f+  (\Tilde{u} - \Tilde{B}) \cdot n \pl_{z} f - \mu (\epsilon) \pl_{zz} f
				\\
				&\quad
				= \frac{1}{\gamma} ( \Tilde{F}_1 + \mu( \epsilon) (\pl_{zz} \gamma) f + \mu(\epsilon) \frac{1}{2} \pl_z (\ln |g|)(\pl_z \gamma) f - ( ((\Tilde{u} -\Tilde{B}) \cdot \triangledown ) \gamma) f 
				\\
				&\quad\quad
				+ (\nu(\epsilon) - \mu(\epsilon) ) ((\pl_{zz} \gamma) \frac{\chi \widetilde{\curl B} \times n}{\gamma} + \frac{1}{2} \pl_z (\ln |g|)(\pl_z \gamma) \frac{\chi \widetilde{\curl B} \times n}{\gamma}) ) + (\nu(\epsilon) - \mu(\epsilon) ) \pl_{zz}( \frac{\chi \widetilde{\curl B} \times n}{\gamma})= \mathcal{S}_1,
			\end{split}
		\end{equation}
		and
		\begin{equation}\label{eqL7-34}
			\begin{split}
				&
				\pl_t h  +  (\Tilde{u}^1 + \Tilde{B}^1) \pl_{y^1} h +  (\Tilde{u}^2 + \Tilde{B}^2) \pl_{y^2} h +  (\Tilde{u} + \Tilde{B}) \cdot n \pl_{z} h - \mu (\epsilon) \pl_{zz} h
				\\
				&\quad
				= \frac{1}{\gamma} ( \Tilde{F}_2 + \mu (\epsilon) (\pl_{zz} \gamma) h + \mu (\epsilon )\frac{1}{2} \pl_z (\ln |g|)(\pl_z \gamma) h - ( ((\Tilde{u} + \Tilde{B}) \cdot \triangledown ) \gamma )h
				\\
				&\quad\quad
				- (\nu(\epsilon) - \mu(\epsilon) )((\pl_{zz} \gamma) \frac{\chi \widetilde{\curl B} \times n}{\gamma} + \frac{1}{2} \pl_z (\ln |g|)(\pl_z \gamma) \frac{\chi \widetilde{\curl B} \times n}{\gamma}) ) - (\nu(\epsilon) - \mu(\epsilon) ) \pl_{zz} (\frac{\chi \widetilde{\curl B} \times n}{\gamma} )= \mathcal{S}_2.
			\end{split}
		\end{equation}
		
		In order to estimate $f$ and $h$, we use Lemma 14 in \cite{M2012} directly to get\\
		\begin{lem}\label{lem8}
			Consider $\rho$ a smooth solution of
			$$\pl_t \rho  +  u \cdot \triangledown {\color{red}\rho} - \mu (\epsilon) \pl_{zz} \rho = \mathcal{S}, \quad z>0, \quad \rho(t,y,0)=0$$
			for some smooth divergence-free vector field $u$ such that $u\cdot n$ vanishes on the
			boundary. Assume that $\rho$ and $\mathcal{S}$ are compactly supported in $z$. Then, we have the
			estimate:
			\begin{equation}\label{eqL8}
				\begin{split}
					&
					\left\| \rho (t) \right\|^2_{1,\infty} \leq \left\| \rho (0) \right\|^2_{1,\infty} 
					\\
					&\quad
					+  C_3 \int^t_0 (( \left\|u\right\|^2_{{2,\infty}} + \left\|\triangledown u\right\|^2_{1,\infty})( \left\| \rho \right\|^2_{1,\infty} + \left\|\rho \right\|^2_{m_0 + 3} )+ \left\| \mathcal{S} \right\|^2_{1,\infty} \, ) \mathrm{d} \tau, 
				\end{split}
			\end{equation}
			for $m_0>1$.
		\end{lem}
		\begin{proof}
			We omit the proof since one can follow the proof in \cite{M2012} and \cite{WX2015} directly.\end{proof}

		\hspace*{\fill}\\

		Therefore, by using Lemma \ref{lem8}, we get that
		\begin{equation}\label{eqL7-35}
			\begin{split}
				&
				\left\| f (t) \right\|^2_{1,\infty} \leq \left\| f (0) \right\|^2_{1,\infty} 
				\\
				&\quad
				+  C_3 \int^t_0 (( \left\|u\right\|^2_{{2,\infty}} + \left\|B \right\|^2_{{2,\infty}} + \left\|\triangledown u\right\|^2_{1,\infty} + \left\|\triangledown B\right\|^2_{1,\infty})( \left\| f \right\|^2_{1,\infty} + \left\|f \right\|^2_{5} )+ \left\| \mathcal{S}_1 \right\|^2_{1,\infty} \,) \mathrm{d} \tau, 
			\end{split}
		\end{equation}
		and
		\begin{equation}\label{eqL7-36}
			\begin{split}
				&
				\left\| h (t) \right\|^2_{1,\infty} \leq \left\| h (0) \right\|^2_{1,\infty} 
				\\
				&\quad
				+  C_3 \int^t_0 (( \left\|u\right\|^2_{{2,\infty}} + \left\|B \right\|^2_{{2,\infty}} + \left\|\triangledown u\right\|^2_{1,\infty} + \left\|\triangledown B\right\|^2_{1,\infty})( \left\| h \right\|^2_{1,\infty} + \left\|h \right\|^2_{5}) + \left\| \mathcal{S}_2 \right\|^2_{1,\infty} \,) \mathrm{d} \tau.
			\end{split}
		\end{equation}
		
		We have
		\begin{equation}\label{eqL7-35-2}
			\begin{split}
				&
				\left\| \mathcal{S}_1 \right\|^2_{1,\infty} \leq C_3 (\left\| (\widetilde{F}_1,f,u,B) \right\|^2_{1,\infty}
				\\
				&\quad
				+  \left\| \frac{1}{\gamma}((\nu(\epsilon) - \mu(\epsilon) ) ((\pl_{zz} \gamma) \frac{\chi \widetilde{\curl B} \times n}{\gamma} + \frac{1}{2} \pl_z (\ln |g|)(\pl_z \gamma) \frac{\chi \widetilde{\curl B} \times n}{\gamma}) ) + (\nu(\epsilon) - \mu(\epsilon) ) \pl_{zz}( \frac{\chi \widetilde{\curl B} \times n}{\gamma}) \right\|^2_{1,\infty}),
			\end{split}
		\end{equation}
		and
		\begin{equation}\label{eqL7-36-2}
			\begin{split}
				&
				\left\| \mathcal{S}_2 \right\|^2_{1,\infty} \leq C_3 (\left\| (\widetilde{F}_2,h,u,B) \right\|^2_{1,\infty}
				\\
				&\quad
				+  \left\|\frac{1}{\gamma}( -(\nu(\epsilon) - \mu(\epsilon) ) ((\pl_{zz} \gamma) \frac{\chi \widetilde{\curl B} \times n}{\gamma} + \frac{1}{2} \pl_z (\ln |g|)(\pl_z \gamma) \frac{\chi \widetilde{\curl B} \times n}{\gamma})  )- (\nu(\epsilon) - \mu(\epsilon) ) \pl_{zz}( \frac{\chi \widetilde{\curl B} \times n}{\gamma}) \right\|^2_{1,\infty}),
			\end{split}
		\end{equation}
		
		To get $\left\| (\mathcal{S}_1,\mathcal{S}_2) \right\|^2_{1,\infty}$, we still need to estimate the terms with $(\nu - \mu)$.
		
		First, we can easily get
		\begin{equation}\label{nu-mu-1}
			\begin{split}
				&
				\left\| \frac{1}{\gamma} (\nu(\epsilon) - \mu(\epsilon) ) ((\pl_{zz} \gamma) \frac{\chi \widetilde{\curl B} \times n}{\gamma} + \frac{1}{2} \pl_z (\ln |g|) (\pl_z \gamma) \frac{\chi \widetilde{\curl B} \times n}{\gamma}) \right\|^2_{1,\infty}\leq C_3 (\left\| \triangledown B \right\|^2_{1,\infty} + \left\| B \right\|^2_{{2,\infty}} ).
			\end{split}
		\end{equation}
		
		Recall \eqref{eqL7-17}, we have
		\begin{equation}\label{nu-mu-2}
			\begin{split}
				&
				\pl_t \frac{\chi \widetilde{\curl B} \times n}{\gamma} +  \Tilde{u}^1 \pl_{y^1} \frac{\chi \widetilde{\curl B} \times n}{\gamma}+  \Tilde{u}^2 \pl_{y^2} \frac{\chi \widetilde{\curl B} \times n}{\gamma}+  \Tilde{u} \cdot n \pl_z \frac{\chi \widetilde{\curl B} \times n}{\gamma}
				\\
				&\quad
				-\Tilde{B}^1 \pl_{y^1} \frac{\Tilde{\eta}}{\gamma} -  \Tilde{B}^2 \pl_{y^2} \frac{\Tilde{\eta}}{\gamma} - \Tilde{B} \cdot n \pl_z \frac{\Tilde{\eta}}{\gamma} - \nu( \epsilon )\pl_{zz} (\frac{\chi \widetilde{\curl B} \times n}{\gamma})
				\\
				&\quad
				= \frac{1}{\gamma} ( \nu (\epsilon) (\pl_{zz} \gamma) \frac{\chi \widetilde{\curl B} \times n}{\gamma} + \nu (\epsilon) \frac{1}{2} \pl_z (\ln |g|) (\pl_z \gamma) \frac{\chi \widetilde{\curl B} \times n}{\gamma} 
				\\
				&\quad\quad
				- ((\Tilde{u} \cdot \triangledown ) \gamma) \frac{\chi \widetilde{\curl B} \times n}{\gamma} + ((\Tilde{B} \cdot \triangledown ) \gamma )\frac{\Tilde{\eta}}{\gamma} + \chi \Pi \Tilde{H_5} \times n + \chi \Pi (K \Tilde{H_6}) + H_7+ \chi H_8).
			\end{split}
		\end{equation}
		
		Then we obtain for $m \geq 5$
		\begin{equation}\label{nu-mu-3}
			\begin{split}
				&
				\left\| (\nu(\epsilon) - \mu(\epsilon) ) \pl_{zz} (\frac{\chi \widetilde{\curl B} \times n}{\gamma}) \right\|^2_{1,\infty}
				\\
				&\quad
				\leq C_4 |\frac{(\nu(\epsilon) - \mu(\epsilon) )}{\nu(\epsilon)}|^2 ( \left\| \triangledown \pl_t B \right\|^2_{1,\infty} + \left\| \triangledown u \right\|^2_{1,\infty}\left\| \triangledown B \right\|^2_{2,\infty} + \left\| \triangledown B \right\|^2_{1,\infty}\left\| \triangledown u \right\|^2_{2,\infty} 
				\\
				&\qquad\qquad\qquad \quad
				+ P(N_m(t)) + \nu(\epsilon)^2 \left\| \triangledown B \right\|^2_{3,\infty})
				\\
				&\quad
				\leq C_4 ( |\frac{(\nu(\epsilon) - \mu(\epsilon) )}{\nu(\epsilon)}|^2(\nu(\epsilon)^2 \left\| \triangledown B \right\|^2_{3,\infty}+ \left\| \triangledown \pl_t B \right\|^2_{1,\infty}) + \delta |\frac{(\nu(\epsilon) - \mu(\epsilon) )}{\nu(\epsilon)}|^4 \left\| (\triangledown u, \triangledown B )\right\|^4_{2,\infty} + C_{\delta} P(N_m(t))).
			\end{split}
		\end{equation}

		By the definition of $f$ and $h$, we obtain
		\begin{equation}\label{eqL7-37}
			\begin{split}
				&
				\left\| \eta (t) \right\|^2_{1,\infty} + \left\| \chi (\curl B \times n )(t) \right\|^2_{1,\infty} \leq \left\| \eta (0) \right\|^2_{1,\infty} + \left\| \chi (\curl B \times n )(0) \right\|^2_{1,\infty}
				\\
				&\quad
				+  C_4 \int^t_0 ( C_{\delta} P(N_m(t)) + \left\| \Pi \triangledown q \right\|^2_{1,\infty} + \delta \epsilon^2 \left\| (\triangledown^2 u , \triangledown^2 B )\right\|^2_{4} 
				\\
				&\qquad\qquad
				+ \delta |\frac{(\nu(\epsilon) - \mu(\epsilon) )}{\nu(\epsilon)}|^4 \left\| \triangledown^2 B \right\|^2_{2} + \delta^3 |\frac{(\nu(\epsilon) - \mu(\epsilon) )}{\nu(\epsilon)}|^8 \left\| (\triangledown^2 u, \triangledown^2 B )\right\|^4_{2}
				\\
				&\qquad\qquad
				+ \delta|\frac{(\nu(\epsilon) - \mu(\epsilon) )}{\nu(\epsilon)}|^4 \left\| \triangledown^2 \pl_t B \right\|^2_{2} + C_{\delta}\left\| \triangledown \pl_t B \right\|^2_{3}\,) \mathrm{d} \tau,
			\end{split}
		\end{equation}
		with $m \geq 6$.
		
		Since $\Pi \triangledown q$ involves only the tangential derivatives, we obtain
		\begin{equation}\label{eqL7-38}
			\int_{0}^{t} \left\| \Pi \triangledown q \right\|^2_{1, \infty} \mathrm{d} \tau \leq C_m \int_{0}^{t} \left\| \triangledown q \right\|^2_{3} \mathrm{d} \tau \leq C_{m+2}P(N_m(t))\int_{0}^{t} P(N_m(\tau)) \mathrm{d} \tau,
		\end{equation}
		for $m \geq 4$.

		As we define that $|\frac{(\nu(\epsilon) - \mu(\epsilon) )}{\nu(\epsilon)}|\leq C\frac{\epsilon^{\frac{5}{4}}}{\epsilon+n(\epsilon)} \leq C \epsilon^{\frac{1}{4}} $ and we have \eqref{tri2-u}, \eqref{tri2-B}, we then have for $m \geq 6$
		\begin{equation}\label{eqL7-39}
			\begin{split}
				&
				\left\| \eta (t) \right\|^2_{1,\infty} + \left\| \chi (\curl B \times n )(t) \right\|^2_{1,\infty} \leq C_{m+2} (P( M_m(0)) 
				\\
				&\qquad \qquad
				+ P(M_m(t))\int^t_0 ( C_{\delta} P(M_m(\tau)) + \delta \epsilon \left\| \triangledown^2 \pl_t B \right\|^2_{2}\,) \mathrm{d} \tau),
			\end{split}
		\end{equation}
		which ends the proof of Lemma \ref{lem7}.
		
	\end{proof}

		\hspace*{\fill}\\
		\hspace*{\fill}

		\subsection{Estimstes for $\pl_t u$ and $\pl_t B$}
		
		Recall
		\begin{equation}
			M(t)= \sup_{0 \leq \tau \leq t} \left\{ \left\| (\pl_t u, \pl_t B) (\tau) \right\|^2_{4} + \left\| (\triangledown \pl_t u, \triangledown \pl_t B) (\tau)\right\|^2_{3}\right\}.
		\end{equation}
	    Within this section, our objective is to estimate the terms in $M(t)$ to conclude the a priori estimate. Following the procedures outlined in the preceding sections, we commence by getting the conormal energy estimates for $(\pl_t u, \pl_t B)$. Subsequently, we establish the normal derivative estimates for these quantities. Finally, we attain the $L^{\infty}$ estimates, thereby concluding our proof for the estimate pertaining to $M(t)$.\\
	    
	    We need the following Proposition 
	    
	    \begin{prop}\label{prop5}
	    	For $\pl_t u,\pl_t v,u,v \in L^{\infty}(\Omega \times [0.T]) \cap L^2(\Omega \times [0,T])$ and $Z^{\alpha}\pl_t u, Z^{\alpha} \pl_t v, Z^{\alpha}u, Z^{\alpha} v \in L^2(\Omega \times [0,T])$ for $|\alpha| \leq m$ with $m \in \mathbb{N}_{+}$ be an integer. It holds that
	    	\begin{equation}\label{eq3-1}
	    		\begin{split}
	    			&
	    			\int^{t}_{0} \left\|\pl_t (Z^{\beta}u Z^{\gamma}v)(\tau)\right\|^2 \, \mathrm{d} \tau \lesssim \left\| u\right\|^2_{L^{\infty}_{x,t}} \int^{t}_{0}  \left\|(\pl_tv, v)(\tau)\right\|^2_{m} \mathrm{d} \tau+ \left\| v\right\|^2_{L^{\infty}_{x,t}} \int^{t}_{0}  \left\|(\pl_t u, u)(\tau)\right\|^2_{m} \mathrm{d} \tau, \quad |\beta|+|\gamma|=m,
	    		\end{split}
	    	\end{equation}
	    	for all $t \in[0,T]$.
	    \end{prop}
        
        One can refer to Proposition 2.2 in {\color{red}\cite{WX2015}}. Here we need this Proposition since in this section we consider $\pl_t u$ and $\pl_t B$ while Proposition \ref{prop1} with no time derivatives.\\

		\begin{lem}\label{lem9}
			For a smooth solution to \eqref{eq1-1} and \eqref{eq1-2}, it holds that for all $\epsilon \in (0,1]$
			\begin{equation}\label{eqL9}
				\begin{split}
					&
					\frac{1}{2} \left\| \pl_t u (t) \right\|^2 + \frac{1}{2} \left\| \pl_t B (t) \right\|^2 + \frac{\mu(\epsilon)}{2C_2} \int^{t}_{0}  \left\| \triangledown \pl_t u\right\|^2 \, \mathrm{d} \tau + \frac{\nu(\epsilon)}{2C_2} \int^{t}_{0}  \left\| \triangledown \pl_t B\right\|^2 \,  \mathrm{d} \tau
					\\
					&\quad
					\leq \frac{1}{2} \left\| \pl_t u (0) \right\|^2 + \frac{1}{2} \left\| \pl_t B (0) \right\|^2 + C_2 (1+ \left\| (\triangledown u, \triangledown B) \right\|^2_{L^{\infty}})  \int^{t}_{0}  \left\| (\pl_t u, \pl_t B) \right\|^2 \,  \mathrm{d} \tau.
				\end{split}
			\end{equation}
		\end{lem}

		\begin{proof}
		
		Recall \eqref{q} and \eqref{eqL1-7}, we will get the following equations for $\pl_t u$ and $\pl_t B$
		
		\begin{equation}\label{ut}
			\pl_t \pl_t u + \pl_t (u \cdot \triangledown u) + \triangledown \pl_t q  = \pl_t ((B \cdot \triangledown)B) + \mu(\epsilon) \triangle \pl_t u,
		\end{equation}
		and
		\begin{equation}\label{Bt}
            \pl_t \pl_t B + \pl_t((u \cdot \triangledown)B) - \pl_t((B \cdot \triangledown)u) = \nu(\epsilon) \triangle \pl_t B.
		\end{equation}
		
		Multiplying \eqref{ut} with $\pl_t u$ and integrating with respect to $x$ and $t$, one obatins
		\begin{equation}\label{eqL9-1}
			\begin{split}
				&
				\int^t_0 \int_{\Omega} \pl_t \pl_t u \cdot \pl_t u + \pl_t(( u \cdot \triangledown) u) \cdot \pl_t u + \triangledown \pl_t q \cdot \pl_t u \mathrm{d} x \mathrm{d} \tau
				\\
				&\quad
				= \int^t_0 \int_{\Omega} - \mu(\epsilon ) (\triangledown \times \pl_t \omega) \cdot \pl_t u + \pl_t (B \cdot \triangledown B) \cdot \pl_t u \, \mathrm{d}x \mathrm{d} \tau.
			\end{split}
		\end{equation}
		
		Integration by parts and applying Young's inequality, we have
		\begin{equation}\label{eqL9-2}
			\begin{split}
				&
				\int^t_0 \int_{\Omega} \pl_t \pl_t u \cdot \pl_t u \, \mathrm{d} x \mathrm{d} \tau = \frac{1}{2} \left\| \pl_t u (t) \right\|^2 - \frac{1}{2} \left\| \pl_t u (0) \right\|^2,
			\end{split}
		\end{equation}
		\begin{equation}\label{eqL9-3}
			\begin{split}
				&
				\int^t_0 \int_{\Omega} \pl_t(( u \cdot \triangledown) u) \cdot \pl_t u \, \mathrm{d} x \mathrm{d} \tau = \int^t_0 \int_{\Omega} ( \pl_t u \cdot \triangledown) u \cdot \pl_t u + \frac{1}{2} u \cdot \triangledown |\pl_t u|^2 \, \mathrm{d} x \mathrm{d} \tau 
				\\
				&\quad
				\geq - (1 + \left\| \triangledown u \right\|^2_{L^{\infty}}) \int^t_0 \left\| \pl_t u \right\|^2 \, \mathrm{d} \tau,
			\end{split}
		\end{equation}
		\begin{equation}\label{eqL9-4}
			\begin{split}
				&
				\int^t_0 \int_{\Omega} \triangledown \pl_t q \cdot \pl_t u \mathrm{d} x \mathrm{d} \tau = \int^t_0 \int_{\Omega} \triangledown \cdot (\pl_t q \pl_t u) - \pl_t \pl_t divu \, \mathrm{d}x \mathrm{d} \tau =0,
			\end{split}
		\end{equation}
		where we use the boundary condition \eqref{eq1-2} such that $(Z^{\alpha}\pl_t(u \cdot n))|_{\pl \Omega}=0$ for all $\alpha$ and $\pl_t(Z^{\alpha}u \cdot n) = Z^{\alpha} \pl_t u \cdot n$ since we have $n$ and $Z^{\alpha}$ independent of $t$.
		
		We use the vector equalities from \cite{CM1993}:
		$$\triangledown \cdot (\pl_t \omega \times \pl_t u)= \pl_t u \cdot \triangledown \times \pl_t \omega- \pl_t \omega \cdot \triangledown \times \pl_t u,$$
		and
		$$\pl_t \omega \times \pl_t u \cdot n= n \times \pl_t \omega \cdot n.$$
		
		Therefore, since we have $n \times \omega=\Pi (K u)$ with $K=2\alpha I - 2 S(n)$ on the boundary, we do integration by parts and then apply Proposition \ref{prop3} (trace theorem)
		\begin{equation}\label{eqL9-5-3}
			\begin{split}
				&
				\int^t_0 \int_{\Omega} - \mu(\epsilon ) (\triangledown \times \pl_t \omega) \cdot \pl_t u  \mathrm{d} x \mathrm{d} \tau = - \mu(\epsilon )\int^t_0 \left\| \triangledown \times\pl_t u \right\|^2 \, \mathrm{d} \tau - \mu(\epsilon) \int^t_0 \int_{\pl \Omega} n \times \pl_t \omega \cdot \pl_t u \, \mathrm{d} \sigma \mathrm{d} \tau
				\\
				&\quad
				\leq - \mu(\epsilon )\int^t_0 \left\| \triangledown \times\pl_t u \right\|^2 \, \mathrm{d} \tau + C_2 \mu(\epsilon) \int^t_0 | \pl_t u |^2_{L^2(\pl \Omega)} \, \mathrm{d} \tau
				\\
				&\quad
				\leq - \mu(\epsilon )\int^t_0 \left\| \triangledown \times\pl_t u \right\|^2 \, \mathrm{d} \tau + C_2 \mu(\epsilon) \int^t_0 (\left\| \triangledown \pl_t u \right\|+ \left\| \pl_t u \right\|)\left\| \pl_t u \right\| \, \mathrm{d} \tau,
			\end{split}
		\end{equation}
		where we use Proposition \ref{prop4}, $div \pl_t u = \pl_t div u =0$ and the boundary condition \eqref{eq1-2} that $\pl_t u \cdot n = \pl_t (u \cdot n)=0$ on $\pl \Omega$
		\begin{equation}\label{eqL9-5-2}
			\begin{split}
				&
				\int^t_0 \left\| \triangledown \pl_t u \right\|^2 \, \mathrm{d} \tau \leq \int_{0}^{t} \left\| \pl_t u \right\|^2_{H^1} \, \mathrm{d} \tau \leq C_2 \int_{0}^{t} (\left\| \triangledown \times \pl_t u \right\|^2 + \left\| \pl_t u \right\|^2 + | \pl_t u \cdot n |^2_{H^{\frac{1}{2}}(\pl \Omega)} \,) \mathrm{d} \tau,
			\end{split}
		\end{equation}
	    then we finally obtain
	    \begin{equation}\label{eqL9-5}
	    	\begin{split}
	    		&
	    		\int^t_0 \int_{\Omega} - \mu(\epsilon ) (\triangledown \times \pl_t \omega) \cdot \pl_t u  \mathrm{d} x \mathrm{d} \tau
	    		\leq - \frac{1}{2}\mu(\epsilon )\int^t_0 \left\| \triangledown \times\pl_t u \right\|^2 \, \mathrm{d} \tau + C_2 \mu(\epsilon) \int^t_0 \left\| \pl_t u \right\|^2 \, \mathrm{d} \tau.
	    	\end{split}
	    \end{equation}

		Combine the last term in \eqref{eqL9-1} with \eqref{Bt}, one obtains by integration by parts
		\begin{equation}\label{eqL9-6}
			\begin{split}
				&
				\int^t_0 \int_{\Omega} \pl_t (B \cdot \triangledown B) \cdot \pl_t u \, \mathrm{d} \sigma \mathrm{d} \tau = \int^t_0 \int_{\Omega} (\pl_t B \cdot \triangledown B) \cdot \pl_t u + ( B \cdot \triangledown) \pl_t B \cdot \pl_t u\, \mathrm{d} \sigma \mathrm{d} \tau
				\\
				&\quad
				= \int^t_0 \int_{\Omega} \pl_t B \cdot \triangledown B \cdot \pl_t B + B \cdot \triangledown(\pl_t u \cdot \pl_t B) - \pl_t \pl_t B \cdot \pl_t B - \pl_t u \cdot \triangledown B \cdot \pl_t B 
				\\
				&\quad\quad
				- \frac{1}{2} u \cdot \triangledown |\pl_t B |^2 + \pl_t B \cdot \triangledown u \cdot \pl_t B + \nu(\epsilon) \triangle \pl_t B \cdot \pl_t B  \, \mathrm{d} x  \mathrm{d} \tau
				\\
				&\quad
				\leq \frac{1}{2} \left\| \pl_t B (0) \right\|^2 - \frac{1}{2} \left\| \pl_t B (t) \right\|^2 + \int^t_0 \int_{\Omega} \nu(\epsilon) \triangle \pl_t B \cdot \pl_t B  \, \mathrm{d} x  \mathrm{d} \tau + (1+ \left\| (\triangledown u ,\triangledown B)\right\|^2_{L^{\infty}}) \int_0^t \left\| (\pl_t u, \pl_t B) \right\|^2 \, \mathrm{d} \tau,
			\end{split}
		\end{equation}
		where we use $div u =0$, $div B =0$, Young's inequality, and the boundary condition \eqref{eq1-2} such that $u \cdot n=0$ and $B \cdot n=0$, $n \times \curl B=0$ on the boundary. Then since we have $\triangle B =\triangledown div B - \triangledown \times \curl B = - \triangledown \times \curl B$ and the vector equality from \cite{CM1993}
		$$\triangledown \cdot (\pl_t\curl B \times \pl_t B) = \pl_t B \cdot \triangledown \times \pl_t \curl B - \pl_t \curl B \cdot \triangledown \times \pl_t B,$$
		and
		$$\pl_t \curl B \times \pl_t B \cdot n = n \times \pl_t \curl B \cdot \pl_t B,$$
		then we have 
		\begin{equation}\label{eqL9-6-2}
			\begin{split}
				&
				\int^t_0 \int_{\Omega} \nu(\epsilon) \triangle \pl_t B \cdot \pl_t B  \, \mathrm{d} x  \mathrm{d} \tau = - \int^t_0 \int_{\Omega} \nu(\epsilon) \triangledown \times \pl_t \curl B \cdot \pl_t B  \, \mathrm{d} x  \mathrm{d} \tau 
				\\
				&\quad
				= - \nu(\epsilon) \int^t_0 \left\| \triangledown \times \pl_t B \right\|^2 \, \mathrm{d} \tau - \int^t_0 \int_{\pl \Omega} \nu(\epsilon) n \times (\triangledown \times \pl_t B) \cdot \pl_t B  \, \mathrm{d} \sigma  \mathrm{d} \tau
				=- \nu(\epsilon) \int^t_0 \left\| \triangledown \times \pl_t B \right\|^2 \, \mathrm{d} \tau,
			\end{split}
		\end{equation}
	    where we have $n \times (\triangledown \times \pl_t B)=\pl_t (n \times \curl B)=0$ on $\pl \Omega$ by the boundary condition \eqref{eq1-2}.
		
		Combine above estimates, we get
		\begin{equation}\label{eqL9-7}
			\begin{split}
				&
				\frac{1}{2} \left\| \pl_t u (t) \right\|^2 + \frac{1}{2} \left\| \pl_t B (t) \right\|^2 + \frac{\mu(\epsilon)}{2} \int^{t}_{0} \left\| \triangledown \times \pl_t u\right\|^2 \mathrm{d} \tau + \nu(\epsilon) \int^{t}_{0} \left\| \triangledown \times \pl_t B\right\|^2 \mathrm{d} \tau
				\\
				&\quad
				\leq \frac{1}{2} \left\| \pl_t u (0) \right\|^2 + \frac{1}{2} \left\| \pl_t B (0) \right\|^2 + C_2 (1+ \left\| (\triangledown u, \triangledown B) \right\|^2_{L^{\infty}})  \int^{t}_{0} \left\| (\pl_t u, \pl_t B) \right\|^2 \mathrm{d} \tau.
			\end{split}
		\end{equation}

		From Proposition \ref{prop4}, one has
		\begin{equation}\label{eqL9-8}
			\begin{split}
				&
			    C_2 \int^{t}_{0} \left\| \triangledown \times \pl_t u\right\|^2 \mathrm{d} \tau \geq \int^{t}_{0} \left\| \pl_t u\right\|_{H^1}^2 - C_2 \left\|\pl_t div u \right\|^2 -C_2 \left\|\pl_t u \right\|^2- C_2 |\pl_t u \cdot n |^2_{H^{\frac{1}{2}}(\pl \Omega)}\mathrm{d} \tau
			    \\
			    &\quad
			    \geq \int_{0}^{t}( \left\|\triangledown \pl_t u \right\|^2-C_2 \left\|\pl_t u \right\|^2 \,)\mathrm{d} \tau,
			\end{split}
		\end{equation}
	    where we use $\pl_t div u=0$ and $(\pl_t u \cdot n)|_{\pl \Omega}=(\pl_t(u \cdot n))|_{\pl \Omega}=0$.
	    
	    Same for $\left\| \triangledown \times \pl_t B\right\|^2$ and then we end the proof.
	    
	\end{proof}
	    
		\hspace*{\fill}\\
		\hspace*{\fill}

		\begin{lem}\label{lem10}
			For a smooth solution to \eqref{eq1-1} and \eqref{eq1-2}, it holds that for all $\epsilon \in (0,1]$ and nonnegative integer $m \geq 0$
			\begin{equation}\label{eqL10}
				\begin{split}
					&
					\frac{1}{2} \left\| \pl_t u (t) \right\|_m^2 + \frac{1}{2} \left\| \pl_t B (t) \right\|_m^2 + \frac{\mu(\epsilon)}{2} \int^{t}_{0} \left\| \triangledown \pl_t u\right\|_m^2 \, \mathrm{d} \tau + \frac{\nu(\epsilon)}{2} \int^{t}_{0} \left\| \triangledown \pl_t B\right\|_m^2 \, \mathrm{d} \tau
					\\
					&\quad
					\leq C_{m+2}(\frac{1}{2} \left\| \pl_t u (0) \right\|_m^2 + \frac{1}{2} \left\| \pl_t B (0) \right\|_m^2 + \delta  \int_0^t ( \mu(\epsilon)^2 \left\| \triangledown^2 \pl_t u \right\|^2_{m-1} + \nu(\epsilon)^2 \left\| \triangledown^2 \pl_t B \right\|^2_{m-1} \, ) \mathrm{d} \tau
					\\
					&\quad\quad
					+  C_{\delta} (1+ \left\| (u,B,Zu, ZB, \pl_t u, \pl_t B,\triangledown u, \triangledown B) \right\|^2_{L^{\infty}})
					\\
					&\quad\quad\quad
					\int^{t}_{0} ( \left\| (\triangledown u, \triangledown B) \right\|_m^2+\left\| (\pl_t u, \pl_t B) \right\|_m^2+\left\| (\triangledown \pl_t u, \triangledown \pl_t B) \right\|_{m-1}^2 \, ) \mathrm{d} \tau
					\\
					&\quad \quad
					+ \int^{t}_{0} ( \left\| \triangledown^2 \pl_t q_1 \right\|_{m-1}^2+ \frac{1}{\epsilon}\left\| \triangledown \pl_t q_2 \right\|_{m-1}^2 \, ) \mathrm{d} \tau + \delta \int^{t}_{0} \left\| \triangledown \pl_t q_2 \right\|^2_1 \, \mathrm{d} \tau).
				\end{split}
			\end{equation}
		\end{lem}

		\begin{proof}
		The case for $m=0$ has been proved in Lemma \ref{lem9}. Therefore, we assume that \eqref{eqL10} is proved for $k \leq m-1$, we then need to prove that it holds for $k=m\geq 1$. Applying $Z^{\alpha}$ with $|\alpha|=m$ to \eqref{ut}, multiplying it with $Z^{\alpha}\pl_t u$ and integrating with respect to $x$ and $t$, we obtain
		\begin{equation}\label{eqL10-1}
			\begin{split}
				&
				\int^t_0 \int_{\Omega} Z^{\alpha}\pl_t \pl_t u \cdot Z^{\alpha}\pl_t u + Z^{\alpha}\pl_t(( u \cdot \triangledown) u) \cdot Z^{\alpha} \pl_t u + Z^{\alpha}\triangledown \pl_t q \cdot Z^{\alpha}\pl_t u \mathrm{d} x \mathrm{d} \tau
				\\
				&\quad
				= \int^t_0 \int_{\Omega} - \mu(\epsilon ) Z^{\alpha} (\triangledown \times \pl_t \omega) \cdot Z^{\alpha} \pl_t u + Z^{\alpha} \pl_t (B \cdot \triangledown B) \cdot Z^{\alpha} \pl_t u \, \mathrm{d}x \mathrm{d} \tau.
			\end{split}
		\end{equation}
		
		Integrating by parts,
		\begin{equation}\label{eqL10-2}
			\begin{split}
				&
				\int^t_0 \int_{\Omega} Z^{\alpha}\pl_t \pl_t u \cdot Z^{\alpha}\pl_t u \, \mathrm{d} x \mathrm{d} \tau = \frac{1}{2} \left\| \pl_t u (t) \right\|^2_m - \frac{1}{2} \left\| \pl_t u (0) \right\|^2_m,
			\end{split}
		\end{equation}
		one has $div u =0$, the boundary condition \eqref{eq1-2} such that $(u \cdot n)|_{\pl \Omega}=0$, and applies Proposition \ref{prop1}, Proposition \ref{prop5}
		\begin{equation}\label{eqL10-3}
			\begin{split}
				&
				\int^t_0 \int_{\Omega} Z^{\alpha}\pl_t(( u \cdot \triangledown) u) \cdot Z^{\alpha} \pl_t u \, \mathrm{d} x \mathrm{d} \tau
				\\
				&\quad
				= \int^t_0 \int_{\Omega} u \cdot [Z^{\alpha} \pl_t , \triangledown]u \cdot Z^{\alpha} \pl_t u + \frac{1}{2} u \cdot \triangledown | Z^{\alpha} \pl_t u|^2
				\\
				&\quad\quad
				+ Z^{\alpha} (\pl_t u \cdot \triangledown u) \cdot Z^{\alpha} \pl_t u + \sum_{| \beta| = m-1} \mathcal{C}_{\beta, \gamma} Z^{\beta} \pl_t (Zu \cdot \triangledown u) \cdot Z^{\alpha} \pl_t u \, \mathrm{d}x \mathrm{d} \tau
				\\
				&\quad
				\geq - C_{m+1} (1+ \left\|(u,Zu,\pl_t u, \triangledown u)\right\|^2_{L^{\infty}}) \int^t_0 ( \left\| \pl_t u \right\|^2_m + \left\| \triangledown u \right\|^2_m + \left\|\triangledown \pl_t u \right\|^2_{m-1} \,) \mathrm{d} \tau.
			\end{split}
		\end{equation}
		
		For $q=q_1+q_2$, we consider $q_1$ and $q_2$ seperately. We apply Young's inequality for $q_1$ and here we have $m \geq 1$
		\begin{equation}\label{eqL10-4-2}
			\begin{split}
				&
				\int^t_0 \int_{\Omega} Z^{\alpha}\triangledown \pl_t q_1 \cdot Z^{\alpha}\pl_t u \mathrm{d} x \mathrm{d} \tau \geq - C_{m+1} \int^t_0 (\left\| \triangledown^2 \pl_t q_1 \right\|^2_{m-1} + \left\| \pl_t u \right\|^2_m \, )\mathrm{d} \tau.
			\end{split}
		\end{equation}
		As for $q_2$, we first consider $m=1$ by using Young's inequality
		\begin{equation}\label{eqL10-4-3}
			\begin{split}
				&
				\int^0_t \int_{\Omega} Z^{\alpha}\triangledown \pl_t q_2 \cdot Z^{\alpha}\pl_t u \mathrm{d} x \mathrm{d} \tau
				\geq - \int^0_t (\delta\left\| \triangledown \pl_t q_2 \right\|^2_1 + C_{\delta} \left\| \pl_t u \right\|^2_1 \,) \mathrm{d} \tau.
			\end{split}
		\end{equation}
	    
	    Next, we consider $m \geq 2$ for $q_2$, we do integration by parts
		\begin{equation}\label{eqL10-4}
			\begin{split}
				&
				\int^t_0 \int_{\Omega} Z^{\alpha}\triangledown \pl_t q_2 \cdot Z^{\alpha}\pl_t u \mathrm{d} x \mathrm{d} \tau = \int^0_t \int_{\Omega} [Z^{\alpha}, \triangledown] \pl_t q_2 \cdot Z^{\alpha}\pl_t u + \triangledown Z^{\alpha} \pl_t q_2 \cdot Z^{\alpha}\pl_t u \mathrm{d} x \mathrm{d} \tau
				\\
				&\quad
				\geq - C_{m+1} \int^t_0 (\left\| \triangledown \pl_t q_2 \right\|^2_{m-1} + \left\| \pl_t u \right\|^2_m \, ) \mathrm{d} \tau + \int^0_t \int_{\Omega} \triangledown Z^{\alpha} \pl_t q_2 \cdot Z^{\alpha}\pl_t u \mathrm{d} x \mathrm{d} \tau
				\\
				&\quad
				= - C_{m+1} \int^t_0 ( \left\| \triangledown \pl_t q_2 \right\|^2_{m-1} + \left\| \pl_t u \right\|^2_m \, ) \mathrm{d} \tau
				\\
				&\quad\quad
				+ \int^0_t \int_{\pl \Omega} Z^{\alpha} \pl_t q_2 Z^{\alpha}\pl_t u \cdot n \mathrm{d} \sigma \mathrm{d} \tau - \int^0_t \int_{\Omega} Z^{\alpha} \pl_t q_2 [\triangledown \cdot , Z^{\alpha}]\pl_t u \mathrm{d} x \mathrm{d} \tau
				\\
				&\quad
				\geq - C_{m+1} \int^t_0 ( \left\| \triangledown \pl_t q_2 \right\|^2_{m-1} + \left\| \pl_t u \right\|^2_m + \left\| \triangledown \pl_t u \right\|^2_{m-1} \, ) \mathrm{d} \tau + \int^0_t \int_{\pl \Omega} Z^{\alpha} \pl_t q_2 Z^{\alpha}\pl_t u \cdot n \mathrm{d} \sigma \mathrm{d} \tau,
			\end{split}
		\end{equation}
		where we do integration by parts and use $\pl_t divu=0$.

		Since we have $(u \cdot n)|_{\pl \Omega}=0$ and if $\alpha_3 \neq 0$, the boundary term will vanish, that is we have
		$$Z^{\alpha}\pl_t u \cdot n = Z^{\alpha}\pl_t(u \cdot n) - \sum_{\beta+ \gamma= \alpha, | \beta| \geq 1} \mathcal{C}_{\beta, \gamma} Z^{\gamma}\pl_t u \cdot Z^{\beta} n= - \sum_{\beta+ \gamma= \alpha, | \beta| \geq 1} \mathcal{C}_{\beta, \gamma} Z^{\gamma}\pl_t u \cdot Z^{\beta} n.$$
		We then do integration by parts again along the boundary
		\begin{equation}\label{eqL10-5}
			\begin{split}
				&
				\int^0_t \int_{\pl \Omega} Z^{\alpha} \pl_t q_2 Z^{\alpha}\pl_t u \cdot n \mathrm{d} \sigma \mathrm{d} \tau = - \int^0_t \int_{\pl \Omega} Z^{\alpha} \pl_t q_2 \sum_{\beta+ \gamma= \alpha, | \beta| \geq 1}\mathcal{C}_{\beta, \gamma} Z^{\gamma}\pl_t u \cdot Z^{\beta} n \mathrm{d} \sigma \mathrm{d} \tau
				\\
				&\quad
				\geq - C_{m+2} \int^0_t \sum_{|\beta|= m-1, |\gamma|\leq m-1}|Z^{\beta} \pl_t q_2 |_{L^2(\pl \Omega)} | Z^{\gamma}\pl_t u |_{H^1(\pl \Omega)} \, \mathrm{d} \sigma \mathrm{d} \tau
				\\
				&\quad
				\geq - C_{m+2} \int_0^t \sum_{|\beta| = m-1, |\gamma|\leq m-1}\frac{1}{\sqrt{\epsilon}} (\left\|\triangledown Z^{\beta}\pl_t q_2\right\|+\left\|Z^{\beta}\pl_t q_2\right\|)\left\|Z^{\beta}\pl_t q_2\right\|
				\\
				&\quad\quad
				+\sqrt{\epsilon} (\left\|\triangledown Z^{\gamma}\pl_t u \right\|_1+\left\|Z^{\gamma}\pl_t u \right\|_1)\left\| Z^{\gamma}\pl_t u\right\|_1  \, \mathrm{d} \tau
				\\
				&\quad
				\geq - C_{m+2} \int_{0}^{t} (\frac{1}{\epsilon} \left\|\triangledown \pl_t q_2 \right\|^2_{m-1} + \delta \frac{\epsilon}{4} \left\|\triangledown \pl_t u \right\|^2_m + C_{\delta}\left\|\pl_t u \right\|^2_m \,) \mathrm{d} \tau,
			\end{split}
		\end{equation}
		we use Proposition \ref{prop3} (trace theorem) in the first inequality.
		
		Therefore, combining the above we get
		\begin{equation}\label{eqL10-6}
			\begin{split}
				&
				\int^t_0 \int_{\Omega} Z^{\alpha}\triangledown \pl_t q \cdot Z^{\alpha}\pl_t u \mathrm{d} x \mathrm{d} \tau 
				\\
				&\quad
				\geq - C_{m+2} C_{\delta} \int^t_0 (\left\| \triangledown^2 \pl_t q_1 \right\|^2_{m-1} + \frac{1}{\epsilon} \left\| \triangledown \pl_t q_2 \right\|^2_{m-1} + \left\| \pl_t u \right\|^2_m + \left\| \triangledown \pl_t u \right\|^2_{m-1} \,) \mathrm{d} \tau
				\\
				&\quad\quad
				-\delta \int_{0}^{t} \left\|\triangledown \pl_t q_2 \right\|^2_1 \mathrm{d}\tau - \delta \epsilon \int_{0}^{t} \left\| \triangledown \pl_t u \right\|^2_m \mathrm{d} \tau.
			\end{split}
		\end{equation}
		
		For the terms on the right-hand side, we do integration by parts and recall $\triangledown \cdot (Z^{\alpha} \pl_t \omega \times Z^{\alpha} \pl_t u)= Z^{\alpha} \pl_t u \cdot \triangledown \times Z^{\alpha} \pl_t \omega - Z^{\alpha}\pl_t \omega \cdot \triangledown \times Z^{\alpha} \pl_t u$ and $Z^{\alpha} \pl_t \omega \times Z^{\alpha} \pl_t u \cdot n = n \times Z^{\alpha} \pl_t \omega \cdot Z^{\alpha} \pl_t u$, we get
		\begin{equation}\label{eqL10-7-2}
			\begin{split}
				&
				\int^t_0 \int_{\Omega} - \mu(\epsilon ) Z^{\alpha} (\triangledown \times \pl_t \omega) \cdot Z^{\alpha} \pl_t u \, \mathrm{d} x \mathrm{d} \tau =  \mu(\epsilon ) \int^t_0 \int_{\Omega} - [Z^{\alpha} , \triangledown \times] \pl_t \omega \cdot Z^{\alpha} \pl_t u - \triangledown \times Z^{\alpha} \pl_t \omega \cdot Z^{\alpha} \pl_t u \, \mathrm{d} x \mathrm{d} \tau 
				\\
				&\quad
				\leq - \mu (\epsilon) \int^t_0 \int_{\Omega}\triangledown \times Z^{\alpha} \pl_t u \cdot \triangledown \times Z^{\alpha} \pl_t u + [Z^{\alpha} , \triangledown \times] \pl_t u \cdot \triangledown \times Z^{\alpha} \pl_t u \, \mathrm{d} x \mathrm{d} \tau
				\\
				&\quad\quad
				- \mu(\epsilon) \int^t_0 \int_{\pl \Omega} n \times Z^{\alpha} \pl_t \omega \cdot Z^{\alpha} \pl_t u \, \mathrm{d} \sigma \mathrm{d} \tau + C_{m+1} \int_0^t ( \delta \mu(\epsilon)^2 \left\| \triangledown^2 \pl_t u \right\|^2_{m-1} + C_{\delta} \left\| \pl_t u \right\|^2_m \, ) \mathrm{d} \tau
				\\
				&\quad
				\leq - \frac{3}{4}\mu (\epsilon) \int^t_0 \left\| \triangledown \times Z^{\alpha} \pl_t u \right\|^2 \, \mathrm{d} \tau + C_{m+1} \int^t_0  ( \mu (\epsilon) \left\| \triangledown \pl_t u \right\|^2_{m-1} + \delta \mu(\epsilon)^2 \left\| \triangledown^2 \pl_t u \right\|^2_{m-1} + C_{\delta} \left\| \pl_t u \right\|^2_m \, ) \mathrm{d} \tau
				\\
				&\quad \quad
				- \mu(\epsilon) \int^t_0 \int_{\pl \Omega} n \times Z^{\alpha} \pl_t \omega \cdot Z^{\alpha} \pl_t u \, \mathrm{d} \sigma \mathrm{d} \tau.
			\end{split}
		\end{equation}

		Then we use the boundary condition \eqref{eq1-13} such that $n \times \omega = \Pi(Ku)$ with $K=2\alpha I- 2S(n)$ and we have $\alpha_3 =0$ since if $\alpha_3 \neq 0$ the boundary term will vanish, thus
		\begin{equation}\label{eqL10-7-3}
			\begin{split}
				&
				(n \times Z^{\alpha} \pl_t \omega)|_{\pl \Omega}=(Z^{\alpha}\pl_t (n \times \omega)-\sum_{\beta+ \gamma= \alpha, | \beta| \geq 1}\mathcal{C}_{\beta, \gamma}(Z^{\beta}n \times Z{\gamma}\pl_t \omega))|_{\pl \Omega}
				\\
				&\quad
				=(Z^{\alpha}\pl_t\Pi(Ku) - \sum_{\beta+ \gamma= \alpha, | \beta| \geq 1}\mathcal{C}_{\beta, \gamma}(Z^{\beta}n \times Z^{\gamma} \pl_t \omega))|_{\pl \Omega}.
			\end{split}
		\end{equation}
		Applying Proposition \ref{prop3} (Trace Theorem), we obtain
		\begin{equation}\label{eqL10-7}
			\begin{split}
				&
				- \mu(\epsilon) \int^t_0 \int_{\pl \Omega} n \times Z^{\alpha} \pl_t \omega \cdot Z^{\alpha} \pl_t u \, \mathrm{d} \sigma \mathrm{d} \tau
				\\
				&\quad
				= - \mu(\epsilon) \int^t_0 \int_{\pl \Omega} Z^{\alpha} \pl_t ( \Pi (K u) ) \cdot Z^{\alpha} \pl_t u + \sum_{\beta+ \gamma= \alpha, | \beta| \geq 1} \mathcal{C}_{\beta,\gamma} Z^{\beta}n \times Z^{\gamma} \pl_t \omega \cdot Z^{\alpha} \pl_t u \, \mathrm{d} \sigma \mathrm{d} \tau
				\\
				&\quad
				\leq \frac{1}{4}\mu (\epsilon) \int^t_0 \left\| \triangledown \times Z^{\alpha} \pl_t u \right\|^2 \, \mathrm{d} \tau + C_{m+2} \int^t_0 (\mu (\epsilon) \left\| \triangledown \pl_t u \right\|^2_{m-1} + \delta \mu(\epsilon)^2 \left\| \triangledown^2 \pl_t u \right\|^2_{m-1} + C_{\delta} \left\| \pl_t u \right\|^2_m \,) \mathrm{d} \tau,
			\end{split}
		\end{equation}
		where we use Proposition \ref{prop4} such that
		\begin{equation}\label{eqL10-7-4}
			\begin{split}
				&
				\left\|\triangledown Z^{\alpha} \pl_t u \right\|^2 \leq \left\| Z^{\alpha} \pl_t u \right\|^2_{H^1} \leq C_2 (\left\|\triangledown \times Z^{\alpha} \pl_t u \right\|^2 + \left\| div Z^{\alpha} \pl_t u \right\|^2 + \left\|Z^{\alpha} \pl_t u \right\|^2 + |Z^{\alpha} \pl_t u \cdot n |^2_{H^{\frac{1}{2}}(\pl \Omega)}),
			\end{split}
		\end{equation}
		with the commutator $[\triangledown \cdot , Z^{\alpha}]\pl_t u$ already has estimates since we know $\pl_t div u=0$. Also, the boundary term will vanish if $\alpha_3 \neq 0$ so we consider $\alpha_3=0$ and apply Proposition \ref{prop3} (trace theorem)
		\begin{equation}\label{eqL10-7-5}
			\begin{split}
				&
				|Z^{\alpha} \pl_t u \cdot n |^2_{H^{\frac{1}{2}}(\pl \Omega)} = |Z^{\alpha}\pl_t (u \cdot n)- \sum_{\beta+ \gamma=\alpha, |\beta| \geq 1}(Z^{\gamma}\pl_t u \cdot Z^{\beta}n)|^2_{H^{\frac{1}{2}}(\pl \Omega)}
				\\
				&\quad
				\leq C_{m+2}\sum_{| \beta| \leq m-1}( \left\|\triangledown Z^{\beta} \pl_t u \right\|+\left\|Z^{\beta} \pl_t u \right\|)\left\|Z^{\beta}\pl_t u \right\|_1.
			\end{split}
		\end{equation}

		Applying Proposition \ref{prop1} and Proposition \ref{prop5}, we have
		\begin{equation}\label{eqL10-8}
			\begin{split}
				&
				\int^t_0 \int_{\Omega}  Z^{\alpha} \pl_t (B \cdot \triangledown B) \cdot Z^{\alpha} \pl_t u \, \mathrm{d}x \mathrm{d} \tau = \int^t_0 \int_{\Omega} \sum_{| \beta| = m-1} \mathcal{C}_{\beta, \gamma} Z^{\beta} \pl_t ( Z B \cdot \triangledown B)  \cdot Z^{\alpha} \pl_t u 
				\\
				&\quad
				+ Z^{\alpha}(\pl_t B \cdot \triangledown B) \cdot Z^{\alpha} \pl_t u + B \cdot [Z^{\alpha} , \triangledown] \pl_t B \cdot Z^{\alpha} \pl_t u + B \cdot \triangledown Z^{\alpha} \pl_t B \cdot Z^{\alpha} \pl_t u  \, \mathrm{d}x \mathrm{d} \tau
				\\
				&\quad
				\leq C_{m+1} (1+ \left\|(B,ZB, \pl_t B,\triangledown B)\right\|^2_{L^{\infty}}) \int_0^t (\left\| \pl_t u \right\|_m^2 + \left\| \pl_t B \right\|_m^2 + \left\| \triangledown \pl_t B \right\|_{m-1}^2 + \left\| (B,\triangledown B) \right\|_m^2 \,) \mathrm{d} \tau 
				\\
				&\quad\quad
				+ \int^t_0 \int_{\Omega} B \cdot \triangledown Z^{\alpha} \pl_t B \cdot Z^{\alpha} \pl_t u  \, \mathrm{d}x \mathrm{d} \tau.
			\end{split}
		\end{equation}
		
		Then recall \eqref{Bt}
		\begin{equation}
			\pl_t \pl_t B + \pl_t((u \cdot \triangledown)B) - \pl_t((B \cdot \triangledown)u) = \nu(\epsilon) \triangle \pl_t B.
		\end{equation}
		
		We have by integrating by parts and applying Proposition \ref{prop5}
		\begin{equation}\label{eqL10-9}
			\begin{split}
				&
				\int^t_0 \int_{\Omega} B \cdot \triangledown Z^{\alpha} \pl_t B \cdot Z^{\alpha} \pl_t u  \, \mathrm{d}x \mathrm{d} \tau = \int^t_0 \int_{\Omega} -Z^{\alpha} \pl_t \pl_t B \cdot Z^{\alpha} \pl_t B - u \cdot [Z^{\alpha}, \triangledown] \pl_t B \cdot Z^{\alpha} \pl_t B 
				\\
				&\quad
				- \frac{1}{2} u \cdot \triangledown |Z^{\alpha} \pl_t B|^2 - Z^{\alpha} (\pl_t u \cdot \triangledown B) \cdot Z^{\alpha} \pl_t B - \sum_{| \beta| = m-1} \mathcal{C}_{\beta, \gamma} Z^{\beta}(Zu \cdot \triangledown B) \cdot Z^{\alpha} \pl_t B 
				\\
				&\quad
				+ B \cdot [Z^{\alpha}, \triangledown] \pl_t u \cdot Z^{\alpha} \pl_t B + B \cdot \triangledown (Z^{\alpha} \pl_t u \cdot Z^{\alpha} \pl_t B) + Z^{\alpha}(\pl_t B \cdot \triangledown u) \cdot Z^{\alpha} \pl_t B
				\\
				&\quad
				+ \sum_{| \beta| = m-1} \mathcal{C}_{\beta, \gamma} Z^{\beta}\pl_t (Z B \cdot \triangledown u) \cdot Z^{\alpha}\pl_t B + \nu(\epsilon) Z^{\alpha} \triangle \pl_t B \cdot Z^{\alpha} \pl_t B \, \mathrm{d}x \mathrm{d} \tau
				\\
				&\quad
				\leq \frac{1}{2} \left\| \pl_t B (0) \right\|^2_m - \frac{1}{2} \left\| \pl_t B (t) \right\|^2_m + \int^t_0 \int_{\Omega} \nu(\epsilon) Z^{\alpha} \triangle \pl_t B \cdot Z^{\alpha} \pl_t B \, \mathrm{d}x \mathrm{d} \tau
				\\
				&\quad\quad
				+C_{m+1}(1+ \left\| (u, B, Zu, ZB, \pl_t u, \pl_t B, \triangledown u, \triangledown B) \right\|^2_{L^{\infty}}) 
				\\
				&\qquad\qquad
				\int^t_0 (\left\| (\triangledown u, \triangledown B) \right\|^2_m + \left\| (\pl_t u, \pl_t B) \right\|^2_m + \left\| (\triangledown \pl_t u, \triangledown \pl_t B) \right\|^2_{m-1} \,) \mathrm{d} \tau,
				\end{split}
		\end{equation}
		where we use $div u =0$, $div B=0$, and the boundary condition \eqref{eq1-2} such that $u \cdot n=0$ and $B \cdot n=0$ on the boundary.

		Recall
		$$\triangle B = \triangledown \pl_t div B - \triangledown \times \pl_t \curl B= - \triangledown \times \pl_t \curl B,$$
		$$\triangledown \cdot (Z^{\alpha} \pl_t \curl B \times Z^{\alpha}\pl_t B )= Z^{\alpha} \pl_t B \cdot \triangledown \times Z^{\alpha} \pl_t \curl B - Z^{\alpha} \pl_t \curl B \cdot \triangledown \times Z^{\alpha} \pl_t B,$$
		and
		$$Z^{\alpha} \pl_t \curl B \times Z^{\alpha} \pl_t B \cdot n = n \times Z^{\alpha} \pl_t \curl B \cdot Z^{\alpha} \pl_t B.$$
		
		Then we get
		\begin{equation}\label{eqL10-10}
			\begin{split}
				&
				\int^t_0 \int_{\Omega} \nu(\epsilon) Z^{\alpha} \triangle \pl_t B \cdot Z^{\alpha} \pl_t B \, \mathrm{d}x \mathrm{d} \tau = - \nu(\epsilon) \int^t_0 \int_{\Omega}  \triangledown \times Z^{\alpha} \pl_t \curl B \cdot Z^{\alpha} \pl_t B + [Z^{\alpha}, \triangledown \times] \pl_t \curl B \cdot Z^{\alpha} \pl_t B \, \mathrm{d}x \mathrm{d} \tau
				\\
				&\quad
				\leq C_{m+1} \int^t_0 (\delta \nu(\epsilon)^2 \left\| \triangledown^2 \pl_t B \right\|^2_{m-1} + C_{\delta} \left\| \pl_t B \right\|^2_m \,) \mathrm{d} \tau
				\\
				&\quad\quad
				- \nu(\epsilon)\int_{0}^{t} \int_{\Omega} Z^{\alpha} \pl_t \curl B \cdot \triangledown \times Z^{\alpha} \pl_t B \mathrm{d} x \mathrm{d} \tau - \nu(\epsilon) \int_{0}^{t} \int_{\pl \Omega} n \times Z^{\alpha} \pl_t \curl B \cdot Z^{\alpha} \pl_t B \mathrm{d} \sigma \mathrm{d} \tau
				\\
				&\quad
				\leq - \frac{3}{4}\nu (\epsilon) \int^t_0 \left\| \triangledown \times Z^{\alpha} \pl_t B \right\|^2 \, \mathrm{d} \tau + C_{m+1} \int^t_0 (\nu (\epsilon) \left\| \triangledown \pl_t B \right\|^2_{m-1} + \delta \nu(\epsilon)^2 \left\| \triangledown^2 \pl_t B \right\|^2_{m-1} + C_{\delta} \left\| \pl_t B \right\|^2_m \,) \mathrm{d} \tau
				\\
				&\quad\quad
				- \nu(\epsilon) \int_{0}^{t} \int_{\pl \Omega} n \times Z^{\alpha} \pl_t \curl B \cdot Z^{\alpha} \pl_t B \mathrm{d} \sigma \mathrm{d} \tau.
			\end{split}
		\end{equation}
		
		For the boundary term, we have according to \eqref{eq1-2} such that $n \times \curl B=0$ on $\pl \Omega$ and consider $\alpha_3=0$ since if $\alpha_3 \neq 0$ this term vanished, thus
		$$(n \times Z^{\alpha}\pl_t \curl B)|_{\pl \Omega}=(Z^{\alpha}\pl_t (n \times \curl B)- \sum_{\beta+ \gamma= \alpha, | \beta| \geq 1} \mathcal{C}_{\beta, \gamma} Z^{\beta}n \times Z^{\gamma}\pl_t \curl B)|_{\pl \Omega}$$
		$$=(-  \sum_{\beta+ \gamma= \alpha, | \beta| \geq 1} \mathcal{C}_{\beta, \gamma} Z^{\beta}n \times Z^{\gamma}\pl_t \curl B)|_{\pl \Omega}.$$
		
		Applying Proposition \ref{prop3} (trace theorem), we obtain
		\begin{equation}\label{eqL10-10-2}
			\begin{split}
				&
				- \nu(\epsilon) \int_{0}^{t} \int_{\pl \Omega} n \times Z^{\alpha} \pl_t \curl B \cdot Z^{\alpha} \pl_t B \mathrm{d} \sigma \mathrm{d} \tau = \nu(\epsilon) \int_{0}^{t} \int_{\pl \Omega} \sum_{\beta+ \gamma= \alpha, | \beta| \geq 1} \mathcal{C}_{\beta, \gamma} Z^{\beta}n \times Z^{\gamma}\pl_t \curl B \cdot Z^{\alpha} \pl_t B \mathrm{d} \sigma \mathrm{d} \tau
				\\
				&\quad
				\leq C_{m+2} \int^t_0 (\nu(\epsilon) \left\| \triangledown \pl_t B \right\|^2_{m-1} + \delta \nu(\epsilon)^2 \left\| \triangledown^2 \pl_t B \right\|^2_{m-1} + C_{\delta} \left\| \pl_t B \right\|^2_m \,) \mathrm{d} \tau.
			\end{split}
		\end{equation}

		Combine all above, we now have
		\begin{equation}\label{eqL10-11}
			\begin{split}
				&
				\frac{1}{2} \left\| \pl_t u (t) \right\|_m^2 + \frac{1}{2} \left\| \pl_t B (t) \right\|_m^2 + \frac{\mu(\epsilon)}{2} \int^{t}_{0} \left\| \triangledown \times Z^{\alpha} \pl_t u\right\|^2 \,  \mathrm{d} \tau + \frac{\nu(\epsilon)}{2} \int^{t}_{0}  \left\| \triangledown \times Z^{\alpha} \pl_t B\right\|^2 \, \mathrm{d} \tau
				\\
				&\quad
				\leq \frac{1}{2} \left\| \pl_t u (0) \right\|_m^2 + \frac{1}{2} \left\| \pl_t B (0) \right\|_m^2 + \delta  \int_0^t ( \mu(\epsilon)^2 \left\| \triangledown^2 \pl_t u \right\|^2_{m-1} + \nu(\epsilon)^2 \left\| \triangledown^2 \pl_t B \right\|^2_{m-1} \, ) \mathrm{d} \tau
				\\
				&\quad\quad
				+ C_{m+2} C_{\delta} (1+ \left\| (u,B, Zu, ZB, \pl_t u, \pl_t B,\triangledown u, \triangledown B)\right\|^2_{L^{\infty}})
				\\
				&\quad\quad\quad\quad
				\int^{t}_{0} ( \left\| (\triangledown u, \triangledown B) \right\|_m^2+\left\| (\pl_t u, \pl_t B) \right\|_m^2+\left\| (\triangledown \pl_t u, \triangledown \pl_t B) \right\|_{m-1}^2 \, ) \mathrm{d} \tau
				\\
				&\quad \quad
				+C_{m+2} \int^{t}_{0} ( \left\| \triangledown^2 \pl_t q_1 \right\|_{m-1}^2+ \frac{1}{\epsilon}\left\| \triangledown \pl_t q_2 \right\|_{m-1}^2 \, ) \mathrm{d} \tau + \delta \int^{t}_{0} \left\| \triangledown \pl_t q_2 \right\|^2_1 \mathrm{d} \tau,
			\end{split}
		\end{equation}
		we end the proof by using Proposition \ref{prop4} and recall \eqref{eqL10-7-5}
		\begin{equation}\label{eqL10-12}
			\begin{split}
				&
				C_2(\left\|\triangledown \times Z^{\alpha} \pl_t u \right\|^2 ) \geq \left\|Z^{\alpha} \pl_t u \right\|^2_{H^1} - C_2(\left\|div Z^{\alpha}\pl_t u \right\|^2+ \left\|Z^{\alpha} \pl_t u \right\|^2 + | Z^{\alpha}\pl_t u \cdot n |^2_{H^{\frac{1}{2}}(\pl \Omega)}) 
				\\
				&\quad
				\geq \left\|\triangledown \pl_t u \right\|_m^2 - C_{m+2}(\left\|\triangledown \pl_t u \right\|^2_{m-1} + \left\|\pl_t u \right\|^2_m),
			\end{split}
		\end{equation}
		same for $\left\|\triangledown \times Z^{\alpha} \pl_t B \right\|^2$, we then end the proof.
		
	\end{proof}

		\hspace*{\fill}\\
		\hspace*{\fill}

		\begin{lem}\label{lem11}
			For a smooth solution to \eqref{eq1-1} and \eqref{eq1-2}, it holds that for all $\epsilon \in (0,1]$
			\begin{equation}\label{eqL11}
				\begin{split}
					&
					\frac{1}{2} \left\| \pl_t \eta (t) \right\|^2 + \frac{1}{2} \left\| \pl_t \chi (\curl B \times n) (t) \right\|^2 + \frac{\mu(\epsilon)}{2} \int^{t}_{0} \left\| \triangledown \pl_t \eta \right\|^2 \, \mathrm{d} \tau + \frac{\nu(\epsilon)}{2} \int^{t}_{0} \left\| \triangledown \pl_t \chi (\curl B \times n) \right\|^2 \, \mathrm{d} \tau
					\\
					&\quad
					\leq \frac{1}{2} \left\| \pl_t \eta (0) \right\|^2 + \frac{1}{2} \left\| \pl_t \chi (\curl B \times n) (0) \right\|^2 + \delta  \int_0^t ( \mu(\epsilon)^2 \left\| \triangledown^2 \pl_t u \right\|^2 + \nu(\epsilon)^2 \left\| \triangledown^2 \pl_t B \right\|^2 \, ) \mathrm{d} \tau
					\\
					&\quad\quad
					+ C_{\delta} C_3 (1+ \left\| (u, B, Zu, ZB, \pl_t u, \pl_t B, \triangledown u, \triangledown B, Z\pl_t u, Z\pl_t B) \right\|^2_{L^{\infty}})
					\\
					&\quad\quad
					\int^{t}_{0} ( \left\| (u, B, \triangledown u, \triangledown B) \right\|_1^2 + \left\| (\pl_t u, \pl_t B) \right\|^2 + \left\| (\triangledown \pl_t u, \triangledown \pl_t B) \right\|^2 + \left\| \Pi \triangledown \pl_t q \right\|^2 \, )\mathrm{d} \tau.
				\end{split}
			\end{equation}
		\end{lem}

		\begin{proof}
		
		Recall the coordinate we used in Lemma \ref{lem7}
		\begin{equation}
			\Psi^n(y,z)= \left( \begin{array}{c}
				y\\
				\psi(y)
			\end{array} \right) - z n(y)=x,
		\end{equation}
		where n is the unit outward normal such that
		\begin{equation}
			n(y)= \frac{1}{\sqrt{1+|\triangledown \psi (y)|^2}}\left( \begin{array}{c}
				\pl_1\psi (y)\\
				\pl_2 \psi(y)\\
				-1
			\end{array} \right).
		\end{equation}
		We extend $n$ and $\Pi$ in the interior by setting
		\begin{equation}
			n(\Psi^n(y,z))=n(y), \, \Pi(\Psi^n(y,z))=\Pi(y).
		\end{equation}
		
		Noted that in the associated local basis $(\pl_{y^1}, \pl_{y^2}, \pl_z)$ of $\mathbb{R}^3$, we have $\pl_z=\pl_n$ and
		\begin{equation}
			(\pl_{y^i})|_{\Psi^n(y,z)} \cdot (\pl_z)|_{\Psi^n(y,z)}=0.
		\end{equation}
		
		The scalar product on $\mathbb{R}^3$ induces in the coordinate system the Riemannian metric $g$ with the form
		\begin{equation}
			g(y,z)= \left( \begin{matrix}
				\Tilde{g}(y,z)  & 0\\
				0 & 1
			\end{matrix} \right).
		\end{equation}
		
		As we defined before
		\begin{equation}
			\eta(y,z)= \chi (\omega \times n + \Pi (K u)),
		\end{equation}
		then $\eta$ solves the following function
		\begin{equation}\label{new}
			\pl_t \eta + (u \cdot \triangledown) \eta - (B \cdot \triangledown) ( \chi  (\curl B \times n)) - \mu (\epsilon) \triangle \eta  = \chi F_1 + \chi F_2 + F_3 + F_4=F,
		\end{equation}
		where
		\begin{equation}
			F_1= (\omega \cdot \triangledown)u \times n - (\curl B \cdot \triangledown)B \times n - \Pi(K \triangledown q) + \Pi(K((B \cdot \triangledown) B)),
		\end{equation}
		\begin{equation}
			\begin{split}
				&
				F_2= - 2 \sum^2_{i=1} \mu (\epsilon) \pl_i \omega \times \pl_i n - \mu (\epsilon) \omega \times \triangle n + \omega \times( (u \cdot \triangledown) n) 
				\\
				&\quad
				- \curl B \times ((B \cdot \triangledown ) n )+ (u \cdot \triangledown) (\Pi K) u - 2 \mu (\epsilon) \sum^2_{i=1} \pl_i (\Pi K) \pl_i u,
			\end{split}
		\end{equation}
		\begin{equation}
			\begin{split}
				&
				F_3= (u \cdot \triangledown) \chi \cdot ( \omega \times n + \Pi(Ku)) - (B \cdot \triangledown) \chi \cdot (\curl B \times n)
				\\
				&\quad
				-2 \sum^3_{i=1} \mu (\epsilon) \pl_i \chi \cdot \pl_i ( \omega \times n + \Pi(Ku)) - \mu( \epsilon) \triangle \chi \cdot( \omega \times n + \Pi(Ku)),
			\end{split}
		\end{equation}
		and
		\begin{equation}
			F_4= \mu (\epsilon )\chi \triangle (\Pi K) u.
		\end{equation}

		Multiplying $\pl_t$\eqref{new} with $\pl_t \eta$ and integrating it with respect to $x$ and $t$, one has
		\begin{equation} \label{eqL11-1}
			\begin{split}
				&
				\int^t_0 \int_{\Omega} \pl_t \pl_t \eta \cdot \pl_t \eta +  \pl_t ((u \cdot \triangledown) \eta) \cdot \pl_t \eta -\pl_t (B \cdot \triangledown (\chi (\curl B \times n))) \cdot \pl_t \eta 
				\\
				&\quad
				- \mu (\epsilon) \triangle \pl_t \eta \cdot \pl_t \eta \, \mathrm{d} x  \mathrm{d} \tau = \int^t_0 \int_{\Omega} \pl_t F \cdot \pl_t \eta \, \mathrm{d} x  \mathrm{d} \tau.
			\end{split}
		\end{equation}

		Firstly, we have
		\begin{equation} \label{eqL11-2}
			\begin{split}
				&
				\int^t_0 \int_{\Omega} \pl_t \pl_t \eta \cdot \pl_t \eta \, \mathrm{d} x  \mathrm{d} \tau = \frac{1}{2} \left\| \pl_t \eta (t) \right\|^2 - \frac{1}{2} \left\| \pl_t \eta (0) \right\|^2.
			\end{split}
		\end{equation}
		
		Integrating by parts and applying boundary condition \eqref{eq1-2}, we have
		\begin{equation} \label{eqL11-3}
			\begin{split}
				&
				\int^t_0 \int_{\Omega} \pl_t ((u \cdot \triangledown) \eta) \cdot \pl_t \eta \, \mathrm{d} x  \mathrm{d} \tau = \int^t_0 \int_{\Omega} ((\pl_t u \cdot \triangledown) \eta) \cdot \pl_t \eta + \frac{1}{2} u \cdot \triangledown |\pl_t \eta|^2 \, \mathrm{d} x  \mathrm{d} \tau
				\\
				&\quad
				= \int^t_0 \int_{\Omega} ((\pl_t u \cdot \triangledown) \eta) \cdot \pl_t \eta \, \mathrm{d} x  \mathrm{d}.
			\end{split}
		\end{equation}
		
		In the local coordinate, we introduced here, as we set
		$$\Tilde{\omega}(y,z)= \omega(\Psi^n(y,z)), \quad \Tilde{u}(y,z) = u(\Psi^n(y,z)),$$
		we have
		\begin{equation} \label{basis}
			\begin{split}
				&
				(u \cdot \triangledown \eta) \circ \Psi^n = \Tilde{u}^1 \pl_{y^1} \Tilde{\eta} + \Tilde{u}^2 \pl_{y^2} \Tilde{\eta} + \Tilde{u} \cdot n \pl_{z} \Tilde{\eta},
			\end{split}
		\end{equation}
		and note that we have $\pl_z=\pl_n$.

		We have for $z \neq 0$ from \eqref{eqL7-6} and by boundary condition $(u \cdot n)|_{\pl \Omega}=0$
		\begin{equation} \label{eqL11-4}
			\begin{split}
				&
				|\pl_t  \Tilde{u} \cdot n \pl_z \Tilde{\eta} |= |\frac{\pl_t ( \Tilde{u} \cdot n)}{\varphi(z)} \varphi(z)\pl_z \Tilde{\eta}|
				\\
				&\quad
				\lesssim | \pl_z \pl_t ( \Tilde{u} \cdot n) Z_3 \Tilde{\eta} | + | \pl_t ( \Tilde{u} \cdot n) Z_3 \Tilde{\eta}| = | \pl_n \pl_t ( \Tilde{u} \cdot n) Z_3 \Tilde{\eta} | + | \pl_t ( \Tilde{u} \cdot n) Z_3 \Tilde{\eta} |.
			\end{split}
		\end{equation}
		
		By divergence-free
		\begin{equation} \label{eqL11-5}
			\begin{split}
				&
				\frac{1}{|g|^{\frac{1}{2}}} ( \pl_{y^1} ( |g|^{\frac{1}{2}} \Tilde{u}^1 ) + \pl_{y^2} ( |g|^{\frac{1}{2}} \Tilde{u}^2 ) + \pl_n ( |g|^{\frac{1}{2}} \Tilde{u} \cdot n ) )=0,
			\end{split}
		\end{equation}
	    thus we have
	    \begin{equation} \label{un}
	    	\begin{split}
	    		&
	    		\pl_n (\Tilde{u} \cdot n) = -\frac{1}{|g|^{\frac{1}{2}}} ( \pl_{y^1} ( |g|^{\frac{1}{2}} \Tilde{u}^1 ) + \pl_{y^2} ( |g|^{\frac{1}{2}} \Tilde{u}^2 ) + (\pl_n |g|^{\frac{1}{2}}) (\Tilde{u} \cdot n ) ),
	    	\end{split}
	    \end{equation}
		then we get
		\begin{equation} \label{eqL11-6}
			\begin{split}
				&
				\left\| \pl_t  u \cdot n \pl_z \eta \right\|^2 \leq C_3 \left\|(\pl_t u , Z\pl_t u) \right\|^2_{L^{\infty}} \left\|\eta\right\|^2_1.
			\end{split}
		\end{equation}
		
		Therefore,
		\begin{equation} \label{eqL11-7}
			\begin{split}
				&
				\int^t_0 \int_{\Omega} \pl_t ((u \cdot \triangledown) \eta) \cdot \pl_t \eta \, \mathrm{d} x  \mathrm{d} \tau 
				\geq - C_3( 1+ \left\|(\pl_t u , Z\pl_t u) \right\|^2_{L^{\infty}} ) \int^t_0 (\left\|(u , \triangledown u )\right\|^2_1 + \left\|(\pl_t u , \pl_t \triangledown u) \right\|^2 \,) \mathrm{d} \tau.
			\end{split}
		\end{equation}
		
		Similar to above, we define $\widetilde{B}(y,z)=B(\Psi^n(y,z))$ and $\widetilde{\curl B}(y,z)=\curl B(\Psi^n(y,z))$. Then we have
		\begin{equation} \label{basis-B}
			\begin{split}
				&
				(B \cdot \triangledown (\chi (\curl B \times n))) \circ \Psi^n = \Tilde{B}^1 \pl_{y^1} \widetilde{\chi (\curl B \times n)} + \Tilde{B}^2 \pl_{y^2} \widetilde{\chi (\curl B \times n)} + \Tilde{B} \cdot n \pl_{z} \widetilde{\chi (\curl B \times n)}.
			\end{split}
		\end{equation}
	
	    Since $(B\cdot n)|_{\pl \Omega}=0$, we obtain by having $\pl_z=\pl_n$
	    \begin{equation} \label{eqL11-4-B}
	    	\begin{split}
	    		&
	    		|\pl_t  \Tilde{B} \cdot n \pl_z \widetilde{\chi (\curl B \times n)} |= |\frac{\pl_t ( \Tilde{B} \cdot n)}{\varphi(z)} \varphi(z)\pl_z \widetilde{\chi (\curl B \times n)}|
	    		\\
	    		&\quad
	    		\lesssim | \pl_n \pl_t ( \Tilde{B} \cdot n) Z_3 \widetilde{\chi (\curl B \times n)} | + | \pl_t ( \Tilde{B} \cdot n) Z_3 \widetilde{\chi (\curl B \times n)} |.
	    	\end{split}
	    \end{equation}
        
        Since $div B=0$, we get
		\begin{equation} \label{Bn}
			\begin{split}
				&
				\pl_n (\Tilde{B} \cdot n) = -\frac{1}{|g|^{\frac{1}{2}}} ( \pl_{y^1} ( |g|^{\frac{1}{2}} \Tilde{B}^1 ) + \pl_{y^2} ( |g|^{\frac{1}{2}} \Tilde{B}^2 ) + (\pl_n |g|^{\frac{1}{2}}) (\Tilde{B} \cdot n ) ),
			\end{split}
		\end{equation}
	    then for $B$ we also have
	    \begin{equation} \label{eqL11-8-2}
	    	\begin{split}
	    		&
	    		\left\| \pl_t  B \cdot n \pl_z \widetilde{\chi (\curl B \times n)} \right\|^2 \leq C_3 \left\|(\pl_t B , Z\pl_t B) \right\|^2_{L^{\infty}} \left\|\widetilde{\chi (\curl B \times n)}\right\|^2_1.
	    	\end{split}
	    \end{equation}
		
		Therefore, one has
		\begin{equation} \label{eqL11-8}
			\begin{split}
				&
				\int^t_0 \int_{\Omega} -\pl_t ( B \cdot \triangledown (\chi (\curl B \times n))) \cdot \pl_t \eta \, \mathrm{d} x  \mathrm{d} \tau
				\\
				&\quad
				= \int^t_0 \int_{\Omega} - ( \pl_t B \cdot \triangledown (\chi (\curl B \times n))) \cdot \pl_t \eta - B \cdot \triangledown \pl_t \chi (\curl B \times n) \cdot \pl_t \eta \, \mathrm{d} x  \mathrm{d} \tau
				\\
				&\quad
				\geq - C_3( 1+ \left\|(\pl_t B , Z\pl_t B) \right\|^2_{L^{\infty}} ) \int^t_0 (\left\|(B , \triangledown B )\right\|^2_1 + \left\|(\pl_t u , \pl_t \triangledown u) \right\|^2 \,) \mathrm{d} \tau
				\\
				&\quad\quad
				- \int^t_0 \int_{\Omega} B \cdot \triangledown \pl_t \chi (\curl B \times n) \cdot \pl_t \eta \, \mathrm{d} x  \mathrm{d} \tau.
			\end{split}
		\end{equation}

		Recall we have
		\begin{equation}
			\pl_t ( \chi (\curl B \times n)) - (B \cdot \triangledown) \eta + (u \cdot \triangledown)( \chi (\curl B \times n))- \nu (\epsilon) \triangle ( \chi (\curl B \times n)) = G,
		\end{equation}
		where $G= G_1+ G_2 +G_3$ with
		\begin{equation}
			G_1= \chi ( \omega \cdot \triangledown )B \times n - \chi (\curl B \cdot \triangledown)(u \times n) - (B \cdot \triangledown)(\chi \Pi(Ku)),
		\end{equation}
		\begin{equation}
			\begin{split}
				&
				G_2= - \nu (\epsilon) \triangle \chi \cdot \curl B \times n - 2 \nu (\epsilon) \sum^3_{i=1} \pl_i \chi \cdot \pl_i(\curl B \times n) 
				\\
				&\quad
				- (B \cdot \triangledown) \chi \cdot (\omega \times n) + (u \cdot \triangledown) \chi \cdot (\curl B \times n),
			\end{split}
		\end{equation}
		and
		\begin{equation}
			\begin{split}
				&
				G_3= -\nu (\epsilon) \chi ( \curl B \times \triangle n) - 2 \nu (\epsilon )\chi \sum^2_{i=1} \pl_i \curl B \times \pl_i n 
				- \chi \omega \times ( B \cdot \triangledown n) + \chi \omega \times ( u \cdot \triangledown n).
			\end{split}
		\end{equation}

		We obtain by integrating by parts and using the boundary condition \eqref{eq1-2} such that $(u\cdot n)|_{\pl \Omega}=0$, $(B \cdot n)|_{\pl \Omega}=0$. We need to use the argument of getting \eqref{eqL11-6} and \eqref{eqL11-8}
		\begin{equation} \label{eqL11-9}
			\begin{split}
				&
				- \int^t_0 \int_{\Omega} B \cdot \triangledown \pl_t \chi (\curl B \times n) \cdot \pl_t \eta \, \mathrm{d} x  \mathrm{d} \tau
				\\
				&\quad
				= \int_{0}^{t} \int_{\Omega} \pl_t \pl_t (\chi (\curl B \times n)) \cdot \pl_t \chi (\curl B \times n) + \pl_t ((u \cdot \triangledown)(\chi (\curl B \times n))) \cdot \pl_t \chi (\curl B \times n) 
				\\
				&\quad\quad
				- \nu (\epsilon) \triangle \pl_t \chi (\curl B \times n) \cdot \pl_t \chi (\curl B \times n) - \pl_t G \cdot \pl_t \chi (\curl B \times n)
				\\
				&\quad\quad
				- (\pl_t B \cdot \triangledown) \eta \cdot \pl_t  \chi (\curl B \times n) - (B \cdot \triangledown) (\pl_t \eta \cdot \pl_t  \chi (\curl B \times n))\, \mathrm{d}x \mathrm{d} \tau
				\\
				&\quad
				= \frac{1}{2} \left\| \pl_t \chi (\curl B \times n) (t) \right\|^2 -  \frac{1}{2} \left\| \pl_t \chi (\curl B \times n) (0) \right\|^2 
				\\
				&\quad\quad
				+ \int_{0}^{t} \int_{\Omega}  (\pl_t u \cdot \triangledown)(\chi (\curl B \times n)) \cdot \pl_t \chi (\curl B \times n) - (\pl_t B \cdot \triangledown) \eta \cdot \pl_t  \chi (\curl B \times n)
				\\
				&\quad\quad
				- \nu (\epsilon) \triangle \pl_t \chi (\curl B \times n) \cdot \pl_t \chi (\curl B \times n) - \pl_t G \cdot \pl_t \chi (\curl B \times n) \, \mathrm{d}x \mathrm{d} \tau
				\\
				&\quad
				\geq \frac{1}{2} \left\| \pl_t \chi (\curl B \times n) (t) \right\|^2 -  \frac{1}{2} \left\| \pl_t \chi (\curl B \times n) (0) \right\|^2 
				\\
				&\quad\quad
				- C_3( 1+ \left\|(\pl_t u ,\pl_t B , Z \pl_t u, Z\pl_t B) \right\|^2_{L^{\infty}} ) \int^t_0 ( \left\|(u,B ,\triangledown u, \triangledown B )\right\|^2_1 + \left\|(\pl_t B , \pl_t \triangledown B) \right\|^2 \,) \mathrm{d} \tau
				\\
				&\quad\quad
				+ \int_{0}^{t} \int_{\Omega} - \nu (\epsilon) \triangle \pl_t \chi (\curl B \times n) \cdot \pl_t \chi (\curl B \times n) - \pl_t G \cdot \pl_t \chi (\curl B \times n) \, \mathrm{d}x \mathrm{d} \tau.
			\end{split}
		\end{equation}
	    
	    Recall we have
	    $$\triangledown \cdot (\triangledown \pl_t \chi (\curl B \times n) \cdot \pl_t \chi (\curl B \times n))=\triangle \pl_t \chi (\curl B \times n) \cdot \pl_t \chi (\curl B \times n) + \triangledown \pl_t \chi (\curl B \times n) : \triangledown \pl_t \chi (\curl B \times n),$$
	    with
	    $$\triangledown \pl_t \chi (\curl B \times n) : \triangledown \pl_t \chi (\curl B \times n) = \sum_{i=1,2,3} \triangledown \pl_t \chi (\curl B \times n)_i \cdot \triangledown \pl_t \chi (\curl B \times n)_i \, .$$

		Integration by parts, since we have the boundary condition $(\pl_t \chi (\curl B \times n) )_{\pl \Omega}=0$
		\begin{equation} \label{eqL11-10}
			\begin{split}
				&
				\int_{0}^{t} \int_{\Omega} - \nu (\epsilon) \triangle \pl_t \chi (\curl B \times n) \cdot \pl_t \chi (\curl B \times n) \, \mathrm{d}x \mathrm{d} \tau = \nu(\epsilon) \int^t_0 \left\| \triangledown(\chi \curl (\pl_t B) \times n ) \right\|^2 \, \mathrm{d} \tau.
			\end{split}
		\end{equation}
		
		For the term involving $G$, we apply Proposition 3.1
		\begin{equation} \label{eqL11-11}
			\begin{split}
				&
				\int_{0}^{t} \left\| \pl_t G_1 \right\|^2_{m} \, \mathrm{d} \tau \geq - C_{m+3} (1 + \left\|(B,Zu, ZB,\triangledown u, \triangledown B)\right\|^2_{L^{\infty}}) \int^t_0 \left\| (\pl_t B, \pl_t \triangledown u, \pl_t \triangledown B) \right\|_m^2 \, \mathrm{d} \tau,
			\end{split}
		\end{equation}
		\begin{equation} \label{eqL11-12}
			\begin{split}
				&
				\int_{0}^{t} \left\| \pl_t G_2 \right\|^2_{m} \, \mathrm{d} \tau \geq - C_{m+3}(\nu(\epsilon))^2 \int^t_0 \left\| \triangledown^2 \pl_t B \right\|^2_{m} \, \mathrm{d} \tau
				\\
				&\quad
				- C_{m+3} (1 + \left\|(u, B,Zu, ZB,\triangledown u, \triangledown B)\right\|^2_{L^{\infty}}) \int^t_0 \left\| (\pl_t B, \pl_t \triangledown u, \pl_t \triangledown B) \right\|_m^2 \, \mathrm{d} \tau,
			\end{split}
		\end{equation}
		\begin{equation} \label{eqL11-13}
			\begin{split}
				&
				\int_{0}^{t} \left\| \pl_t G_3 \right\|^2_{m} \, \mathrm{d} \tau \geq - C_{m+3}(\nu(\epsilon))^2 \int^t_0 \left\| \triangledown^2 \pl_t B \right\|^2_{m} \, \mathrm{d} \tau
				\\
				&\quad
				- C_{m+3} (1 + \left\|(u, B,Zu, ZB,\triangledown u, \triangledown B)\right\|^2_{L^{\infty}}) \int^t_0 \left\| (\pl_t B, \pl_t \triangledown u, \pl_t \triangledown B) \right\|_m^2 \, \mathrm{d} \tau.
			\end{split}
		\end{equation}
		
		Integrating by parts as above and applying the boundary condition that $\pl_t \eta=0$ on the boundary, one obtains
		\begin{equation} \label{eqL11-14}
			\begin{split}
				&
				\int^t_0 \int_{\Omega} - \mu (\epsilon) \triangle \pl_t \eta \cdot \pl_t \eta \, \mathrm{d} x  \mathrm{d} \tau = \mu(\epsilon) \int^t_0 \left\| \triangledown \pl_t \eta \right\|^2 \, \mathrm{d} x  \mathrm{d} \tau.
			\end{split}
		\end{equation}
		
		The last term we need to consider is $F$, where we use Proposition \ref{prop5}
		\begin{equation} \label{eqL11-15}
			\begin{split}
				&
				\int_{0}^{t} \left\| \pl_t F_1 \right\|^2_{m} \, \mathrm{d} \tau \leq C_{m+3} (1 + \left\|(u, B,Zu, ZB,\triangledown u, \triangledown B)\right\|^2_{L^{\infty}})
				\int^t_0 \left\| (\pl_t B, \pl_t \triangledown u, \pl_t \triangledown B,\Pi \pl_t \triangledown q) \right\|_m^2 \, \mathrm{d} \tau,
			\end{split}
		\end{equation}
		\begin{equation} \label{eqL11-16}
			\begin{split}
				&
				\int_{0}^{t} \left\| \pl_t F_2 \right\|^2_{m} \, \mathrm{d} \tau \leq C_{m+3}(\mu(\epsilon))^2 \int^t_0 \left\| \triangledown^2 \pl_t u \right\|^2_m \, \mathrm{d} \tau
				\\
				&\quad
				+ C_{m+3} (1 + \left\|(u, B,Zu, ZB,\triangledown u, \triangledown B)\right\|^2_{L^{\infty}})\int^t_0 \left\| (\pl_t B, \pl_t \triangledown u, \pl_t \triangledown B) \right\|_m^2 \, \mathrm{d} \tau,
			\end{split}
		\end{equation}
		\begin{equation} \label{eqL11-17}
			\begin{split}
				&
				\int_{0}^{t} \left\| \pl_t F_3 \right\|^2_{m} \, \mathrm{d} \tau \leq C_{m+3}(\mu(\epsilon))^2 \int^t_0 \left\| \triangledown^2 \pl_t u \right\|^2_m \, \mathrm{d} \tau
				\\
				&\quad
				+ C_{m+3} (1 + \left\|(u, B,Zu, ZB,\triangledown u, \triangledown B)\right\|^2_{L^{\infty}})\int^t_0 \left\| (\pl_t B, \pl_t \triangledown u, \pl_t \triangledown B) \right\|_m^2 \, \mathrm{d} \tau.
			\end{split}
		\end{equation}
	
		Consider $|\alpha|=m$, we do integration by parts and have $Z^{\alpha} \pl_t \eta$ vanishing on the boundary for all $\alpha$
		\begin{equation} \label{eqL11-18}
			\begin{split}
				&
				\int^t_0 \int_{\Omega} Z^{\alpha}\pl_t F_4 \cdot Z^{\alpha}\pl_t \eta \, \mathrm{d} x \mathrm{d} \tau = \int^t_0 \int_{\Omega} \mu(\epsilon) Z^{\alpha} \chi \triangle (\Pi K) \pl_t u \cdot Z^{\alpha} \pl_t \eta \, \mathrm{d} x \mathrm{d} \tau
				\\
				&\quad
				\leq \delta (\mu(\epsilon))^2 \int^t_0 \left\| \triangledown \pl_t \eta \right\|^2_m \mathrm{d} \tau+ C_{\delta} C_{m+3} \int^t_0 (\left\| \triangledown \pl_t u \right\|^2_m + \left\|\pl_t u \right\|^2_m \,) \mathrm{d} \tau.
			\end{split}
		\end{equation}
		
		In this lemma, we consider $m=0$ and then end the proof.
		
	\end{proof}

		\hspace*{\fill}\\
		\hspace*{\fill}

		\begin{lem}\label{lem12}
			For very sufficiently smooth solutions defined on [0,T] of \eqref{eq1-1}, \eqref{eq1-2}, we have the following estimate 
			For a smooth solution to \eqref{eq1-1} and \eqref{eq1-2}, it holds that for integer $m \geq 1$ and all $\epsilon \in (0,1]$
			\begin{equation}\label{eqL12}
				\begin{split}
					&
					\frac{1}{2} \left\| \pl_t \eta (t) \right\|_{m-1}^2 + \frac{1}{2} \left\| \pl_t \chi (\curl B \times n) (t) \right\|_{m-1}^2 + \frac{\mu(\epsilon)}{2} \int^{t}_{0} \left\| \triangledown \pl_t \eta \right\|_{m-1}^2 \, \mathrm{d} \tau + \frac{\nu(\epsilon)}{2} \int^{t}_{0} \left\| \triangledown \pl_t \chi (\curl B \times n) \right\|_{m-1}^2 \, \mathrm{d} \tau
					\\
					&\quad
					\leq \frac{1}{2} \left\| \pl_t \eta (0) \right\|_{m-1}^2 + \frac{1}{2} \left\| \pl_t \chi (\curl B \times n) (0) \right\|_{m-1}^2 + \delta  \int_0^t ( \mu(\epsilon)^2 \left\| \triangledown^2 \pl_t u \right\|^2_{m-1} + \nu(\epsilon)^2 \left\| \triangledown^2 \pl_t B \right\|^2_{m-1} \, ) \mathrm{d} \tau
					\\
					&\quad\quad
					+ C_{\delta} C_{m+2} (1+ \left\| (u, B, Zu, ZB, \pl_t u, \pl_t B, \triangledown u, \triangledown B, Z\pl_t u, Z\pl_t B, Z \triangledown u, Z \triangledown B) \right\|^2_{L^{\infty}})
					\\
					&\quad\quad
					\int^{t}_{0} ( \left\| (u, B, \triangledown u, \triangledown B) \right\|_m^2 + \left\| (\pl_t u, \pl_t B) \right\|_{m-1}^2 + \left\| (\triangledown \pl_t u, \triangledown \pl_t B) \right\|_{m-1}^2 + \left\| \Pi \triangledown \pl_t q \right\|_{m-1}^2)\mathrm{d} \tau.
				\end{split}
			\end{equation}
		\end{lem}

		\begin{proof}
		Since we already have Lemma \ref{lem11} that is \eqref{eqL12} for $m=1$, we assume that we have the estimates for $k \leq m-2$. Therefore, we just need to prove \eqref{eqL12} holds for $k=m-1\geq 1$. We set $|\alpha|=m-1$, multiply $Z^{\alpha}\pl_t$\eqref{new} with $Z^{\alpha} \pl_t \eta$ and integrate it with respect to $x$ and $t$
		\begin{equation} \label{eqL12-1}
			\begin{split}
				&
				\int^t_0 \int_{\Omega} Z^{\alpha}\pl_t \pl_t \eta \cdot Z^{\alpha} \pl_t \eta + Z^{\alpha} \pl_t ((u \cdot \triangledown) \eta) \cdot Z^{\alpha}\pl_t \eta - Z^{\alpha} \pl_t (B \cdot \triangledown (\chi (\curl B \times n))) \cdot Z^{\alpha} \pl_t \eta 
				\\
				&\quad
				- \mu (\epsilon) Z^{\alpha} \triangle \pl_t \eta \cdot Z^{\alpha} \pl_t \eta \, \mathrm{d} x  \mathrm{d} \tau = \int^t_0 \int_{\Omega} Z^{\alpha} \pl_t F \cdot Z^{\alpha} \pl_t \eta \, \mathrm{d} x  \mathrm{d} \tau.
			\end{split}
		\end{equation}
		
		Frist, we have
		\begin{equation} \label{eqL12-2}
			\begin{split}
				&
				\int^t_0 \int_{\Omega} Z^{\alpha}\pl_t \pl_t \eta \cdot Z^{\alpha} \pl_t \eta \, \mathrm{d} x \mathrm{d} \tau= \frac{1}{2} \left\| \pl_t \eta (t) \right\|^2_{m-1} - \frac{1}{2} \left\| \pl_t \eta (0) \right\|^2_{m-1}.
			\end{split}
		\end{equation}
	
	    Recall we can write for $z \neq 0$ by having $\pl_z=\pl_n$ in this local coordinate
	    \begin{equation} \label{eqL12-3-2}
	    	\begin{split}
	    		&
	    		Z^{\alpha}(\pl_t \Tilde{u} \cdot \triangledown \Tilde{\eta})=  Z^{\alpha}(\pl_t \Tilde{u}^1 \pl_{y^1} \Tilde{\eta} + \pl_t \Tilde{u}^2 \pl_{y^2} \Tilde{\eta} + \pl_t \Tilde{u} \cdot n \pl_z \Tilde{\eta} )
	    		\\
	    		&\quad
	    		= Z^{\alpha}(\pl_t \Tilde{u}^1 \pl_{y^1} \Tilde{\eta} + \pl_t \Tilde{u}^2 \pl_{y^2} \Tilde{\eta} )+ \sum_{\beta+ \gamma= \alpha} \mathcal{C}_{\beta,\gamma} \frac{Z^{\beta}\pl_t\Tilde{u} \cdot n}{\varphi(z)} \varphi(z) Z^{\gamma} \pl_z \Tilde{\eta} ).
	    	\end{split}
	    \end{equation}
    
        Since $Z^{\alpha}\pl_t(u \cdot n)=0$ for all $\alpha$ on the boundary, we can write as (one can refer to \cite{M2012} for the details of this expression)
        \begin{equation} \label{eqL12-3-6}
        	\begin{split}
        		&
        		\sum_{\beta+ \gamma= \alpha} \mathcal{C}_{\beta,\gamma} \frac{Z^{\beta}\pl_t\Tilde{u} \cdot n}{\varphi(z)} \varphi(z) Z^{\gamma} \pl_z \Tilde{\eta} ) = \sum_{|\tilde{\beta}|+ |\tilde{\gamma}| \leq m-1, |\tilde{\gamma}| \neq m-1} C_{\tilde{\beta}, \tilde{\gamma}} Z^{\tilde{\beta}} ( \frac{\pl_t \Tilde{ u} \cdot n}{\varphi(z)} ) Z^{\tilde{\gamma}} (\varphi(z)\pl_z \Tilde{\eta}) 
        	\end{split}
        \end{equation}
	    then Similar to \eqref{eqL11-6} and applying Proposition \ref{prop1}, we have
	    \begin{equation} \label{eqL12-3-3}
	    	\begin{split}
	    		&
	    		\left\| Z^{\alpha}(\pl_t \Tilde{u} \cdot \triangledown \Tilde{\eta}) \right\|^2 \geq -C_{m+2}\left\|(Z \Tilde{u}, \pl_t \Tilde{u}, Z\triangledown \Tilde{u},\pl_z (\pl_t \Tilde{u} \cdot n)) \right\|^2_{L^{\infty}} \left\|(Z \Tilde{u}, \pl_t \Tilde{u}, Z \triangledown \Tilde{u}, \pl_z (\pl_t \Tilde{u} \cdot n) )\right\|^2_{m-1}.
	    	\end{split}
	    \end{equation}
        
        Recall that on this local basis we have $\pl_z=\pl_n$ and $\pl_t divu=0$
        \begin{equation} \label{eqL12-3-4}
        	\begin{split}
        		&
        		\pl_z (\pl_t \Tilde{u} \cdot n )=\pl_n (\pl_t \Tilde{u} \cdot n )=\pl_t (\pl_n \Tilde{ u} \cdot n)
        		\\
        		&\quad
        		= -\frac{1}{|g|^{\frac{1}{2}}}(\pl_{y^1}(|g|^{\frac{1}{2}}\pl_t \Tilde{ u}^1) + \pl_{y^2} (|g|^{\frac{1}{2}} \pl_t \Tilde{ u}^2) + (\pl_n |g|^{\frac{1}{2}} )\pl_t \Tilde{ u} \cdot n ),
        	\end{split}
        \end{equation}
		therefore
		\begin{equation} \label{eqL12-3-5}
			\begin{split}
				&
				\left\| Z^{\alpha}(\pl_t u \cdot \triangledown \eta) \right\|^2 \geq -C_{m+2} \left\|(Zu, \pl_t u, Z\pl_t u, Z \triangledown u)\right\|^2_{L^{\infty}} \left\|(Zu, \pl_t u, Z \pl_t u , Z \triangledown u) \right\|^2_{m-1}.
			\end{split}
		\end{equation}
	    
	    Similarly, we have by applying Proposition \ref{prop5}
	    \begin{equation} \label{eqL12-3-7}
	    	\begin{split}
	    		&
	    		\left\| \sum_{| \beta| = m- 2} C_{\beta} Z^{\beta}\pl_t (Z u \cdot \triangledown \eta) \right\|^2 \geq -C_{m+2} \left\|(Zu, Z \triangledown u)\right\|^2_{L^{\infty}} \left\|(Zu, Z \pl_t u , Z \triangledown u, Z \triangledown \pl_t u) \right\|^2_{m-2}.
	    	\end{split}
	    \end{equation}
		
		We have the boundary condition that $Z^{\alpha}\pl_t \eta=0$ on the boundary for $\alpha$ and $div u=0$. We do integration by parts and apply Proposition \ref{prop5} to get
		\begin{equation} \label{eqL12-3}
			\begin{split}
				&
				\int^t_0 \int_{\Omega} Z^{\alpha} \pl_t ((u \cdot \triangledown) \eta) \cdot Z^{\alpha}\pl_t \eta \, \mathrm{d} x  \mathrm{d} \tau
				\\
				&\quad
				= \int^t_0 \int_{\Omega} u \cdot [Z^{\alpha}, \triangledown] \pl_t \eta \cdot Z^{\alpha} \pl_t \eta + \frac{1}{2} u \cdot \triangledown |Z^{\alpha} \pl_t \eta|^2
				\\
				&\quad\quad
				+Z^{\alpha}(\pl_t u \cdot \triangledown \eta) \cdot Z^{\alpha} \pl_t \eta + \sum_{| \beta| = m- 2} \mathcal{C}_{\beta} Z^{\beta}\pl_t (Z u \cdot \triangledown \eta) \cdot Z^{\alpha} \pl_t \eta \, \mathrm{d} x  \mathrm{d} \tau
				\\
				&\quad
				\geq -C_{m+2}(1+\left\| (u, Zu, \pl_t u, \triangledown u ,Z \triangledown u, Z \pl_t u) \right\|^2_{L^{\infty}}) \int^t_0 (\left\| \triangledown \pl_t u \right\|^2_{m-1} + \left\| \pl_t u \right\|^2_m + \left\| (u, \triangledown u ) \right\|^2_m \,) \mathrm{d} \tau.
			\end{split}
		\end{equation}

		We also have the boundary condition $B \cdot n=0$ on the boundary. Applying Proposition \ref{prop1} and Proposition \ref{prop5}
		\begin{equation} \label{eqL12-4}
			\begin{split}
				&
				\int^t_0 \int_{\Omega} - Z^{\alpha} \pl_t (B \cdot \triangledown (\chi (\curl B \times n))) \cdot Z^{\alpha} \pl_t \eta \, \mathrm{d} x  \mathrm{d} \tau
				\\
				&\quad
				= \int^t_0 \int_{\Omega} - B \cdot [Z^{\alpha}, \triangledown] \pl_t \chi (\curl B \times n) \cdot Z^{\alpha} \pl_t \eta - B \cdot \triangledown Z^{\alpha} \pl_t \chi (\curl B \times n) \cdot Z^{\alpha} \pl_t \eta
				\\
				&\quad\quad
				-Z^{\alpha}(\pl_t B \cdot \triangledown \chi (\curl B \times n)) \cdot Z^{\alpha} \pl_t \eta + \sum_{| \beta| =m- 1} \mathcal{C}_{\beta} Z^{\beta}\pl_t (Z B \cdot \triangledown\chi (\curl B \times n)) \cdot Z^{\alpha} \pl_t \eta \, \mathrm{d} x  \mathrm{d} \tau
				\\
				&\quad
				\geq  \int^t_0 \int_{\Omega} - B \cdot \triangledown Z^{\alpha} \pl_t \chi (\curl B \times n) \cdot Z^{\alpha} \pl_t \eta \, \mathrm{d} x  \mathrm{d} \tau
				\\
				&\quad
				-C_{m+2}(1+\left\| (B, ZB, \pl_t B, \triangledown B, Z \triangledown B, Z \pl_t B) \right\|^2_{L^{\infty}}) \int^t_0 ( \left\| (\triangledown \pl_t u, \triangledown \pl_t B)\right\|^2_{m-1} + \left\| (\pl_t u,\pl_t B) \right\|^2_m + \left\|  \triangledown B \right\|^2_m \,) \mathrm{d} \tau,
			\end{split}
		\end{equation}
		where we use the same argument as \eqref{eqL12-3-5} and \eqref{eqL12-3-7}.

		Recall
		\begin{equation}\label{newB}
			\pl_t ( \chi (\curl B \times n)) - (B \cdot \triangledown) \eta + (u \cdot \triangledown)( \chi (\curl B \times n))- \nu (\epsilon) \triangle ( \chi (\curl B \times n)) = G.
		\end{equation}

		Then, we do integration by parts and apply Proposition \ref{prop1}, and Proposition \ref{prop5} as above
		\begin{equation} \label{eqL12-5}
			\begin{split}
				&
				\int^t_0 \int_{\Omega} - B \cdot \triangledown Z^{\alpha} \pl_t \chi (\curl B \times n) \cdot Z^{\alpha} \pl_t \eta \, \mathrm{d} x  \mathrm{d} \tau
				\\
				&\quad
				= \int^t_0 \int_{\Omega} Z^{\alpha}\pl_t \pl_t \chi (\curl B \times n) \cdot Z^{\alpha} \pl_t \chi (\curl B \times n) - B \cdot \triangledown (Z^{\alpha} \pl_t \chi (\curl B \times n) \cdot Z^{\alpha} \pl_t \eta)
				\\
				&\quad\quad
				- B \cdot [Z^{\alpha} , \triangledown] \pl_t \eta \cdot Z^{\alpha} \pl_t \chi (\curl B \times n) - Z^{\alpha}(\pl_t B \cdot \triangledown \eta) \cdot Z^{\alpha} \pl_t \chi (\curl B \times n) + \frac{1}{2} u \cdot \triangledown| Z^{\alpha} \pl_t \chi (\curl B \times n)|^2
				\\
				&\quad\quad
				+ \sum_{| \beta| = m-1}\mathcal{C}_{\beta} Z^{\beta} \pl_t (Z B \cdot \triangledown \eta) \cdot Z^{\alpha} \pl_t \chi (\curl B \times n)
				+ u \cdot [Z^{\alpha}, \triangledown] \pl_t \chi (\curl B \times n) \cdot Z^{\alpha} \pl_t \chi (\curl B \times n)
				\\
				&\quad\quad
				+ Z^{\alpha}(\pl_t u \cdot \triangledown \chi (\curl B \times n)) \cdot Z^{\alpha} \pl_t \chi (\curl B \times n)
				+ \sum_{| \beta| = m-1}\mathcal{C}_{\beta} Z^{\beta} \pl_t (Z u \cdot \triangledown \chi (\curl B \times n)) \cdot Z^{\alpha} \pl_t \chi (\curl B \times n)
				\\
				&\quad\quad
				- \nu(\epsilon) Z^{\alpha}\triangle \pl_t \chi (\curl B \times n) \cdot Z^{\alpha} \pl_t \chi (\curl B \times n) - Z^{\alpha} \pl_t G \cdot Z^{\alpha} \pl_t \chi (\curl B \times n)\, \mathrm{d} x  \mathrm{d} \tau
				\\
				&\quad
				\geq \frac{1}{2} \left\| \pl_t \chi (\curl B \times n)(t) \right\|^2_{m-1} - \frac{1}{2} \left\| \pl_t \chi (\curl B \times n)(0) \right\|^2_{m-1}
				\\
				&\quad\quad
				- \int^t_0 \int_{\Omega} \nu(\epsilon) Z^{\alpha}\triangle \pl_t \chi (\curl B \times n) \cdot Z^{\alpha} \pl_t \chi (\curl B \times n) + Z^{\alpha} \pl_t G \cdot Z^{\alpha} \pl_t \chi (\curl B \times n)\, \mathrm{d} x  \mathrm{d} \tau
				\\
				&\quad\quad
				- C_{m+2} (1+ \left\|(u,B, Zu, ZB, \triangledown u,\triangledown B, \pl_t u, \pl_t B, Z \triangledown u, Z\triangledown B, Z\pl_t u, Z\pl_t B) \right\|^2_{L^{\infty}})
				\\
				&\quad\quad
				\int^t_0 ( \left\| (\pl_t \eta , \pl_t \chi (\curl B \times n))\right\|^2_{m-1} + \left\| (\pl_t u,\pl_t B) \right\|^2_m + \left\| (u,B, \triangledown u, \triangledown B) \right\|^2_m \, ) \mathrm{d} \tau.
			\end{split}
		\end{equation}
	
	    In the local basis, we can write
	    $$\pl_i=\mathcal{\beta}^1_i \pl_{y^1} + \mathcal{\beta}^2_i \pl_{y^2} + \mathcal{\beta}^3_i \pl_z, \qquad i=1,2,3.$$
	    
	    Therefore, we get
	    \begin{equation} \label{eqL12-6-2}
	    	\begin{split}
	    		&
	    		Z^{\alpha}\triangle \pl_t \chi (\curl B \times n) =[Z^{\alpha}, \triangledown \cdot] \triangledown \pl_t \chi (\curl B \times n) + \triangledown \cdot Z^{\alpha} \triangledown \pl_t \chi (\curl B \times n)
	    		\\
	    		&\quad
	    		= \triangledown \cdot Z^{\alpha} \triangledown \pl_t \chi (\curl B \times n) + \sum_{| \beta| \leq m-2} ( \mathcal{C}_{1\beta} \pl_{zz} Z^{\beta} \pl_t \chi (\curl B \times n)  
	    		\\
	    		&\quad\quad
	    		+ \mathcal{C}_{2 \beta} \pl_z Z_y Z^{\beta} \pl_t \chi (\curl B \times n)+  \mathcal{C}_{3 \beta} Z_y Z_y Z^{\beta} \pl_t \chi (\curl B \times n) ).
	    	\end{split}
	    \end{equation}

		Integrating by parts and applying the boundary condition that $(Z^{\alpha}\pl_t \chi (\curl B \times n))|_{\pl \Omega}=0$ for all $\alpha$ on the boundary
		\begin{equation} \label{eqL12-6}
			\begin{split}
				&
				- \int^t_0 \int_{\Omega} \nu(\epsilon) Z^{\alpha}\triangle \pl_t \chi (\curl B \times n) \cdot Z^{\alpha} \pl_t \chi (\curl B \times n) \, \mathrm{d} x  \mathrm{d} \tau
				\\
				&\quad
				= \int^t_0 \int_{\Omega} -\nu(\epsilon) [Z^{\alpha}, \triangledown \cdot] \triangledown \pl_t \chi (\curl B \times n) \cdot Z^{\alpha} \pl_t \chi (\curl B \times n)
				\\
				&\quad\quad
				- \nu(\epsilon) \triangledown \cdot Z^{\alpha} \triangledown \pl_t \chi (\curl B \times n) \cdot Z^{\alpha} \pl_t \chi (\curl B \times n) \, \mathrm{d} x  \mathrm{d} \tau
				\\
				&\quad
				\geq \frac{3}{4} \nu(\epsilon) \int^t_0 \left\| \triangledown \pl_t \chi (\curl B \times n) \right\|^2_{m-1} \mathrm{d} \tau - C \nu(\epsilon) \int^t_0 \left\| \triangledown \pl_t\chi (\curl B \times n) \right\|^2_{m-2} \mathrm{d} \tau
				\\
				&\quad\quad
				- C_{m+2} \int^t_0 \left\| \pl_t \chi (\curl B \times n) \right\|^2_{m-1} \mathrm{d} \tau,
			\end{split}
		\end{equation}
	    where we use the vector equality $\triangledown \cdot (Z^{\alpha} \triangledown \pl_t \chi (\curl B \times n) \cdot Z^{\alpha} \pl_t \chi (\curl B \times n)) = \triangledown \cdot Z^{\alpha} \triangledown \pl_t \chi (\curl B \times n) \cdot Z^{\alpha} \pl_t \chi (\curl B \times n) + Z^{\alpha} \triangledown \pl_t \chi (\curl B \times n) : \triangledown Z^{\alpha} \pl_t \chi (\curl B \times n)$ with $Z^{\alpha} \triangledown \pl_t \chi (\curl B \times n) : \triangledown Z^{\alpha} \pl_t \chi (\curl B \times n)= \sum_{i=1,2,3} Z^{\alpha} \triangledown \pl_t \chi (\curl B \times n)_i \cdot \triangledown Z^{\alpha} \pl_t \chi (\curl B \times n)_i$.
		
		For the last term, we follow above by also having the boundary condition that $(Z^{\alpha}\pl_t \eta)|_{\pl \Omega}=0$ for all $\alpha$. Integrating by parts, one obtains that
		\begin{equation} \label{eqL12-7}
			\begin{split}
				&
				- \int^t_0 \int_{\Omega} \mu(\epsilon) Z^{\alpha}\triangle \pl_t \eta \cdot Z^{\alpha} \pl_t \eta \, \mathrm{d} x  \mathrm{d} \tau
				= \int^t_0 \int_{\Omega} -\mu(\epsilon) [Z^{\alpha}, \triangledown \cdot] \triangledown \pl_t \eta \cdot Z^{\alpha} \pl_t \eta - \mu(\epsilon) \triangledown \cdot Z^{\alpha} \triangledown \pl_t \eta \cdot Z^{\alpha} \pl_t \eta \, \mathrm{d} x  \mathrm{d} \tau
				\\
				&\quad
				\geq \frac{3}{4} \mu(\epsilon) \int^t_0 \left\| \triangledown \pl_t \eta \right\|^2_{m-1} \mathrm{d} \tau - C \mu(\epsilon) \int^t_0 \left\| \triangledown \pl_t \eta \right\|^2_{m-2} \mathrm{d} \tau - C_{m+2} \int^t_0 \left\| \pl_t \eta \right\|^2_{m-1} \mathrm{d} \tau.
			\end{split}
		\end{equation}
		
		Combining the above with the terms about $\left\| (\pl_t F, \pl_t G )\right\|^2_{m-1}$ in the proof of Lemma \ref{lem11}, we get
		\begin{equation} \label{eqL12-8}
			\begin{split}
				&
				\frac{1}{2} \left\| \pl_t \eta (t) \right\|_{m-1}^2 + \frac{1}{2} \left\| \pl_t \chi (\curl B \times n) (t) \right\|_{m-1}^2 + \frac{\mu(\epsilon)}{2} \int^{t}_{0} \left\| \triangledown \pl_t \eta \right\|_{m-1}^2 \, \mathrm{d} \tau + \frac{\nu(\epsilon)}{2} \int^{t}_{0} \left\| \triangledown \pl_t \chi (\curl B \times n) \right\|_{m-1}^2 \, \mathrm{d} \tau
				\\
				&\quad
				\leq \frac{1}{2} \left\| \pl_t \eta (0) \right\|_{m-1}^2 + \frac{1}{2} \left\| \pl_t \chi (\curl B \times n) (0) \right\|_{m-1}^2 + \delta  \int_0^t ( \mu(\epsilon)^2 \left\| \triangledown^2 \pl_t u \right\|^2_{m-1} + \nu(\epsilon)^2 \left\| \triangledown^2 \pl_t B \right\|^2_{m-1} \, ) \mathrm{d} \tau
				\\
				&\quad\quad
				+ C_{\delta} C_{m+2} (1+ \left\| (u, B, Zu, ZB, \pl_t u, \pl_t B, \triangledown u, \triangledown B, Z\pl_t u, Z\pl_t B, Z \triangledown u, Z \triangledown B) \right\|^2_{L^{\infty}})
				\\
				&\quad\quad
				\int^{t}_{0} ( \left\| (u, B, \triangledown u, \triangledown B) \right\|_m^2 + \left\| (\pl_t u, \pl_t B) \right\|_{m-1}^2 + \left\| (\triangledown \pl_t u, \triangledown \pl_t B) \right\|_{m-1}^2 + \left\| \Pi \triangledown \pl_t q \right\|_{m-1}^2 \, )\mathrm{d} \tau
				\\
				&\quad\quad
				+C\nu(\epsilon) \int^{t}_{0} \left\|\triangledown \pl_t \chi (\curl B \times n) \right\|^2_{m-2} \, \mathrm{d} \tau + C \mu(\epsilon) \int^{t}_{0} \left\| \triangledown \pl_t \eta \right\|^2_{m-2} \,  \mathrm{d} \tau,
			\end{split}
		\end{equation}
		by induction, we end the proof.
		
	\end{proof}

		\hspace*{\fill}\\
		\hspace*{\fill}

		\begin{lem}\label{lem13}
			For every $m \geq 2$ and $\epsilon \in (0,1]$, a smooth solution of \eqref{eq1-1}, \eqref{eq1-2} on $[0,T]$ satisfies the estimates
			\begin{equation}\label{eqL13}
				\begin{split}
					&
					\int^t_0 \left\| \triangledown \pl_t q_1 \right\|^2_{m-1} + \left\| \triangledown^2 \pl_t q_1 \right\|^2_{m-1} \, \mathrm{d} \tau \leq C_{m+2} (1+ \left\| (u, B , \triangledown u, \triangledown B) \right\|^2_{L^{\infty}})
					\\
					&\qquad\qquad\qquad\qquad
					\int^t_0 ( \left\|(u, B, \pl_t u, \pl_t B)\right\|^2_m + \left\| (\triangledown u , \triangledown B, \triangledown \pl_t u, \triangledown \pl_t B)\right\|^2_{m-1} \, ) \mathrm{d} \tau,
					\\
					&
					\int^t_0 \left\| \triangledown \pl_t q_2 \right\|^2_{m-1} \, \mathrm{d} \tau \leq C_{m+2} \mu(\epsilon) \int^t_0 ( \left\| \pl_t u \right\|^2_m + \left\| \triangledown \pl_t u \right\|^2_{m-1} \, ) \mathrm{d} \tau,
				\end{split}
			\end{equation}
			for $t \in [0,T]$.
		\end{lem}

		\begin{proof}

		Recall we have
		\begin{equation}
			\triangle \pl_t q_1 =\pl_t( -\triangledown u \cdot \triangledown u) + \pl_t( \triangledown B \cdot \triangledown B) \quad , \quad x \in \Omega,
		\end{equation}
		with the boundary condition
		\begin{equation}
			\pl_n \pl_t q_1 = \pl_t (- u \cdot \triangledown u+ B \cdot \triangledown B) \cdot n  \quad , \quad x \in \pl \Omega,
		\end{equation}
		
		and
		\begin{equation}
			\triangle \pl_t q_2 = 0 \quad , \quad x \in \Omega,
		\end{equation}
		with the boundary condition
		\begin{equation}
			\pl_n \pl_t q_2 = \mu(\epsilon) \triangle \pl_t u \cdot n  \quad , \quad x \in \pl \Omega.
		\end{equation}
		
		From standard elliptic regularity results with Neumann boundary conditions \cite{M2012}, we get that
		\begin{equation}\label{eqL13-1}
			\begin{split}
				&
				\int^t_0 (\left\| \triangledown \pl_t q_1 \right\|^2_{m-1} + \left\| \triangledown^2 \pl_t q_1 \right\|^2_{m-1} \, ) \mathrm{d} \tau \leq C_{m+2} \int^t_0 ( \left\| \pl_t(\triangledown u \cdot \triangledown u) \right\|^2_{m-1} + \left\| \pl_t(\triangledown B \cdot \triangledown B) \right\|^2_{m-1}
				\\
				&\qquad \qquad\qquad
				+ \left\| \pl_t(u \cdot \triangledown u) \right\|^2 + \left\| \pl_t(B \cdot \triangledown B) \right\|^2 + | \pl_t (u \cdot \triangledown u) \cdot n |^2_{H^{m-\frac{1}{2}}(\pl \Omega)} + | \pl_t (B \cdot \triangledown B) \cdot n |^2_{H^{m-\frac{1}{2}}(\pl \Omega)} \,) \mathrm{d} \tau.
			\end{split}
		\end{equation}
		
		Since $u \cdot n=0$ on the boundary, we note that
		\begin{equation}\label{eqL13-2}
			\pl_t (u \cdot \triangledown u) \cdot n = \pl_t((u \cdot \triangledown u) \cdot n) = - \pl_t ((u \cdot \triangledown n) \cdot u) \, , \qquad x \in \pl \Omega,
		\end{equation}
		and hence by Proposition \ref{prop3} (trace theorem)
		\begin{equation}\label{eqL13-3}
			\int^t_0 | \pl_t (u \cdot \triangledown u) \cdot n |^2_{H^{m-\frac{1}{2}}(\pl \Omega)} \mathrm{d} \tau \leq C_{m+2} \int^t_0 ( \left\| \pl_t(\triangledown u \cdot \triangledown u) \right\|^2_{m-1} + \left\| \pl_t( u \cdot \triangledown u) \right\|^2 + \left\| \pl_t(u \otimes u) \right\|^2_{m} \,) \mathrm{d} \tau.
		\end{equation}
		
		Then by Proposition \ref{prop5}, we get
		\begin{equation}\label{eqL13-4}
			\int^t_0 | \pl_t (u \cdot \triangledown u) \cdot n |^2_{H^{m-\frac{1}{2}}(\pl \Omega)} \mathrm{d} \tau \leq C_{m+2} (1+ \left\| (u, \triangledown u) \right\|^2_{L^{\infty}}) \int^t_0 ( \left\| (u , \pl_t u) \right\|^2_{m} + \left\|(\triangledown u, \triangledown \pl_t u) \right\|^2_{m-1} \, ) \mathrm{d} \tau.
		\end{equation}
		
		Similarly, since we also have $B \cdot n=0$ on the boundary,
		\begin{equation}\label{eqL13-5}
			\int^t_0 | \pl_t (B \cdot \triangledown B) \cdot n |^2_{H^{m-\frac{1}{2}}(\pl \Omega)} \mathrm{d} \tau \leq C_{m+2} (1+ \left\| (B, \triangledown B) \right\|^2_{L^{\infty}}) \int^t_0 ( \left\| (B , \pl_t B) \right\|^2_{m} + \left\|(\triangledown B, \triangledown \pl_t B) \right\|^2_{m-1} \,) \mathrm{d} \tau.
		\end{equation}
		
		Therefore, we have
		\begin{equation}\label{eqL13-6}
			\begin{split}
				&
				\int^t_0 \left\| \triangledown \pl_t q_1 \right\|^2_{m-1} + \left\| \triangledown^2 \pl_t q_1 \right\|^2_{m-1} \, \mathrm{d} \tau \leq C_{m+2} (1+ \left\| (u,B,\triangledown u, \triangledown B) \right\|^2_{L^{\infty}}) 
				\int^t_0 (\left\| (u,B, \pl_t u, \pl_t B) \right\|^2_m + \left\|(\triangledown u ,\triangledown B, \triangledown \pl_t u, \triangledown \pl_t B) \right\|^2_{m-1} \,) \mathrm{d} \tau.
			\end{split}
		\end{equation}
		
		It remains to estimate $\pl_t q_2$. By using the regularity for the Neumann problem \cite{M2012}, we get that for $m \geq 2$,
		\begin{equation}\label{eqL13-7}
			\int^t_0 \left\| \triangledown \pl_t q_2 \right\|^2_{m-1} \mathrm{d} \tau \leq \mu(\epsilon) C_{m} \int^t_0 | \triangle \pl_t u \cdot n |^2_{H^{m-\frac{3}{2}}(\pl \Omega)}.
		\end{equation}
		
		To estimate the right-hand side, we shall use the Navier boundary condition $\eqref{eqL1-2}_1$. Since
		\begin{equation}\label{eqL13-8}
			2 \triangle \pl_t u \cdot n = \triangledown \cdot \pl_t (Su \, n) - \sum_{j} \pl_t ( Su \pl_j n)_j,
		\end{equation}
		we first get that
		\begin{equation}\label{eqL13-9}
			| \triangle \pl_t u \cdot n |^2_{H^{m-\frac{3}{2}}(\pl \Omega)} \lesssim | \triangledown \cdot \pl_t (Su \, n)|^2_{H^{m-\frac{3}{2}}(\pl \Omega)} + C_{m+1} | \triangledown \pl_t u|^2_{H^{m-\frac{3}{2}}(\pl \Omega)},
		\end{equation}
		and, hence, since we have $\eqref{n1}_1$ and $\eqref{n2}_1$
		\begin{equation}\label{eqL13-10}
			| \triangledown \pl_t u|^2_{H^{m-\frac{3}{2}}(\pl \Omega)} \lesssim | \Pi(\pl_t u \cdot \triangledown n)- \Pi(\alpha \pl_t u)|^2_{H^{m-\frac{3}{2}}(\pl \Omega)} + | \sum_{i=1,2} (\Pi\pl_{y^i} \pl_t u)^i |^2_{H^{m-\frac{3}{2}}(\pl \Omega)} + |\pl_t u|^2_{H^{m-\frac{1}{2}}(\pl \Omega)},
		\end{equation}
		finally, we get
		\begin{equation}\label{eqL13-11}
			| \triangle \pl_t u \cdot n |^2_{H^{m-\frac{3}{2}}(\pl \Omega)} \lesssim | \triangledown \cdot \pl_t (Su \, n)|^2_{H^{m-\frac{3}{2}}(\pl \Omega)} + C_{m+1} | \pl_t u|^2_{H^{m-\frac{1}{2}}(\pl \Omega)}.
		\end{equation}
		
		We can rewrite
		\begin{equation}\label{eqL13-12}
			\triangledown \cdot \pl_t (Su \, n) = \pl_n \pl_t (Su \, n) \cdot n + ( \Pi \pl_{y^1} \pl_t (Su \, n) )^1 + ( \Pi \pl_{y^2} \pl_t (Su \, n) )^2. 
		\end{equation}
		
		Therefore, we have
		\begin{equation}\label{eqL13-13}
			\begin{split}
				&
				| \triangledown \cdot \pl_t (Su \, n)|^2_{H^{m-\frac{3}{2}}(\pl \Omega)} \lesssim | \pl_n \pl_t (Su \, n) \cdot n|^2_{H^{m-\frac{3}{2}}(\pl \Omega)} 
				+ C_{m+1}( | \Pi \pl_t (Su \, n)|^2_{H^{m-\frac{1}{2}}(\pl \Omega)} + | \triangledown \pl_t u |^2_{H^{m-\frac{3}{2}}(\pl \Omega)})
				\\
				&\quad
				\lesssim | \pl_n \pl_t (Su \, n) \cdot n|^2_{H^{m-\frac{3}{2}}(\pl \Omega)}+ C_{m+1} | \pl_t u |^2_{H^{m-\frac{1}{2}}(\pl \Omega)},
			\end{split}
		\end{equation}
		where we consider the boundary condition $\eqref{eqL1-2}_1$ such that $\Pi(Su n)+ \Pi(\alpha u)=0$ on the boundary and \eqref{eqL13-10}.
		
		Finally, we need to estimate the first term on the right-hand side in the above inequality by \eqref{eqL13-10}
		\begin{equation}\label{eqL13-14}
			\begin{split}
				&
				| \pl_n \pl_t (Su \, n) \cdot n|^2_{H^{m-\frac{3}{2}}(\pl \Omega)} \lesssim  | \pl_n \pl_t (\pl_n u \cdot n)|^2_{H^{m-\frac{3}{2}}(\pl \Omega)} + C_{m+1} | \triangledown \pl_t u |^2_{H^{m-\frac{3}{2}}(\pl \Omega)}
				\\
				&\quad
				\lesssim | \pl_n \pl_t (\pl_n u \cdot n)|^2_{H^{m-\frac{3}{2}}(\pl \Omega)} + C_{m+1} | \pl_t u |^2_{H^{m-\frac{1}{2}}(\pl \Omega)}
				\\
				&\quad
				\lesssim | \Pi \pl_n \pl_t u |^2_{H^{m-\frac{1}{2}}(\pl \Omega)} + C_{m+1} | \pl_t u |^2_{H^{m-\frac{1}{2}}(\pl \Omega)},
			\end{split}
		\end{equation}
		the last inequality used $\pl_n u \cdot n= - (\Pi \pl_{y^1} u )^1 - (\Pi \pl_{y^2} u )^2 $ as $div u =0$.
		
		By the boundary condition $\eqref{eqL1-2}_1$ such that $\Pi \pl_n u = \Pi( u \cdot \triangledown n) - \Pi(\alpha u)$, one has
		\begin{equation}\label{eqL13-15}
			| \Pi \pl_n \pl_t u |^2_{H^{m-\frac{1}{2}}(\pl \Omega)} \lesssim C_{m+2} | \pl_t u |^2_{H^{m-\frac{1}{2}}(\pl \Omega)}. 
		\end{equation}
		
		Thus, we obtain
		\begin{equation}\label{eqL13-16}
			\int^t_0 \left\| \triangledown \pl_t q_2 \right\|^2_{m-1} \, \mathrm{d} \tau \leq C_{m+2} \mu(\epsilon) \int^t_0 (\left\| \pl_t u \right\|^2_m + \left\| \triangledown \pl_t u \right\|^2_{m-1} \,) \mathrm{d} \tau.
		\end{equation}
        
    \end{proof}

		\hspace*{\fill}\\
		\hspace*{\fill}
		
		Recall \eqref{tri2-u}, \eqref{tri2-B}, one obtains
		\begin{equation}\label{tri2-ut}
			\nu (\epsilon) \int^t_0 \left\| \triangledown^2 \pl_t u \right\|^2_{m-1} \mathrm{d} \tau \leq C_{m+2} \mu (\epsilon) \int^t_0 (\left\| \triangledown \pl_t u \right\|^2_{m} + \left\| \triangledown \pl_t \eta \right\|^2_{m-1} + \left\| \pl_t u \right\|^2_{m} \,) \mathrm{d} \tau,
		\end{equation}
		and
		\begin{equation}\label{tri2-Bt}
			\nu (\epsilon) \int^t_0 \left\| \triangledown^2 \pl_t B \right\|^2_{m-1} \mathrm{d} \tau \leq C_{m+2} \nu (\epsilon) \int^t_0 (\left\| \triangledown \pl_t B \right\|^2_{m} + \left\| \triangledown \pl_t \chi (\curl B \times n) \right\|^2_{m-1} + \left\| \pl_t B \right\|^2_{m} \,) \mathrm{d} \tau,
		\end{equation}
		all the terms on the right-hand side have estimates as we already proved above.
		
		For now, we already have for nonnegative integer $m \leq 4$
		\begin{equation}\label{allbeforeLinfinityfort}
			\begin{split}
				&
				\left\|(\pl_t u, \pl_t B)\right\|^2_m + \left\|(\triangledown \pl_t u, \triangledown \pl_t B) \right\|^2_{m-1} + \epsilon \int_{0}^{t} \left\| (\triangledown \pl_t u, \triangledown \pl_t B) \right\|^2_{m} + \left\|(\triangledown \pl_t \eta, \triangledown \pl_t (\chi (\curl B \times n))) \right\|^2_{m-1} \mathrm{d} \tau
				\\
				&\quad
				\leq C_{m+2}( P( M(0))
				\\
				&\quad\quad
				+ P(\left\|(u,B, Zu, ZB,\pl_t u, \pl_t B, \triangledown u, \triangledown B, Z\pl_t u, Z\pl_t B ,Z \triangledown u, Z \triangledown B) \right\|^2_{L^{\infty}} )\int_{0}^{t}(P( N_{m+1}(\tau) ) + P(M(\tau)) \,)\mathrm{d} \tau ),
			\end{split}
		\end{equation}
	    that is 
	    \begin{equation}\label{allbeforeLinfinityfort-1}
	    	\begin{split}
	    		&
	    		M(t) + \epsilon \int_{0}^{t} \left\| (\triangledown \pl_t u, \triangledown \pl_t B) \right\|^2_{4} + \left\|(\triangledown \pl_t \eta, \triangledown \pl_t (\chi (\curl B \times n))) \right\|^2_{3} \mathrm{d} \tau
	    		\\
	    		&\quad
	    		\leq C_{6}( P( M(0))
	    		\\
	    		&\quad\quad
	    		+ (P(N_{4}(t))+P(\left\|(\pl_t u, \pl_t B, Z\pl_t u, Z\pl_t B) \right\|^2_{L^{\infty}} ))\int_{0}^{t}(P( N_{5}(\tau) ) + P(M(\tau))\,)\mathrm{d} \tau ),
	    	\end{split}
	    \end{equation}
		thus we then need to estimate $L^{\infty}$ terms.\\

		\begin{lem}\label{lem14}
			For every $m_0 >1$, it holds that 
			\begin{equation}\label{eqL14}
				\left\| (\pl_t u, \pl_t B) \right\|^2_{1,\infty} \leq C_m( \left\| (\triangledown \pl_t u, \triangledown \pl_t B)\right\|^2_{m-1} + \left\| ( \pl_t u, \pl_t B )\right\|^2_{m}), \, m \geq m_0+2.
			\end{equation}
		\end{lem}

		\begin{proof}
			One just needs to apply Proposition \ref{prop2} directly.
		\end{proof}

		\hspace*{\fill}\\
		\hspace*{\fill}

		\begin{lem}\label{lem15}
			For very sufficiently smooth solutions defined on [0,T] of \eqref{eq1-1}, \eqref{eq1-2}, we have the following estimate for $m \geq 5$
			\begin{equation}\label{eqL15}
				M(t) + \epsilon \int_0^t (\left\| (\triangledown \pl_t u, \triangledown \pl_t B) \right\|^2_4 + \left\| (\triangledown^2 \pl_t u, \triangledown^2 \pl_t B) \right\|^2_{3} \,) \mathrm{d} \tau \leq  C_{m+2}(P(M_m(0)) + P(M_m(t)) \int_0^t P(M_m(\tau)) \, \mathrm{d} \tau).
			\end{equation}
		\end{lem}

		\begin{proof}
			We combine Lemma \ref{lem10}, Lemma \ref{lem12}, Lemma \ref{lem13}, and Lemma \ref{lem14} to get Lemma \ref{lem15}.
		\end{proof}

		\hspace*{\fill}\\
		\hspace*{\fill}

		Recall we already have for $m \geq 6$,
	    \begin{equation}\label{all1}
		\begin{split}
			&
			N_m(t) + \epsilon \int^{t}_{0} \left\| (\triangledown u, \triangledown B)\right\|_{m}^2 + \left\| (\triangledown^2 u, \triangledown^2 B)\right\|_{m-1}^2 \, \mathrm{d} \tau
			\\
			&\quad
			\leq C_{m+2}( P(M_m(0))+ P(M_m(t)) \cdot \int^t_0 ( C_{\delta} P(M_m(\tau)) + \delta \epsilon \left\| \triangledown^2 \pl_t B \right\|_{2}^2 \,)\mathrm{d} \tau).
		\end{split}
	    \end{equation}
		Combine \eqref{eqL15}, we finial get for $m \geq 6$,
		\begin{equation}\label{all2}
			\begin{split}
				&
				M_m(t) + \epsilon \int^{t}_{0} (\left\| (\triangledown u, \triangledown B)\right\|_{m}^2 + \left\| (\triangledown^2 u, \triangledown^2 B)\right\|_{m-1}^2 + \left\| (\triangledown \pl_t u, \triangledown \pl_t B) \right\|^2_4 + \left\| (\triangledown^2 \pl_t u, \triangledown^2 \pl_t B) \right\|^2_{3}\,) \mathrm{d} \tau
				\\
				&\quad
				\leq C_{m+2}( P(M_m(0))+ P(M_m(t)) \cdot \int^t_0 P(M_m(\tau)) \mathrm{d} \tau),
			\end{split}
		\end{equation}
		which completes the proof of {Theorem \ref{thm3}}.\\
	\begin{rmk}	
		For the initial value $M(0)$, recall {\eqref{eq1-1}, \eqref{q},} \eqref{new} and \eqref{newB}, we obtain
		\begin{equation}\label{initial}
			\begin{split}
				&
				M(0) \leq P(N_5(0)) + \epsilon^2 \left\| (\triangle u , \triangle B)(0) \right\|^2_4 + \epsilon^2 \left\| (\triangle \eta , \triangle (\chi (\curl B \times n)))(0) \right\|^2_3.
			\end{split}
		\end{equation}
		\end{rmk}
		
		\section{Proofs of Theorem 1.1 and Theorem 1.2} 

With the uniform a priori estimates in Theorem 3.1, the proofs of the local existence of solutions satisfying (1.23) in a time interval $[0, T] $ independent of $\epsilon$ and the uniform regularity 
(Theorem 1.1) and the convergence of the solutions (Theorem 1.2) (stated in Theorem 1.1 ) become standard, as for the incompressible Navier-Stokes equations (\cite{M2012}). Indeed,  one may first 
 mollify  the initial data  so that  a standard well-posedness
result can be applied. Then  the local existence of a solution can be obtained by  the a priori estimates given in Theorem 3.1 and a compactness
argument.  With the Lipschitz regularity,  the uniqueness of the solution is clear. The fact that the lifetime of the solution is independent of $\epsilon$  then follows by again using Theorem 3.1 and a continuous induction
argument. The proof of Theorem 1.2 is also a direct consequence of the uniform regularity estimate (1.23), following the compactness argument as in \cite{M2012}.  We omit the details here. 
		
\section*{Aknowledgement}  This research is supported by a grant from the Research Grants Council of the Hong Kong Special Administrative Region, China (Project No. 11307420 ).

	\noindent{Yingzhi Du}\\
	{ Department of Mathematics, City University of Hong Kong, 83 Tat Chee Avenue, Kowloon Tong, Hong Kong. \\
E-mail: yingzhidu2-c@cityu.edu.hk}\\
{Tao Luo}\\
{ Department of Mathematics, City University of Hong Kong, 83 Tat Chee Avenue, Kowloon Tong, Hong Kong. \\
E-mail: taoluo@cityu.edu.hk}


\begin{thebibliography}{9}
			
			\bibitem{B2008}Basson, A., G$\acute{e}$rard-Varet, D.: Wall laws for fluid flows at a boundary with random roughness. $\mathit{Commun.}$ $\mathit{Pure}$ $\mathit{Appl.}$ $\mathit{Math.}$ $\textbf{61}$(7) (2008), 941-987. 

            \bibitem{beirao} Beir\~ ao da Veiga, H., Crispo, F.: Sharp inviscid limit results under Navier type boundary conditions. An $ L^p$  theory. $\mathit{J.}$ $\mathit{Math.}$ $\mathit{Fluid}$ $\mathit{Mech.}$ $\textbf{12}$ (2010), no.3, 397-411.

			\bibitem{beirao1} Beir\~ ao da Veiga, H., Crispo, F.: Concerning the $W^{k,p}$-inviscid limit for 3D flows under a slip boundary condition. $\mathit{J.}$ $\mathit{Math.}$ $\mathit{Fluid}$ $\mathit{Mech.}$ $\textbf{13}$ (2011), 117-135.
			
			
	        \bibitem{CS2017}Cheng, C.H.A., Shkoller, S.: Solvability and Regularity for an Elliptic System Prescribing the Curl, Divergence, and Partial Trace of a Vector Field on Sobolev-Class Domains. $\mathit{J.}$ $\mathit{Math.}$ $\mathit{Fluid}$ $\mathit{Mech.}$, $\textbf{19}$ (2017), 375-422.
		
			
			\bibitem{CM1993}Chorin, A. J., Marsden, J. E.: A mathematical introduction to fluid mechanics (Vol. 3, pp. 269-286). (1990) New York: Springer.
						
			\bibitem{DXX2022}Duan, Q., Xiao, Y., Xin, Z.: On the vanishing dissipation limit for the incompressible MHD equations on bounded domains. $\mathit{Sci.}$ $\mathit{China}$ $\mathit{Math}$. $\textbf{65}$ (2022), 31-50.
			
		\bibitem{GLW} Gao, Song; Li, Shengxin; Wang, Jing: 	Vanishing dissipation limit of solutions to initial boundary value problem for three dimensional incompressible magneto-hydrodynamic equations with transverse magnetic field. J. Differential Equations 374 (2023), 29–55.

			
	         \bibitem{GMAS}Gérard-Varet, D., Masmoudi, N.: Relevance of the slip condition for fluid flows near an irregular boundary. $\mathit{Commun.}$ $\mathit{Math.}$ $\mathit{Phys.}$ $\textbf{295}$(1) (2010), 99-137.
			
			\bibitem{G2012}Gie, G.M., Kelliher, J.P.: Boundary layer analysis of the Navier-Stokes equations with generalized Navier boundary conditions. $\mathit{J.}$ $\mathit{Differ.}$ $\mathit{Equ.}$ $\textbf{253}$(6) (2012), 1862-1892.
			
			\bibitem{G1990}Gues, O.: Probleme mixte hyperbolique quasi-lineaire caracteristique. $\mathit{Comm.}$ $\mathit{Partial}$ $\mathit{Differential}$ $\mathit{Equations.}$ $\textbf{15}$ (1990), no.5, 595-645.
			
			\bibitem{G1991}Gunzburger, M.D., Meir, A.J. and Peterson, J.S.: On the Existence, Uniqueness, and Nite Element Approximation of Solutions of the Equations of Stationary, Incompressible Magnetohydrodynamics. $\mathit{Math.}$ $\mathit{Comp.}$, $\textbf{56}$ (1991), 523-563.
			

			
			\bibitem{I2011}Iftimie, D., Sueur, F.: Viscous boundary layers for the Navier-Stokes equations with the Navier slip conditions. $\mathit{Arch}$ $\mathit{Rational}$ $\mathit{Mech}$ $\mathit{Anal}$ $\textbf{199}$ (2011), 145-175.
			
 \bibitem{JLX} Ju, Qiangchang; Luo, Tao; Xu, Xin, 
Singular limits for the Navier-Stokes-Poisson equations of the viscous plasma with the strong density boundary layer, 
Sci. China Math. 66 (2023), no. 7, 1495–1528.
						
			
		
			
		
			
			\bibitem{LXY2019}Liu, C.-J., Xie, F., Yang, T.: MHD boundary layers in Sobolev spaces without monotonicity.
			I. Well-posedness theory, $\mathit{Commun.}$ $\mathit{Pure}$ $\mathit{Appl.}$ $\mathit{Math.}$, $\textbf{72}$ (2019), 63-121.
			
			\bibitem{LXY2019-2}Liu, C.-J., Xie, F., Yang, T.: Justification of Prandtl ansatz for MHD boundary layer, $\mathit{SIAM}$ $\mathit{J.}$ $\mathit{Math.}$ $\mathit{Anal.}$, $\textbf{51}$ (2019), 2748-2791.
			
			\bibitem{LW2020}Liu, C.-J., Wang, D., Xie, F., Yang, T.: Magnetic effects on the solvability of 2D MHD
			boundary layer equations without resistivity in Sobolev spaces, $\mathit{J.}$ $\mathit{Funct.}$ $\mathit{Anal.}$, $\textbf{279}$ (2020), no. 7, 108637.
			
			\bibitem{LXY2021}Liu, C.-J., Xie, F., Yang, T.: Uniform regularity and vanishing viscosity limit for the incompressible non-resistive MHD system with TMF. $\mathit{Communications}$ $\mathit{on}$ $\mathit{Pure}$ $\mathit{and}$ $\mathit{Applied}$ $\mathit{Analysis}$, $\textbf{20}$(7-8) (2021), 2725-2750.
			
		\bibitem{LYZ}	Liu, Cheng-Jie; Yang, Tong; Zhang, Zhu;  Validity of Prandtl expansions for steady MHD in the Sobolev framework,
SIAM J. Math. Anal. 55 (2023), no. 3, 2377–2410.
			
			\bibitem{M2014}Maekawa, Y.: On the Inviscid Limit Problem of the Vorticity Equations for Viscous Incompressible Flows in the Half-Plane. $\mathit{Commun.}$ $\mathit{Pur.}$ $\mathit{Appl.}$ $\mathit{Math.}$ $\textbf{67}$ (2014), 1045-1128.
			
			
			
		    \bibitem{M2012}Masmoudi, N., Rousset, F.: Uniform regularity for the Navier-Stokes equation with Navier boundary condition. $\mathit{Arch.}$ $\mathit{Ration.}$ $\mathit{Mech.}$ $\mathit{Anal.}$ $\textbf{203}$ (2012), 529-575.
			
			\bibitem{M2017}Masmoudi, N., Rousset, F.: Uniform Regularity and Vanishing Viscosity Limit for the Free Surface Navier-Stokes Equations. $\mathit{Archive}$ $\mathit{for}$ $\mathit{Rational}$ $\mathit{Mechanics}$ $\mathit{and}$ $\mathit{Analysis}$. $\textbf{223}$ (2017), 301-417.
			
	        \bibitem{M2003}Masmoudi, N., Saint-Raymond, L.: From the Boltzmann equation to the Stokes-Fourier system in a bounded domain. $\mathit{Commun.}$ $\mathit{Pur.}$ $\mathit{Appl.}$ $\mathit{Math.}$ $\textbf{56}$(9) (2003), 1263-1293.

			
			\bibitem{N1827}Navier, C.L.M.H.: Sur les lois d'$\acute{e}$quilibre et du mouvement des corps $\acute{e}$lastiques. $\mathit{M\acute{e}m.}$ $\mathit{Acad.}$ $\mathit{Sci.}$ $\textbf{7}$ (1827), 375-394.	
			
            \bibitem{paddick}Paddick, M.: The strong inviscid limit of the isentropic compressible Navier-Stokes equations with Navier boundary conditions. $\mathit{Contin.}$ $\mathit{Dyn.}$ $\mathit{Syst.}$ $\textbf{36}$ (2016), 2673-2709.
			
				
			\bibitem{SC1998}Sammartino, M., Caflisch, R. E.: Zero viscosity limit for analytic solutions of the Navier-Stokes equation on a half-space. I. Existence for Euler and Prandtl equations. $\mathit{Commun.}$ $\mathit{Math.}$ $\mathit{Phys.}$ $\textbf{192}$ (1998), no. 2, 433-461.
			
			
			\bibitem{SC1998-2}Sammartino, M., Caflisch, R. E.: Zero viscosity limit for analytic solutions of the Navier-Stokes equation on a half-space. II. Construction of the Navier-Stokes solution. $\mathit{Commun.}$ $\mathit{Math.}$ $\mathit{Phys.}$ $\textbf{192}$ (1998), no. 2, 463-491.
			
		\bibitem{zhangzf}	C. Wang, Y. Wang, and Z. Zhang, ``Zero-viscosity limit of the Navier-Stokes equations in the analytic setting," Arch. Ration. Mech. Anal. 224(2), 555--595 (2017).
			
			
			\bibitem{WW2013}Wang, Y. G., Williams, M.: The inviscid limit and stability of characteristic boundary layers for the compressible Navier-Stokes equations with Navier-friction boundary conditions. $\mathit{Ann.}$ $\mathit{Inst.}$ $\mathit{Fourier(Grenoble)}$ $\textbf{62}$ (2013), no. 6, 2257-2314.
			
			
			\bibitem{wangyong}Wang, Y.: Uniform regularity and vanishing dissipation limit for the full compressible Navier-Stokes system in three dimensional bounded domain. $\mathit{Arch.}$ $\mathit{Ration.}$ $\mathit{Mech.}$ $\mathit{Anal.}$ $\textbf{221}$ (2015), 4123-4191.
			
						
			
			
			\bibitem{WX2015}Wang, Y., Xin, Z., Yong, Y.: Uniform regularity and vanishing viscosity limit for the compressible Navier-Stokes with general Navier-slip boundary conditions in 3-dimensional domains. $\mathit{SIAM}$ $\mathit{Journal}$ $\mathit{on}$ $\mathit{Mathematical}$ $\mathit{Analysis.}$ $\textbf{47}$ (2015)
			
			
			
			\bibitem{XX2007}Xiao, Y.L., Xin, Z.P.: On the vanishing viscosity limit for the 3D Navier-Stokes equations with a slip boundary condition. $\mathit{Commun.}$ $\mathit{Pure}$ $\mathit{Appl.}$ $\mathit{Math.}$ $\textbf{LX}$ (2007), 1027-1055.
			
			\bibitem{XXW2009}Xiao, Y. L., Xin, Z. P., Wu, J. H.: Vanishing viscosity limit for the 3D magnetohydrodynamic system with a slip boundary condition, $\mathit{J.}$ $\mathit{Funct.}$ $\mathit{Anal.}$, $\textbf{257}$ (2009), 3375-3394.
			
			\bibitem{XX2013}Xiao, Y., Xin, Z.: On the Inviscid Limit of the 3D Navier-Stokes Equations with Generalized Navier-Slip Boundary Conditions. $\mathit{Commun.}$ $\mathit{Math.}$ $\mathit{Stat.}$ $\textbf{1}$ (2013), 259-279.
			
		


			
			
		\end{thebibliography}
\end{document}